\documentclass[12pt,a4paper]{article}
 
\usepackage[latin1]{inputenc}
\usepackage[english]{babel}
\usepackage[T1]{fontenc}
\usepackage{amsmath} 
\usepackage{amsfonts}
\usepackage{amssymb}
\usepackage[normalem]{ulem}

\usepackage{amsthm}
\usepackage{ dsfont }
\usepackage{ stmaryrd }
\usepackage{tikz}
\usetikzlibrary{matrix,arrows.meta}
\usepackage[titletoc]{appendix}

\allowdisplaybreaks

\usepackage{hyperref}

\usepackage[a4paper,top=2.1cm,bottom=2.60cm,left=2.1cm,right=2.1cm]{geometry} 

\theoremstyle{plain}
\newtheorem{theorem}{Theorem}[section]
\newtheorem{lemma}[theorem]{Lemma}

\newtheorem{proposition}[theorem]{Proposition} 
\newtheorem{corollary}[theorem]{Corollary}

\theoremstyle{definition} 
\newtheorem{remark}[theorem]{Remark}

\makeatletter
\newcommand{\eqnum}{\refstepcounter{equation}\textup{\tagform@{\theequation}}}
\makeatother

\begin{document}

\author{Matteo Dalla Riva  \thanks{Dipartimento di Ingegneria, Universit\`a degli Studi di Palermo, Viale delle Scienze, Ed. 8, 90128 Palermo, Italy {\tt matteo.dallariva@unipa.it}}, Pier Domenico Lamberti \thanks{Dipartimento di Tecnica e Gestione dei Sistemi Industriali, Universit\`a degli Studi di Padova, Stradella S. Nicola 3, 36100 Vicenza,  Italy
{\tt pierdomenico.lamberti@unipd.it}},  Paolo Luzzini \thanks{Dipartimento di Scienze e Innovazione Tecnologica, Universit\`a del Piemonte Orientale, Viale Teresa Michel 11, 15121 Alessandria, Italy {\tt paolo.luzzini@uniupo.it}}, Paolo Musolino \thanks{Dipartimento di Matematica `Tullio Levi Civita', Universit\`a degli Studi di Padova, Via Trieste 63, 35121 Padova, Italy  {\tt paolo.musolino@unipd.it }}}

\title{Shape sensitivity analysis of Neumann-Poincar\'e eigenvalues}

\date{Jannuary 27, 2025}

\maketitle

\abstract{This paper concerns the eigenvalues of the Neumann-Poincar\'e operator, a boundary integral operator associated with the harmonic double-layer potential. Specifically, we examine how the eigenvalues depend on the support of integration and prove that the map associating the support's shape to the eigenvalues is real-analytic. We then compute its first derivative and present applications of the resulting formula. The proposed method allows for handling infinite-dimensional perturbation parameters for multiple eigenvalues and perturbations that are not necessarily in the normal direction.}  
 
\vspace{10pt}

\vspace{9pt}


\tableofcontents

\section{Introduction}\label{sec:intro}

In this paper, we investigate the behavior of the eigenvalues of the Neumann-Poincar\'e (NP) operator when the support of integration is perturbed. We prove an analyticity result for the dependence on a boundary parameterization, derive a formula for the differential, and explore several implications, particularly in relation to a Rellich-Pohozhaev-type formula and an extension of a result that suggests the stability of certain sums of eigenvalues. This, in turn, would imply that the $n$-dimensional sphere minimizes the second-largest NP-eigenvalue. Our results can be compared to those of Grieser \cite{Gr14} and Ando et al.~\cite{AnKaMiPu21}, and are derived through a modification of \cite{LaLa02, LaLa04} approach. In this introduction, we describe how our results differ from the existing literature and what improvements they provide.

To proceed, we begin by introducing details. We fix a natural number $n \ge 2$ that represents the space dimension and consider a domain $\Omega \subseteq \mathbb{R}^n$ that satisfies the following conditions:
\begin{equation}\label{Omega_def}
\begin{split}
&\text{$\Omega$ is a bounded open subset of $\mathbb{R}^{n}$ of class $ C^{1,\alpha} $ for some $ \alpha \in (0,1)$,}\\
&\text{it is connected, and its exterior $\mathbb{R}^n \setminus \overline{\Omega}$ is also connected.}
\end{split}
\end{equation}
Then, we denote by $E_n : \mathbb{R}^n \setminus \{0\} \to \mathbb{R}$ the standard fundamental solution of the Laplace operator $-\Delta := -\sum_{j=1}^n \partial^2_{x_j}$. This function, as we recall, maps $x \in \mathbb{R}^n \setminus \{0\}$ to
\[
E_{n}(x):=
\left\{
\begin{array}{lll}
-\frac{1}{2\pi}\log |x| \qquad &   \text{if}\ n=2\,,
\vspace{.1cm}\\
-\frac{1}{(2-n)s_{n}}|x|^{2-n}\qquad &   \text{if}\ n\geq3\,,
\end{array}
\right.
\]
where $s_n$ is the $(n-1)$-dimensional measure of the unit sphere in $\mathbb{R}^n$.

In literature, the term Neumann-Poincar\'e operator may refer to both $K_{\partial \Omega}$ and $K^*_{\partial \Omega}$, the operators from $L^2(\partial \Omega)$ to itself that map a function $\psi$ to 
\[
K_{\partial \Omega}[\psi](x) := \int_{\partial \Omega} \nu_{\Omega}(y) \cdot \nabla E_n(x - y) \psi(y) \, d\sigma_y \quad \text{for all } x \in \partial \Omega,
\]
and to
\[
K^*_{\partial \Omega}[\psi](x) := -\int_{\partial \Omega} \nu_{\Omega}(x) \cdot \nabla E_n(x - y) \psi(y) \, d\sigma_y \quad \text{for all } x \in \partial \Omega,
\]
respectively. Here $\nu_{\Omega}$ denotes the outer normal vector field to the boundary $\partial \Omega$ of $\Omega$, and $d\sigma$ is the area element of the $(n-1)$-dimensional manifold $\partial \Omega$ embedded in $\mathbb{R}^n $.

The significance of the operators \( K_{\partial \Omega} \) and \( K^*_{\partial \Omega} \) lies in their strict connection to the Dirichlet and Neumann problems for the Laplace operator. In the following section, we will explore this relationship in greater detail. For now, we simply note that, under the assumption \eqref{Omega_def}, \( K_{\partial \Omega} \) and \( K^*_{\partial \Omega} \) are compact on \( L^2(\partial \Omega) \) and are adjoints of one another. Consequently, they share the same spectrum, which, apart from 0, consists of a countable set of real eigenvalues \( \lambda_j \). If this set is infinite, the eigenvalues form a sequence converging to 0. 

Moreover, as observed by Kellogg \cite[Chap.~XI, Sec.~11]{Ke29}, the spectrum of \( K_{\partial \Omega} \) and \( K^*_{\partial \Omega} \) lies within the interval \([-1/2, 1/2]\), with \( 1/2 \) always being an eigenvalue, while \(-1/2\) is an eigenvalue only if the exterior of \( \Omega \) has at least two connected components (see also \cite[Sec.~6.12]{DaLaMu21}). The question of whether 0 is an eigenvalue is an intriguing one, but it falls outside the scope of this paper. Readers may find further insight in the work of Ebenfelt, Khavinson, and Shapiro \cite{EbEtAl01}. The variational interpretation of the NP-eigenvalue problem originates from Poincar\'e's work \cite{Po97}, and for a modern treatment of the subject, we refer to Khavinson, Putinar, and Shapiro \cite{KhPuSh07}.

A certain regularity of the boundary of $\Omega$ is necessary for the compactness of \( K_{\partial \Omega} \) and \( K^*_{\partial \Omega} \). The presence of corners, as noted by Carleman in his doctoral dissertation \cite{Ca16}, can result in a continuous spectrum. For a recent review on this topic, the reader may refer to Ando et al.~\cite{AnKaMiPu21}. Here, we mention the works of Perfekt and Putinar \cite{PePu17} on the essential NP spectrum in a wedge, and that of
Kang et al.~\cite{KaLiYu17} on lens domains. Assumption \eqref{Omega_def} ensures that this situation is avoided, and the structure of the spectrum is as described above.

Since \( K_{\partial \Omega} \) and \( K^*_{\partial \Omega} \) share the same spectrum, we may use either operator in our analysis of the eigenvalues. However, $K^*_{\partial \Omega}$ offers certain advantages. While neither \( K_{\partial \Omega} \) nor \( K^*_{\partial \Omega} \) is generally self-adjoint--except in the special case where $\Omega$ is a ball and
\[
(x - y) \cdot \nu_{\Omega}(y) = (y - x) \cdot  \nu_{\Omega}(x) \quad \text{for all } x, y \in \partial \Omega
\]
(see Khavinson, Putinar, and Shapiro \cite[Thm. 8.1]{KhPuSh07})--the operator $K^*_{\partial \Omega}$ can always be symmetrized via the {\it Calder\'on identity}, also known as {\it Plemelj's symmetrization principle}:
\begin{equation}\label{PSP}
K_{\partial \Omega} S_{\partial \Omega} = S_{\partial \Omega} K^*_{\partial \Omega}\,,
\end{equation}
where $S_{\partial \Omega}$ is the single-layer potential operator (see Plemelj \cite{Pl11}). For the readers' convenience, we include a brief proof of this in the next section, where we also recall the definition of the operator $S_{\partial \Omega}$. Thus, we will proceed with $K^*_{\partial \Omega}$ in our analysis.


Our primary goal is to investigate how the eigenvalues \( \lambda_j \) depend on perturbations of the domain \( \Omega \). Specifically, we introduce a family \( \mathcal{A}^{1,\alpha}_{\partial \Omega} \) of \( C^{1,\alpha} \) functions \( \phi \) that are diffeomorphisms between \( \partial \Omega \) and the image \( \phi(\partial \Omega) \). We denote by \( \Omega[\phi] \) the bounded open set enclosed by \( \phi(\partial \Omega) \), so that \( \partial \Omega[\phi] = \phi(\partial \Omega) \). We then extend the definitions of  \( K^*_{\partial \Omega} \) to \( \Omega[\phi] \), writing  \( K^*_{\phi(\partial \Omega)} \) accordingly (noting that \( \Omega[\phi] \) satisfies the conditions in \eqref{Omega_def}). Denoting by \( \lambda[\phi] \) an  eigenvalue of  \( K^*_{\phi(\partial \Omega)} \), we ask ourselves what is the regularity of the map
\[
\mathcal{A}^{1,\alpha}_{\partial\Omega}\ni\phi\mapsto\lambda[\phi]\in\mathbb{R}\,,
\]
and we aim to prove an analyticity result.

Our argument builds on the analyticity result established by Lanza and Rossi \cite{LaRo04} (see also  \cite{DaLuMu22}) for the map that takes a diffeomorphism \( \phi \in \mathcal{A}^{1,\alpha}_{\partial \Omega} \) to the $\phi$-pullback of \( K^*_{\phi(\partial \Omega)} \), viewed as an element of an appropriate operator space. This analyticity can then be leveraged to prove an analyticity result for the corresponding Riesz projector, whose image coincides with the eigenspace of an eigenvalue \( \lambda[\phi] \) or the space generated by the eigenfunctions of set of eigenvalues $\{\lambda_1[\phi],\ldots,\lambda_m[\phi]\}$. From this, we obtain an analyticity result for a basis of such space, which in turn, we use to analyze the eigenvalues.

We note that this approach differs from the method of \cite{LaLa02, LaLa04} with Lanza, which was designed for self-adjoint operators and involves a detailed analysis of bilinear forms rather than Riesz projectors. Our strategy appears to be better suited for handling symmetrizable operators.

We can compare our results with those of Grieser \cite{Gr14}. In Grieser's paper, the focus is on smooth domains and one-dimensional normal perturbations of the form
\[
\phi_t(x) := x + t a(x) \nu_\Omega(x),
\]
where \( a \) is a smooth scalar function on \( \partial \Omega \), and the goal is to analyze the dependence of the eigenvalues on the real parameter \( t \). When restricted to perturbations of the form \( \phi_t \), our results recover those of Grieser, but with a lower regularity assumption.

However, extending to a multi-dimensional perturbation parameter (which in our case is infinite-dimensional) comes at an intrinsic cost. Specifically, under perturbations, an eigenvalue of multiplicity greater than one may split into several eigenvalues, whose total multiplicity equals the original. While for one-dimensional perturbations it is possible to isolate analytic branches, this is generally not feasible for even two-dimensional perturbations, let alone for the infinite-dimensional perturbations considered here (see Rellich's book \cite[Chap.~I, \S.2]{Re69} for a classical example).

A way to overcome this problem, without limiting ourselves to one-dimensional perturbations, is to consider the symmetric functions of the (potentially multiple) eigenvalues--an idea introduced, to the best of our knowledge, in \cite{LaLa04}. These symmetric functions are real functions defined by
\[
\Lambda_h^m(\xi_{1}, \ldots, \xi_{m}) := \sum_{\substack{j_1, \ldots, j_h \in \{1, \ldots, m\} \\ j_1 < \cdots < j_h}} \xi_{j_1} \cdots \xi_{j_h},
\]
for all \( \xi_1, \ldots, \xi_m \in \mathbb{R} \), where the integer $h$ runs from $1$ to $m$. If \( A \) is an \( (m \times m) \) real matrix with eigenvalues \( \xi_1, \ldots, \xi_m \in \mathbb{R} \), then the \( \Lambda_h^m(\xi_{1}, \ldots, \xi_{m}) \)'s are, up to a sign, the coefficients of the characteristic polynomial of \( A \). Specifically,
\begin{equation}\label{secular}
\mathrm{det}(A - \tau I_m) = \sum_{h=0}^m (-1)^h \tau^{h} \Lambda_{m-h}^m(\xi_{1}, \ldots, \xi_{m}),
\end{equation}
where we set \( \Lambda_0^m(\xi_{1}, \ldots, \xi_{m}) := 1 \). Notably, while the entries of the matrix \( A \) depend on the specific basis chosen for \( \mathbb{R}^m \), the symmetric functions \( \Lambda_h^m(\xi_{1}, \ldots, \xi_{m}) \) depend only on the eigenvalues \( \xi_1, \ldots, \xi_m \), making them independent of the choice of basis.

In our Theorem \ref{main}, we consider a diffeomorphism \( \phi_0 \in \mathcal{A}^{1,\alpha}_{\partial\Omega} \) and an eigenvalue \( \lambda \) of \( K^*_{\phi_0(\partial \Omega)} \) with multiplicity \( m \). We prove that, for \( \phi \) in a sufficiently small neighborhood \( \mathcal{U} \) of \( \phi_0 \), the eigenvalue \( \lambda \) splits into \( m \) eigenvalues
\[
\lambda_1[\phi] \leq \cdots \leq \lambda_m[\phi],
\]
(repeated according to their multiplicity), and the maps taking \( \phi \in \mathcal{U} \) to the symmetric functions
\[
\Lambda_h^m(\lambda_1[\phi], \ldots, \lambda_m[\phi])
\]
are real analytic. We also prove that the vector space generated by the eigenfunctions corresponding to \( \lambda_1[\phi], \ldots, \lambda_m[\phi] \) has a basis that depends analytically on \( \phi \in \mathcal{U} \). While this basis is not unique, the symmetric functions \( \Lambda_h^m(\lambda_1[\phi], \ldots, \lambda_m[\phi]) \) are intrinsic to the eigenvalues, making them a natural choice in this context. As a corollary, Theorem \ref{main} implies the continuity of the eigenvalues $\lambda_1[\phi], \ldots, \lambda_m[\phi]$ as $\phi$ tends to $\phi_0$. Moreover, if $\lambda$ happens to be a simple eigenvalue (so that $m=1$), then Theorem \ref{main} implies that we have a real analytic map $\phi\mapsto\lambda[\phi]$, where $\lambda[\phi]$ is an eigenvalue of $K^*_{\phi(\partial \Omega)}$ and $\lambda[\phi_0]=\lambda$.

We recall that the analyticity of the map \( \phi \mapsto \Lambda_h^m(\lambda_1[\phi], \ldots, \lambda_m[\phi]) \) (or $\phi\mapsto\lambda[\phi]$ if $\lambda$ is simple) means that it can be expanded into a convergent power series in \( \phi - \phi_0 \). In particular, it is differentiable. Our next objective is to derive an explicit formula for the differential
\[
d_\phi \Lambda^m_h(\lambda_1[\phi_0], \ldots, \lambda_m[\phi_0]).
\]
As a first step, we show that this computation reduces to finding the \( \phi \)-differential of the pulled-back integral operator \( K^*_{\phi(\partial\Omega)}[\mu_j[\phi_0]] \circ \phi \), where \( \mu_j[\phi_0] \) is an element of the eigenspace corresponding to \( \lambda \). After a (not-so-brief) manipulation of this differential, we arrive at Theorem \ref{dLambda}, where we provide an expression for the differential of \( \Lambda^m_h(\lambda_1[\phi_0], \ldots, \lambda_m[\phi_0]) \).

We note that a similar formula was obtained by Grieser \cite{Gr14} in the case of smooth normal perturbations, using a completely unrelated approach, which he also uses to compute the second derivative (see also Grieser et al.~\cite{GrEtAl09}). It is important to emphasize that our formula for the first-order derivative depends only on the normal component of the perturbation, as one might expect from the Zol\'esio velocity method for domain perturbation problems. However, the relatively low regularity assumption we employ prevents us from directly extending Grieser's formula to non-normal perturbations. We discuss these aspects in Section \ref{sec:comparison}.

In the final part of this paper, we investigate some consequences of the formula computed for the differential of \( \Lambda^m_h(\lambda_1[\phi], \ldots, \lambda_m[\phi]) \). First, we consider a dilation of the domain \(\Omega\), specifically a perturbation of the form
\[
\phi_t(x) := x+tx.
\] 
(It should be noted that this is not a normal perturbation.) Since NP-eigenvalues are invariant under dilation, we set the derivative formula obtained in Theorem \ref{dLambda} equal to zero and deduce a Rellich-Poho\v{z}aev-type formula for the eigenvalues.

In the second application, we examine the so-called $1/2$-conjecture, which states that if the dimension \(n = 3\), then for any domain \(\Omega\) obtained by perturbing a ball \(\mathbb{B}_3\), and for any \(k \in \mathbb{N}\), there are \(2k+1\) NP-eigenvalues (counting multiplicities) whose sum is \(1/2\). If true, this conjecture would imply that the ball minimizes the second NP-eigenvalue. In fact, the NP-eigenvalues of the 3-dimensional ball take the values 
\[
\frac{1}{2(2k+1)} \quad \text{for } k \in \mathbb{N},
\]
each with multiplicity \(2k+1\). The second eigenvalue is \({1}/{6}\), with multiplicity 3, and any other combination of three numbers whose sum is \(1/2\) must include a number greater than or equal to \({1}/{6}\)  (this argument is taken from Miyanishi and Suzuki \cite[proof of Theorem 4.3]{MiSu17}).

In the works of Martensen \cite{Ma99} and Ritter \cite{Ri95}, the $1/2$-conjecture is proven for domains \(\Omega\) that are spheroids or triaxial ellipsoids (see also Dobner and Ritter \cite{DoRi97}). Furthermore, Ando et al.~\cite{AnKaMiUs19} demonstrated that the ball is a critical shape for the sum of the eigenvalues splitting from \(\frac{1}{2(2k+1)}\). Specifically, they showed that if \(\lambda_1(t), \ldots, \lambda_{2k+1}(t)\) are the eigenvalues splitting from \(\frac{1}{2(2k+1)}\) when a normal perturbation \(\phi_t(x) := x + ta(x)\nu_\Omega(x)\) is applied to the ball, then the derivative of the function \(t \mapsto \lambda_1(t) + \ldots + \lambda_{2k+1}(t)\) is zero at \(t = 0\). While this result is consistent with the $1/2$-conjecture, it does not constitute a proof.

We observe that the sum \(\lambda_1(t) + \ldots + \lambda_{2k+1}(t)\) is the symmetric function \(\Lambda_1^{2k+1}\), computed at the eigenvalues \(\lambda_1(t), \ldots, \lambda_{2k+1}(t)\). Using the derivative formula derived in Theorem \ref{dLambda}, we can replicate the result of Ando and collaborators, extending it to cases where the perturbation is not necessarily normal and to any dimension \(n \geq 3\). This suggests that a variant of the $1/2$-conjecture could potentially be extended to any dimension \(n \geq 3\). By the same reasoning as in the three-dimensional case, this would imply that the ball minimizes the second NP-eigenvalue for all dimensions $n\ge 3$ (see Remark \ref{criticality} for details).

We conclude this introduction by noting that, in addition to the aforementioned papers, several works on the spectral properties of the NP-operator have emerged in recent years. To highlight a few, we reference Miyanishi and Rozenblum \cite{MiRo19} for Weyl's asymptotics. For the spectral analysis of the NP-operator in elasticity, we mention the works of Deng et al.~\cite{DeLiLi19}, Miyanishi and Rozenblum \cite{MiRo21}, Fukushima et al.~\cite{FuYoKa24}, and Rozenblum \cite{Ro23}. Additionally, a shape derivative approach has been employed by Ando et al.~\cite{AnKaMiPu23} to establish generic properties of the NP-spectrum. 

The paper is organized as follows. In Section \ref{sec:potential}, we introduce some key concepts from potential theory and discuss the PDE counterpart of the NP-eigenvalue problem, namely the plasmonic boundary value problem. At the end of this section, we prove the Calder\'on identity and explain why the operator \(K^*_{\partial \Omega}\) is symmetrizable. In Section \ref{sec:RP}, we introduce the Riesz projector for a general symmetrizable operator on a Hilbert space and examine some of its properties. Section \ref{sec:shape} returns to our specific perturbation problem. We first demonstrate that the Riesz projector corresponding to (a pulled-back version of) \(K^*_{\phi(\partial \Omega)}\) depends analytically on the shape diffeomorphism \(\phi\). Then, we deduce Theorem~\ref{main}, showing that the symmetric functions of the eigenvalues depend analytically on \(\phi\). From here, our goal is to compute the first shape derivative of the symmetric functions, which is done in Section \ref{derivatives}. In Section \ref{sec:comparison}, we compare our findings with Grieser's results and with Zol\'esio's speed method. Finally, in Section \ref{sec:applications}, we explore some applications, including a Rellich-Poho\v{z}aev-type formula and the criticality of the sphere for certain sums of eigenvalues. An appendix is included at the end of the paper, where we provide detailed proofs of certain equalities involving hypersingular integral operators related to the gradient of the double layer. Although these may be familiar to experts in potential theory, due to the lack of accessible references, we present a thorough argument.

\section{Some notions of potential theory}\label{sec:potential}

In this section, we summarize some key concepts of potential theory. We refrain from presenting proofs here and instead direct the reader to \cite{Fo95} and \cite{DaLaMu21} for detailed explanations. 

We have previously introduced the boundary operators $K_{\partial\Omega}$ and $K^*_{\partial\Omega}$, which, as a reminder, are compact on $L^2(\partial\Omega)$ (because $\Omega$ is of class $C^{1,\alpha}$) and are adjoint to each other. We also recall that $K_{\partial\Omega}$ is compact on the space $C^{0,\alpha}(\partial\Omega)$ of H\"older continuous functions, and $K^*_{\partial\Omega}$ is compact on $C^{1,\alpha}(\partial\Omega)$, the space of differentiable functions with H\"older continuous derivatives, as was noted by Schauder in \cite{Sc31, Sc32}. 

\paragraph{The double-layer potential.}

For a density function $\psi \in L^2(\partial\Omega)$, the double-layer potential $D_{\partial\Omega}[\psi]$ is defined by
\[
D_{\partial\Omega}[\psi](x) := \int_{\partial\Omega}\nu_\Omega(y)\cdot \nabla E_n(x-y)\psi(y)\,d\sigma_y \quad \text{for $x \in \mathbb{R}^n\setminus\partial\Omega$,}
\]
and if $\psi\in C^{0,\alpha}(\partial\Omega)$, then the restriction of $D_{\partial\Omega}[\psi]$ to $\Omega$ extends to a function $D^+_{\partial\Omega}[\psi]$ in $C^{0,\alpha}(\overline{\Omega})$. Similarly, the restriction of $D_{\partial\Omega}[\psi]$ to $\mathbb{R}^n\setminus\overline{\Omega}$ extends to a function $D^-_{\partial\Omega}[\psi]$ in $C^{0,\alpha}(\mathbb{R}^n\setminus{\Omega})$ (notice that $\mathbb{R}^n\setminus{\Omega}$ is the closure of an open set). 

Both $D^+_{\partial\Omega}[\psi]$ and $D^-_{\partial\Omega}[\psi]$ are defined on the boundary $\partial\Omega$, where we have the jump formulas
\begin{equation}\label{jump}
 D^{\pm}_{\partial\Omega}[\psi] = \pm \frac{1}{2}\psi +K_{\partial\Omega}[\psi] \qquad \mbox{ on } \partial \Omega\,.
\end{equation}

Since the double-layer potential \(D_{\partial\Omega}[\psi]\) is harmonic in \(\mathbb{R}^n \setminus \partial\Omega\), the jump formulas \eqref{jump} allow us to reduce both the interior and exterior Dirichlet problems to second-kind boundary integral equations for the operator \(K_{\partial\Omega}\). For instance, \(u = D^+_{\partial\Omega}[\psi]\) solves the Dirichlet problem

\begin{equation}\label{eq:dirbvp}
\begin{cases}
\Delta u = 0 \quad &\text{in } \Omega, \\
u = g \quad &\text{on } \partial \Omega,
\end{cases}
\end{equation}
whenever
\begin{equation}\label{eq:inteqdir}
\frac{1}{2} \psi + K_{\partial\Omega}[\psi] = g \quad \text{on } \partial \Omega.
\end{equation}
If both \(\Omega\) and its exterior are connected, \(\frac{1}{2}I + K_{\partial\Omega}\) is invertible, making problem \eqref{eq:dirbvp} and equation \eqref{eq:inteqdir} equivalent (the connectedness assumption is necessary, cf.~\cite[Chap.~6]{DaLaMu21}).

One of the earliest appearances of this approach can be found in the seminal works of Neumann (see, e.g., \cite{Ne77}). To solve equation \eqref{eq:inteqdir} by inverting the operator \(\frac{1}{2}I + K_{\partial\Omega}\), Neumann used what is now called the {\em Neumann series}. However, he was successful only under the assumption that \(\Omega\) is convex. Later, Poincar\'e \cite{Po97} extended Neumann's approach to domains that are not necessarily convex. Further developments were made by Korn \cite{Ko99}, Steklov \cite{St00}, and Zaremba \cite{Za97}.

For more detailed discussions on this topic, see the works of Costabel \cite{Co07}, Khavinson, Putinar, and Shapiro \cite{KhPuSh07}, and Wendland \cite{We09}. 

If $\psi$ is of class $C^{1,\alpha}$, then both $D^+_{\partial\Omega}[\psi]$ and $D^-_{\partial\Omega}[\psi]$ are of class $C^{1,\alpha}$. Namely, $D^+_{\partial\Omega}[\psi]\in C^{1,\alpha}(\overline\Omega)$ and $D^-_{\partial\Omega}[\psi]\in C^{1,\alpha}(\mathbb{R}^n\setminus\Omega)$. In this case, the normal derivatives of $D^+_{\partial\Omega}[\psi]$ and $D^-_{\partial\Omega}[\psi]$ do not jump at the boundary. To wit, if $\nu_\Omega$ denotes the outward unit normal to $\partial\Omega$, then  we have
\[
\nu_\Omega\cdot\nabla D^+_{\partial\Omega}[\psi]=\nu_\Omega\cdot\nabla D^-_{\partial\Omega}[\psi]\qquad\text{on }\partial\Omega.
\] 
The tangential derivatives have instead a jump, as described by Proposition \ref{thetangentialjump} in the appendix (see also \eqref{nablaDjump}).

It will be convenient to have a symbol for the normal derivative of the double-layer potential. So we set
\begin{equation}\label{Mr.T}
T_{\partial\Omega}[\psi]:=\nu_\Omega\cdot\nabla D^+_{\partial\Omega}[\psi]\qquad\text{on }\partial\Omega
\end{equation}
(or, equivalently, $T_{\partial\Omega}[\psi]:=\nu_\Omega\cdot\nabla D^-_{\partial\Omega}[\psi]$). We see that $T_{\partial\Omega}$ is a bounded operator from $C^{1,\alpha}(\partial\Omega)$ to $C^{0,\alpha}(\partial\Omega)$.
  
\paragraph{The single-layer potential.} The single-layer potential $S_{\partial\Omega}[\psi]$ is defined by 
\[
S_{\partial\Omega}[\psi](x) :=\int_{\partial\Omega} E_n(x-y)\psi(y)\,d\sigma_y \quad \text{for $x \in \mathbb{R}^n$ and all $\psi\in L^2(\partial\Omega)$,}
\]
and restricts to a function $S^+_{\partial\Omega}[\psi]$ in $C^{1,\alpha}(\overline\Omega)$ and to a function $S^-_{\partial\Omega}[\psi]$ in $C^{1,\alpha}(\mathbb{R}^n\setminus\Omega)$ whenever $\psi$ belongs to $C^{0,\alpha}(\partial\Omega)$. The single-layer potential is continuous on the whole of $\mathbb{R}^n$, but its derivative is not. Particularly,  we have the following jump formula for the gradient of the single-layer potential:
\begin{equation}\label{nablasinglejump}
\nabla S^{\pm}_{\partial\Omega}[\psi](x)=\pm\frac{\nu_\Omega(x)}{2}\psi(x)+\int^*_{\partial\Omega}\psi(y)\nabla E_n(x-y)\,d\sigma_y\qquad\text{for all }x\in\partial\Omega,
\end{equation}
where $\int^*_{\partial\Omega}$ denotes the principal value integral (cf.~e.g.~\cite[Theorem 6.8]{Da13}, see also \eqref{intstar} in the appendix for the definition of $\int^*_{\partial\Omega}$). Then, for the normal derivative we have
\begin{equation}\label{nusinglejump}
\nu_\Omega\cdot\nabla S^{\pm}_{\partial\Omega}[\psi]=\pm\frac{1}{2}\psi-K^*_{\partial\Omega}[\psi]\qquad\text{on }\partial\Omega,
\end{equation}
while the tangential derivative of $S_{\partial\Omega}[\psi]$ displays no jump at the boundary. To wit,
\[
\nabla_{\partial\Omega} S^+_{\partial\Omega}[\psi]=\nabla_{\partial\Omega} S^-_{\partial\Omega}[\psi]\qquad\text{on }\partial\Omega
\]
(cf.~\eqref{tangrad} for the definition of $\nabla_{\partial\Omega}$). Just as the double-layer potential is used to reduce the Dirichlet problem to a second-kind integral equation, the jump formulas \eqref{nusinglejump} make the single-layer potential the principal tool for addressing the Neumann boundary value problem. 

We emphasize that our single-layer potential has the opposite sign compared to that in \cite{Fo95} and \cite{DaLaMu21}. This is intentional, as we prefer the single-layer operator to be positive with respect to the inner product on \( L^2(\partial\Omega) \).

\paragraph{The plasmonic problem.} For \(n \geq 3\), the map that takes \(\psi\) to \(S_{\partial\Omega}[\psi]_{|\partial\Omega}\) is an isomorphism between \(C^{0,\alpha}(\partial\Omega)\) and \(C^{1,\alpha}(\partial\Omega)\). Using this, we can prove that any harmonic function \(u\) in \(C^{1,\alpha}(\overline\Omega)\) can be uniquely written as a single-layer potential \(S^+_{\partial\Omega}[\psi]\), with density \(\psi\) in \(C^{0,\alpha}(\partial\Omega)\). Similarly, any harmonic function \(v\) in \(\mathbb{R}^n \setminus \overline\Omega\), locally of class \(C^{1,\alpha}\) on \(\overline{\mathbb{R}^n \setminus \overline\Omega} = \mathbb{R}^n \setminus \Omega\), and harmonic at infinity (i.e., such that \(v(x) = O(|x|^{2-n})\) as \(|x| \to \infty\)), can be uniquely represented as a single-layer potential \(S^-_{\partial\Omega}[\psi]\), with density \(\psi\) in \(C^{0,\alpha}(\partial\Omega)\). 

In the two-dimensional case, the situation is slightly more complicated. Harmonic functions in both \(\Omega\) and \(\mathbb{R}^n \setminus \overline{\Omega}\), which are $C^{1,\alpha}$ in $\overline{\Omega}$ and locally in $\mathbb{R}^n \setminus \Omega$, and harmonic at infinity, can be represented as the sum of a single-layer potential \(S^\pm_{\partial\Omega}[\psi]\), with density \(\psi\) in \(C^{0,\alpha}(\partial\Omega)\) and satisfying \(\int_{\partial\Omega} \psi \, d\sigma = 0\), and a constant function. For further details, see \cite[Chap.~6.10]{DaLaMu21}.

In any case, both for \(n \geq 3\) and \(n = 2\), we observe that, provided \(\lambda \neq 1/2\), the NP-eigenvalue problem
\begin{equation}\label{eigen*}
K^*_{\partial\Omega}[\psi] = \lambda \psi
\end{equation}
is equivalent to the {\em plasmonic problem}:
\begin{equation}\label{plasmonic}
\begin{cases}
\Delta u = 0 &\text{in } \Omega, \\
\Delta v = 0 &\text{in } \mathbb{R}^n \setminus \overline{\Omega}, \\
v(x) = O(|x|^{2-n}) &\text{as } |x| \to \infty, \\
u - v = 0 &\text{on } \partial\Omega, \\
\epsilon \, \nu_\Omega \cdot \nabla u + \nu_\Omega \cdot \nabla v = 0 &\text{on } \partial\Omega,
\end{cases}
\end{equation}
with
\[
\epsilon = \frac{\frac{1}{2} + \lambda}{\frac{1}{2} - \lambda}\,.
\]

For \(n \geq 3\), the equivalence between \eqref{eigen*} and \eqref{plasmonic} can be established by taking \(u = S^+_{\partial\Omega}[\psi]\) and \(v = S^-_{\partial\Omega}[\psi]\), and using the jump formulas \eqref{nusinglejump}. In the case of \(n = 2\), we take \(u = S^+_{\partial\Omega}[\psi] + \rho\) and \(v = S^-_{\partial\Omega}[\psi] + \rho\), where \(\int_{\partial\Omega} \psi \, d\sigma = 0\) and \(\rho \in \mathbb{R}\). It should be noted that \eqref{eigen*} implies \(\psi\) has zero integral on \(\partial\Omega\) when \(\lambda \neq 1/2\) (cf.~\cite[Lemma 6.11]{DaLaMu21}).

The values of \(\varepsilon\) for which problem \eqref{eigen*} has solutions are referred to as {\em plasmonic eigenvalues}, and the corresponding solutions are called {\em plasmonic eigenfunctions}. Readers interested in this topic will find a substantial body of literature in physics on plasmonic waves.  For example, some of the effects of plasmonic resonance are described in the introduction to the paper of Bonnetier, Dapogny, and Triki \cite{BoDaTr19}, which then discusses the NP spectrum in a domain with small vanishing holes. For an interpretation related to perturbation problems of the kind considered in our paper, we refer to Grieser et al.~\cite{GrEtAl09}. We also mention the connection between the literature on the NP spectrum (and plasmonic resonance) and the extensive body of work devoted to cloaking by anomalous localized resonance (see, e.g., Ammari et al.~\cite{AmCiKaLeMi13}) as well as to the problem of estimating the energy between two close-to-touching inclusions (see Ammari et al.~\cite{AmCiKaLeYu13} for the case of anti-plane elasticity).

Occasionally, it will be more aesthetically pleasing to express results in terms of plasmonic eigenfunctions rather than NP-eigenfunctions. This is the case, for example, with the derivative formula in Theorem \ref{dLambda} and the Rellich-Poho\v{z}aev formula \eqref{RPL.eq3}.

\paragraph{The Calder\'on identity.} As mentioned in the introduction, we may work with either \(K_{\partial\Omega}\) or \(K^*_{\partial\Omega}\), but \(K^*_{\partial\Omega}\) offers certain advantages due to the Calder\'on identity, a proof of which can be found in Lanza \cite{La24} and Mitrea and Taylor \cite[(7.41)]{MiTa99}. For the sake of completeness, we present a concise argument.

Suppose $u\in C^{1,\alpha}(\overline\Omega)$ is harmonic in $\Omega$. Utilizing the third Green's identity (cf.~\cite[Theorem 6.9]{DaLaMu21}), we have
\[
 D_{\partial\Omega}\left[u_{|\partial\Omega}\right](x)+ S_{\partial\Omega}\left[\nu_\Omega\cdot\nabla u_{|\partial\Omega}\right](x)=0\quad\text{for all $x\in \mathbb{R}^n\setminus\overline\Omega$.}
\]
Thus, for $u= S^+_{\partial\Omega}[\mu]$ with $\mu\in C^{0,\alpha}(\partial\Omega)$, we obtain
\[
 D_{\partial\Omega}\left[ S_{\partial\Omega}[\mu]_{|\partial\Omega}\right](x)+ S_{\partial\Omega}\left[\nu_\Omega\cdot\nabla S^+_{\partial\Omega}[\mu]_{|\partial\Omega}\right](x)=0\quad\text{for all $x\in \mathbb{R}^n\setminus\overline\Omega$.}
\]
By letting $x\in \mathbb{R}^n\setminus\overline\Omega$ approach a point $x'\in\partial\Omega$ and employing the jump formulas for the double-layer potential \eqref{jump} and the normal derivative of the single layer potential \eqref{nusinglejump}, we derive
\[
-\frac{1}{2} S_{\partial\Omega}[\mu](x')+K_{\partial\Omega}\left[ S_{\partial\Omega}[\mu]_{|\partial\Omega}\right](x')+ S_{\partial\Omega}\left[\frac{1}{2}\mu-K^*_{\partial\Omega}[\mu]\right](x')=0\quad\text{for all $x'\in \partial\Omega$.}
\]
Thus,
\[
K_{\partial\Omega}\left[ S_{\partial\Omega}[\mu]_{|\partial\Omega}\right](x')- S_{\partial\Omega}\left[K^*_{\partial\Omega}[\mu]\right](x')=0\quad\text{for all $x'\in \partial\Omega$.}
\]
Consequently, $K_{\partial\Omega} S_{\partial\Omega}= S_{\partial\Omega}K^*_{\partial\Omega}$ on $C^{0,\alpha}(\partial\Omega)$, and this equality extends to $L^2(\partial\Omega)$ by a density-plus-continuity argument. 

A consequence of the Calder\'on identity is that we can symmetrize $K^*_{\partial\Omega}$. First, we define a new inner product on $L^2(\partial\Omega)$ as follows: 
\begin{equation*}
\langle f, g \rangle_{ S_{\partial\Omega}} := \langle  S_{\partial\Omega}[f], g \rangle_{L^2(\partial\Omega)} \qquad \text{for all $f,g \in L^2(\partial\Omega)$.}
\end{equation*}
We observe that, with respect to this inner product, $K^*_{\partial\Omega}$ is self-adjoint by \eqref{PSP}. We then define the energy space $\mathcal{E}$ as the completion of $L^2(\partial\Omega)$ 
under the norm 
\[
\|f \|^2_{ S_{\partial\Omega}} := \langle  S_{\partial\Omega}[f], f \rangle_{L^2(\partial\Omega)}.
\]
So, for a sufficiently regular function $f$, and assuming that $\int_{\partial\Omega} f\,d\sigma=0$ if $n=2$, we have
\[
\|f \|^2_{ S_{\partial\Omega}}=\int_{\Omega}\left|\nabla S^+_{\partial\Omega}[f]\right|^2\,dx+\int_{\mathbb{R}^n\setminus\overline{\Omega}}\left|\nabla S^-_{\partial\Omega}[f]\right|^2\,dx\,.
\]

It can be shown that  $\mathcal{E}=H^{-1/2}(\partial\Omega)$, with equivalent norms. Morever, $K^*_{\partial\Omega}$ is a compact, self-adjoint operator on $\mathcal{E}$ and the spectra of $K^*_{\partial\Omega}$ on $\mathcal{E}$ and on $L^2(\partial\Omega)$ are identical. For a proof of these facts, we refer to \cite{KhPuSh07}. For the general theory of symmetrizable operators, we refer to Krein \cite{Kr47}, and also Dieudonn\'e \cite{Di61} and Lax \cite{La54}.

\bigskip

Other notions of potential theory will be recalled as needed in the paper.

\section{Riesz projectors and symmetrizable operators}\label{sec:RP}

In this section, we present some results concerning the Riesz projector for a symmetrizable operator acting on Hilbert spaces.

Let $\left(H,\langle\cdot,\cdot\rangle\right)$ be a  (complex) Hilbert space, and let $B:H \to H$ be a linear bounded operator.  We denote the resolvent set of $B$ by $\rho(B)$ and its spectrum by $\sigma(B):= \mathbb{C}\setminus \rho(B)$. For all $\xi \in \rho(B)$, $R_B(\xi)$ represents the resolvent of $B$, defined as the bounded linear operator from $H$ to itself given by
\[
R_B(\xi) := (B-\xi)^{-1}.
\]
Here, the superscript $\cdot^{-1}$ indicates that we are taking the inverse of an invertible linear operator acting on Hilbert spaces.

Let $\gamma:[0,1]\to\mathbb{C}$ be a simple closed path with positive orientation. We define $\Gamma := \gamma([0,1])$, meaning that $\Gamma$ is the image of the path $\gamma$. 
Then $\Gamma$ divides $\mathbb{C}$ into two open connected components, one bounded and one unbounded. We denote the bounded component of $\mathbb{C}\setminus\Gamma$ by $\Gamma^i$. 
Suppose $\lambda$ is an isolated eigenvalue of $B$.  If 
\begin{equation}\label{gamma}
\Gamma \cap \sigma(B) =\emptyset\quad\text{and} \quad \Gamma^i \cap \sigma(B) = \{\lambda\},
\end{equation}
we say that the path $\gamma$  surrounds (or is around) $\lambda$.
The Riesz projector for the path $\gamma$ around $\lambda$ is the bounded linear operator defined by 
\[
P_{\gamma}[B] := -\frac{1}{2\pi i}\int_{\Gamma} R_B(\xi)\,d\xi \quad \text{in } H.
\]
The operator $P_\gamma[B]$ is a projection, and if $M' := P_\gamma[B] H$, then $\sigma(B_{|M'})= \{\lambda\}$ (see Kato \cite[Thm 6.17, p. 178]{Ka95}).
Moreover, the integral defining the Riesz projector does not depend on the specific choice of the path $\gamma$ around $\lambda$ (see e.g.  Hislop and Sigal \cite[Lemma 6.1, p.~61]{HiSi96}).

It is well known that the kernel of $B-\lambda$ is contained in the image of $P_\gamma[B]$, that is 
\begin{equation}\label{ker-P}
\ker (B-\lambda) \subseteq P_\gamma[B] H.
\end{equation}
To see why, let    $u \in \ker (B-\lambda)$. That is, 
$(B-\lambda)u=0$. Then, for all $\xi \in \rho(B)$, we have $(B-\xi)u=(\lambda -\xi)u$. We deduce that 
 \[
R_B(\xi)u=\frac{u}{\lambda-\xi}.
 \]
 Integrating this relation over $\Gamma$, we obtain
 \begin{equation}\label{Puu}
 P_\gamma[B] u = -\frac{1}{2\pi i}\int_{\Gamma} R_B(\xi)u\,d\xi =  -\frac{1}{2\pi i}\int_{\Gamma}\frac{u}{\lambda-\xi}\,d\xi = u,
 \end{equation}
and we conclude that $u \in P_\gamma[B] H$.  (This straightforward argument has been adapted from \cite[Prop.~6.3 (ii), p.~62]{HiSi96}.)

In general, the inclusion in \eqref{ker-P} is strict. However, if the operator $B$ is self-adjoint, then we have equality
\begin{equation}\label{ker-P2}
\ker (B-\lambda) = P_\gamma[B] H.
\end{equation}
For a proof of this fact, we refer to \cite[Prop.~6.3 (iii), p.~62]{HiSi96}.
Following an idea of Khavinson et al.~\cite{KhPuSh07}, we will now demonstrate that this equality holds also for operators that are not necessarily self-adjoint, but at least symmetrizable.

We say that $B$ is symmetrizable if there exists a  strictly positive self-adjoint operator $S \in \mathcal{L}(H)$ such that 
\begin{equation}\label{sym}
SB=B^*S.
\end{equation}
We recall that strictly positive means that
\[
\langle Su, u \rangle > 0 \qquad \text{for all $u \in H\setminus\{0\}$.}
\]
Then
\[
\langle u,v\rangle_S :=\langle Su, v \rangle  
\]
defines a new inner product in $H$, and $B$ becomes self-adjoint in $\left(H,\langle\cdot,\cdot\rangle_S\right)$ by \eqref{sym}. However, we cannot yet affirm the validity of \eqref{ker-P2}. The issue lies in the fact that, with this newly defined inner product, $H$ might not be complete, and therefore, may not be a Hilbert space. This is precisely the case with the NP-operator $K^\ast_{\partial\Omega}$.

Anyway, it is still possible to prove   \eqref{ker-P2} for a compact symmetrizable operator.
\begin{lemma}\label{lem:kerP}
Let $\left(H,\langle\cdot,\cdot\rangle\right)$ be a Hilbert space and $B$ a compact symmetrizable operator in $H$. Let $\lambda\neq 0$ be an (isolated) eigenvalue of $B$. 
Let $\gamma$ be a simple closed path around $\lambda$  with positive orientation (cf.~\eqref{gamma}).  Then
\[
\ker (B-\lambda) = P_\gamma[B] H.
\]
\begin{proof}
Since $\ker (B-\lambda)\subseteq P_\gamma[B] H $ by \eqref{ker-P}, it suffices to prove that $ P_\gamma[B] H \subseteq\ker (B-\lambda)$. We first observe that $P_\gamma[B]$ is a compact operator. Indeed, 
\[
R_B(\xi) +\xi^{-1} =\xi^{-1}BR_B(\xi) \qquad \text{for all }\xi \in \rho(B)
\]
and since $B$ is compact, it follows that $R_B(\xi) +\xi^{-1}$ is compact for all $\xi \in \rho(B)$. Moreover, $\int_\Gamma \xi^{-1}\,d\xi=0$ whenever $0\notin\Gamma^i$, which can be assumed (see \cite[p.~186]{Ka95}). Then 
\[
P_{\gamma}[B] = -\frac{1}{2\pi i}\int_{\Gamma} R_B(\xi)\,d\xi=-\frac{1}{2\pi i}\int_{\Gamma} R_B(\xi)+\xi^{-1}\,d\xi
\]  is compact as the integral of compact operators. By a standard result, a projection is compact if and only if its range is finite-dimensional (see e.g., \cite[Pb.~4.5, p.~157]{Ka95}). Thus, $M':=P_\gamma[B] H$ has finite dimension.

Now, let $m \in \mathbb{N}$, $m\geq1$, denote the dimension of $M'$. Since $B_{|M'}$ is self-adjoint on the finite-dimensional space $\left(M',\langle\cdot,\cdot\rangle_{S}\right)$,
there exists a basis $\{v_1,\ldots,v_m\}$ for $M'$ such that $Bv_j = \lambda_jv_j$ for some $\lambda_j \in \mathbb{R}$ for all $j=1,\ldots,m$.
This implies that 
\[
\{\lambda_j\}_{j=1,\ldots,m} \subseteq \sigma(B_{|M'}) = \{\lambda\}.
\]
Therefore $\lambda_j=\lambda$ for all $j=1,\ldots,m$  and $M'= P_\gamma[B] H \subseteq\ker (B-\lambda)$.
\end{proof}

\begin{remark}\label{rem:kerP}
All the discussion in this section can be straightforwardly generalized to the case of several eigenvalues. More precisely, if in Lemma \ref{lem:kerP} we replace the assumption that the path $\gamma$ surrounds $\lambda$ and instead we assume that $\gamma$ surrounds a finite set $\{\lambda_1,\ldots,\lambda_m\}$ of eigenvalues, than we have
\[
\bigoplus_{j=1}^m\ker\left(B-\lambda_j\right)=P_\gamma[B] H.
\]
\end{remark}
\end{lemma}

\section{Shape analyticity of the symmetric functions of the NP-eigenvalues}\label{sec:shape}

We show that the symmetric functions of the NP-eigenvalues depend analytically upon variations of the supporting domain. To facilitate our analysis, we set ourselves in the framework of Schauder spaces, where we can use the real analyticity results established  by Lanza and collaborators (see \cite{DaLuMu22}, as well as the monograph \cite{DaLaMu21}).

 In the following Lemma \ref{lem:reg} we show that under the assumption that the domain is of class $C^{1,\alpha}$,  the eigenfunctions of $K^*_{\partial\Omega}$ are $\alpha$-H\"older continuous. 

\begin{lemma}\label{lem:reg}
Let $\alpha \in (0,1)$. Let $\Omega \subseteq \mathbb{R}^n$ satisfy assumption \eqref{Omega_def}. 
If for $\lambda \in \mathbb{R}$, $\lambda \neq 0$, and $\psi \in L^2(\partial\Omega)$ we have
\begin{equation}\label{lem:reg.eq1}
K^*_{\partial\Omega}[\psi] =\lambda\psi,
\end{equation}
then $\psi \in C^{0,\alpha}(\partial\Omega)$.
\end{lemma}
\begin{proof} Classical results in potential theory, derived from computations of the composition kernels, show that the integral operator $\left(K^*_{\partial\Omega}\right)^m$ has a singularity of order $O\left(\frac{1}{|x-y|^{n-1-m\alpha}} \right)$ for $x,y \in \partial\Omega$, $x \neq y$ (see e.g., Miranda \cite[Ch. 2 \S 11]{Mi70}). Thus, for a sufficiently large $m$, specifically for $m>(n-1)/\alpha$, the kernel of $\left(K^*_{\partial\Omega}\right)^m$ is continuous, and consequently, $\left(K^*_{\partial\Omega}\right)^m[\psi]$ is continuous as well. 

Applying $K^*_{\partial\Omega}$ to both sides of equation \eqref{lem:reg.eq1} for $(m-1)$ iterations, we obtain
\[
\left(K^*_{\partial\Omega}\right)^m[\psi] = \lambda^m \psi.
\]
Since $\lambda\neq 0$, we deduce that $\psi$ is also continuous. 

Recalling that, for $\beta\in (0,\alpha)$, $K^*_{\partial\Omega}$ maps $C^0(\partial\Omega)$ to $C^{0,\beta}(\partial\Omega)$, and $C^{0,\beta}(\partial\Omega)$ to $C^{0,\alpha}(\partial\Omega)$ (see e.g., Dondi and Lanza \cite[Thm. 10.1]{DoLa17}, as well as Schauder \cite{Sc31, Sc32} and Miranda \cite{Mir65, Mi70}), equality \eqref{lem:reg.eq1} and a standard bootstrapping argument imply that $\psi$ belongs to $C^{0,\alpha}(\partial\Omega)$. This concludes the proof. \end{proof}

We now introduce a class of domain perturbations. Let $\Omega$ satisfy assumption \eqref{Omega_def}. The set $\Omega$ will play the role of a fixed reference domain,  which we are going to perturb with a diffeomorphism. We define a specific set $\mathcal{A}^{1,\alpha}_{\partial \Omega}$ of $C^{1,\alpha}$-diffeomorphisms: $\mathcal{A}^{1,\alpha}_{\partial \Omega}$ consists of functions in $C^{1,\alpha}(\partial\Omega, \mathbb{R}^{n})$ that are injective and have injective differentials at all points of $\partial\Omega$ (see \cite[\S2.20]{DaLaMu21} for the definition of Schauder spaces on the boundary of a domain). By results from Lanza and Rossi \cite[Lemma 2.2, p.~197]{LaRo08} and \cite[Lemma 2.5, p.~143]{LaRo04}, we know that the set $\mathcal{A}^{1,\alpha}_{\partial \Omega}$ is open in $C^{1,\alpha}(\partial\Omega, \mathbb{R}^{n})$. Moreover, if $\phi \in \mathcal{A}^{1,\alpha}_{\partial \Omega}$, then the Jordan-Leray separation theorem ensures that $\phi(\partial\Omega)$ splits $\mathbb{R}^n$ into exactly two open connected components, one bounded and one unbounded (see e.g., Deimling \cite[Theorem 5.2, p. 26]{De85}). We denote the bounded connected component of $\mathbb{R}^n\setminus \phi(\partial \Omega)$ by $\Omega[\phi]$, so that $\partial\Omega[\phi]=\phi(\partial\Omega)$. We see that $\Omega[\phi]$ is a bounded subset of $\mathbb{R}^n$ of class $C^{1,\alpha}$, is connected, and has a connected exterior $\mathbb{R}^n\setminus\overline{\Omega[\phi]}$. Namely, it satisfies assumptions \eqref{Omega_def}. 

Then, for a diffeomorphism $\phi \in \mathcal{A}^{1,\alpha}_{\partial \Omega}$, we can  write $K_{\phi(\partial\Omega)}$, $K^*_{\phi(\partial\Omega)}$, $D_{\phi(\partial\Omega)}$, $D^\pm_{\phi(\partial\Omega)}$, $S_{\phi(\partial\Omega)}$, and $S^\pm_{\phi(\partial\Omega)}$, extending the definitions of Sections \ref{sec:intro} and \ref{sec:potential} from the case of the domain $\Omega$ to $\Omega[\phi]$.  We consider the eigenvalue problem set on the $\phi$-dependent boundary $\partial \Omega[\phi]= \phi(\partial\Omega)$:
\begin{equation}\label{pb:phi}
K^\ast_{\phi(\partial \Omega)}[\psi] = \lambda \psi \quad \text{on} \quad \phi(\partial\Omega),
\end{equation}
where $\lambda \in \mathbb{R}$ and $\psi \in C^{0,\alpha}(\phi(\partial\Omega))$. We note that, thanks to Lemma \ref{lem:reg}, we do not lose any solutions by considering $\psi \in C^{0,\alpha}(\phi(\partial\Omega))$ instead of $L^2(\phi(\partial\Omega))$, since any $L^2(\phi(\partial\Omega))$ eigenfunction is also in $C^{0,\alpha}(\phi(\partial\Omega))$. However, it is not convenient to work with a space, $C^{0,\alpha}(\phi(\partial\Omega))$, that depends on the perturbation parameter $\phi$. Therefore, we pull back the problem to the fixed boundary $\partial \Omega$, and accordingly, we push forward the density from $\partial \Omega$ to $\phi(\partial\Omega)$.
Specifically, we define the operator
\[
K^\ast_\phi[\mu] := K^\ast_{\phi(\partial \Omega)}[\mu \circ \phi^{-1}] \circ \phi \quad \text{for all } \mu \in C^{0,\alpha}(\partial\Omega),
\]
and we consider the eigenvalue problem
\begin{equation}\label{pb:phi2}
K^\ast_{\phi}[\mu] = \lambda \mu \quad \text{on} \quad \partial\Omega,
\end{equation}
for $\lambda \in \mathbb{R}$ and $\mu \in C^{0,\alpha}(\partial\Omega)$. We observe that problem \eqref{pb:phi2} has exactly the same eigenvalues $\lambda$ as problem \eqref{pb:phi}, and the map $\mu\mapsto \mu\circ\phi^{-1}$ is a bijection between the corresponding eigenspaces.

\subsection{Shape analyticity of the Riesz projector}
Before showing the result for the symmetric functions of the eigenvalues, we prove an intermediate result on the real analyticity of the Riesz projector for the operator $K^\ast_{\phi}$. 

Let $\phi_0 \in \mathcal{A}^{1,\alpha}_{\partial\Omega}$ and $\gamma$ be a simple closed path in $\mathbb{C}$ that does not intersect $\sigma(K^*_{\phi_0})$.
We will show that the map that takes $\phi$ to $P_\gamma[K^\ast_{\phi}]$ is real analytic in a neighborhood of $\phi_0$ in $\mathcal{A}^{1,\alpha}_{\partial\Omega}$. We begin with a lemma.

\begin{lemma}\label{lem:realR}
Let $\alpha \in (0,1)$. Let $\Omega \subseteq \mathbb{R}^n$ satisfy assumptions \eqref{Omega_def}. Let $\phi_0 \in \mathcal{A}^{1,\alpha}_{\partial\Omega}$.
Let $\gamma$ be  a simple closed path in $\mathbb{C}$ such that $\Gamma\cap\sigma(K^*_{\phi_0})=\emptyset$.
 Then there exists an open  neighborhood  $\mathcal{U}_0$ of $\phi_0$ in $\mathcal{A}^{1,\alpha}_{\partial\Omega}$
such that $\Gamma \cap \sigma(K^\ast_{\phi})=\emptyset$ for all $\phi \in \mathcal{U}_0$, and the map from $\mathcal{U}_0$ to $C\left(\Gamma, \mathcal{L}(C^{0,\alpha}(\partial\Omega))\right)$ that takes $\phi$ to the map
\[
\Gamma \ni \xi \mapsto \left(K^\ast_{\phi}-\xi I\right)^{-1} \in  \mathcal{L}(C^{0,\alpha}(\partial\Omega)),
\]
is real analytic.
\end{lemma}
Here above $\left(K^\ast_{\phi}-\xi I\right)^{-1}$ is the inverse of the operator $K^\ast_{\phi}-\xi I \in \mathcal{L}(C^{0,\alpha}(\partial\Omega))$.
\begin{proof}
We observe that $\mathcal{L}(C^{0,\alpha}(\partial\Omega))$ forms a Banach algebra, where the composition of operators acts as a non-commutative multiplication. Similarly, the space $C\left(\Gamma, \mathcal{L}(C^{0,\alpha}(\partial\Omega))\right)$ of continuous maps from $\Gamma$ to $\mathcal{L}(C^{0,\alpha}(\partial\Omega))$, equipped with the supremum norm, is a Banach algebra with non-commutative multiplication defined by
\[
(\xi\mapsto A(\xi))\ast (\xi\mapsto B(\xi)) := (\xi\mapsto A(\xi)\circ B(\xi))
\]
for all elements $(\xi\mapsto A(\xi))$ and $(\xi\mapsto B(\xi))$ in $C\left(\Gamma, \mathcal{L}(C^{0,\alpha}(\partial\Omega))\right)$.
If $\mathcal{I}$ denotes the set of elements in $C\left(\Gamma, \mathcal{L}(C^{0,\alpha}(\partial\Omega))\right)$ that are invertible with respect to $\ast$, then a standard argument based on the Neumann series shows that $\mathcal{I}$ is open in $C\left(\Gamma, \mathcal{L}(C^{0,\alpha}(\partial\Omega))\right)$, and the map that associates an element $(\xi\mapsto A(\xi))$ of $\mathcal{I}$ with its inverse $(\xi\mapsto A(\xi))^{-1}=(\xi\mapsto A(\xi)^{-1})$ is real analytic.

We now use the real analyticity of layer potential operators in Schauder spaces.
By \cite[Theorem 3.2 (iii)]{DaLuMu22}, we know that the map 
\[
\mathcal{A}^{1,\alpha}_{\partial\Omega} \ni \phi \mapsto K^\ast_{\phi} \in  \mathcal{L}(C^{0,\alpha}(\partial\Omega))
\]
is real analytic. Consequently, the map from $\mathcal{A}^{1,\alpha}_{\partial\Omega}$ to  $C\left(\Gamma, \mathcal{L}(C^{0,\alpha}(\partial\Omega))\right)$ that associates 
$\phi$ with 
\begin{equation}\label{reK}
\Gamma \ni \xi \mapsto (K^\ast_{\phi}-\xi I) \in  \mathcal{L}(C^{0,\alpha}(\partial\Omega)),
\end{equation}
is also real analytic. Moreover,  the map $(\xi \mapsto (K^\ast_{\phi_0}-\xi I))$ is an invertible element of 
$C\left(\Gamma, \mathcal{L}(C^{0,\alpha}(\partial\Omega))\right)$. Since we have seen that the set $\mathcal{I}$ of invertible maps is open in $C\left(\Gamma, \mathcal{L}(C^{0,\alpha}(\partial\Omega))\right)$, it follows that there exists an open neighborhood $\mathcal{U}_0$ 
of $\phi_0$ in $\mathcal{A}^{1,\alpha}_{\partial\Omega}$ such that 
\[
\left(\xi \mapsto (K^\ast_{\phi}-\xi I)\right) \in  \mathcal{I} \quad\text{for all $\phi \in \mathcal{U}_0$.}
\]
To complete the proof of the lemma, we compose the map that takes $\phi\in\mathcal{U}_0$ to the map in \eqref{reK} and the inversion map from $\mathcal{I}$ to $C\left(\Gamma, \mathcal{L}(C^{0,\alpha}(\partial\Omega))\right)$, and recall that the composition of real analytic maps remains real analytic.
\end{proof} 

As an immediate consequence of Lemma \ref{lem:realR} we obtain the desired result on the Riesz projector.

\begin{proposition}\label{prop:anR}
Let $\alpha \in (0,1)$. Let $\Omega \subseteq \mathbb{R}^n$ satisfy assumption \eqref{Omega_def}. Let $\phi_0 \in \mathcal{A}^{1,\alpha}_{\partial\Omega}$.
Let $\gamma$ be  a simple closed path in $\mathbb{C}$ such that $\Gamma\cap\sigma(K^*_{\phi_0})=\emptyset$.
 Then there exists an open  neighborhood  $\mathcal{U}_0$ of $\phi_0$ in $\mathcal{A}^{1,\alpha}_{\partial\Omega}$
such that  the map from $\mathcal{U}_0$ to $\mathcal{L}(C^{0,\alpha}(\partial\Omega))$ that takes $\phi$ to 
 $P_\gamma[K^\ast_{\phi}]$ is real analytic.
\end{proposition}
\begin{proof}
The map
\[
C\left(\Gamma, \mathcal{L}(C^{0,\alpha}(\partial\Omega))\right) \ni f \mapsto \int_{\Gamma}f(\xi)\,d\xi \in \mathcal{L}\left(C^{0,\alpha}(\partial\Omega)\right)
\]
is linear and bounded, and hence, it is real analytic. Consequently, the statement follows by composing this map with the real analytic map from Lemma \ref{lem:realR}, and recalling that the composition of real analytic maps remains real analytic.
\end{proof}

 \subsection{Shape analyticity of the symmetric functions}
We are now ready to prove the real analyticity result for the simple eigenvalues and the symmetric functions of multiple eigenvalues of equation \eqref{pb:phi2}. As noted in the introduction, we opt to consider symmetric functions rather than analytic branches of eigenvalues due to the multidimensional nature of the shape function $\phi\in \mathcal{A}^{1,\alpha}_{\partial\Omega}$ (which is, indeed, infinite-dimensional). In fact, the existence of analytic branches is notoriusly not guaranteed in the case of multidimensional perturbations (cf.~Rellich \cite[p.~37]{Re69}). However, we may use Theorem \ref{main} below to recover a result for the analytic branches by introducing a one-dimensional analytic parametrization $t\mapsto\phi_t$ of the shape function (cf.~Section \ref{sec:grieser}). 

We recall the definition of symmetric function introduced earlier. For $m \in \mathbb{N}$, $m \geq 1$, and $h \in \{1,\ldots,m\}$, the symmetric function $\Lambda_h^m$ is defined by
\begin{equation}\label{symmetric}
\Lambda_h^m(\xi_{1},\ldots, \xi_{m}) := \sum_{\substack{j_1,\ldots,j_h\in \{1,\ldots, m\}\\ j_1 < \cdots < j_h}} \xi_{j_1} \cdots \xi_{j_h},
\end{equation}
for all $\xi_1,\ldots,\xi_m \in \mathbb{R}$.

Additionally, for a diffeomorphism \(\phi \in \mathcal{A}^{1,\alpha}_{\partial\Omega}\), we define the operator \( S_\phi\), which maps \(\mu \in C^{0,\alpha}(\partial\Omega)\) to
\[
S_\phi[\mu] := S_{\phi(\partial\Omega)}[\mu\circ\phi^{-1}] \circ \phi \in C^{1,\alpha}(\partial\Omega),
\]
(recall that the single-layer operator maps \(C^{0,\alpha}(\partial\Omega)\) to \(C^{1,\alpha}(\partial\Omega)\); see Section \ref{sec:potential}). We then consider the inner product 
\begin{equation}\label{<>sS}
\langle f,g\rangle_{\sigma_n[\phi] S_\phi} := \langle \sigma_n[\phi]\, S_\phi[f],g\rangle_{L^2(\partial\Omega)} = \langle S_{\phi(\partial\Omega)}[f\circ\phi^{-1}],g\circ\phi^{-1}\rangle_{L^2(\phi(\partial\Omega))},
\end{equation}
for all \(f,g\in C^{0,\alpha}(\partial\Omega)\). The function \(\sigma_n[\phi]\) in \eqref{<>sS} accounts for the change of variable in the area element, so that
\[
\int_{\phi(\partial\Omega)}f\circ\phi^{-1}\,d\sigma=\int_{\partial\Omega}f\,\sigma_n[\phi]\,d\sigma,
\]
for all \(f\in L^1(\partial\Omega)\). The explicit form of \(\sigma_n[\phi]\) can be found in \cite{LaRo04}, where it is also shown that \(\sigma_n[\phi] \in C^{0,\alpha}(\partial\Omega)\), and that the map sending \(\phi\in \mathcal{A}^{1,\alpha}_{\partial\Omega}\) to \(\sigma_n[\phi]\in C^{0,\alpha}(\partial\Omega)\) is real analytic.

We can now prove the following

\begin{theorem}\label{main}
Let $\alpha \in (0,1)$. Let $\Omega \subseteq \mathbb{R}^n$ satisfy assumption \eqref{Omega_def}. Let $\phi_0 \in \mathcal{A}^{1,\alpha}_{\partial\Omega}$,  and let 
$\lambda \neq 0$ be an eigenvalue of  $K^\ast_{\phi_0}$ of multiplicity $m \in \mathbb{N}$, $m\geq 1$. Then the following statements hold:
\begin{itemize}
\item[(i)] There exist $\delta >0$ and an open neighborhood $\mathcal{U}$ of $\phi_0$ in $\mathcal{A}^{1,\alpha}_{\partial\Omega}$ such that for all $\phi \in \mathcal{U}$, the set $(\lambda-\delta,\lambda+\delta)$ contains $m$ eigenvalues $\lambda_{1}[\phi]\le\lambda_2[\phi]\le\dots\le\lambda_{m}[\phi]$ of $K^\ast_{\phi}$, counting multiplicities.
\item[(ii)] There exist real analytic maps $\mathcal{U}\ni\phi\mapsto \mu_j[\phi]\in C^{0,\alpha}(\partial\Omega)$, $j=1,\ldots,m$, such that $\{\mu_1[\phi], \ldots, \mu_m[\phi]\}$ forms a $(\langle \cdot,\cdot\rangle_{\sigma_n[\phi] S_\phi})$-orthonormal basis for the space spanned by the eigenfunctions corresponding to $\lambda_{1}[\phi]\le\lambda_2[\phi]\le\dots\le\lambda_{m}[\phi]$.
\item[(iii)] The function taking $\phi\in \mathcal{U}$ to $\Lambda_h^m(\lambda_{1}[\phi],\ldots,\lambda_{m}[\phi])\in\mathbb{R}$ is real analytic for all $h \in \{1,\ldots, m\}$.
\end{itemize}
\end{theorem}

In the proof of Theorem \ref{main}, we simplify the term ``the space spanned by eigenfunctions corresponding to a set of eigenvalues'' by referring to it as the ``generalized eigenspace.''

\begin{proof} The proof of statement (i) heavily relies on a result from Kato's book \cite{Ka95}. We choose $\delta >0$ so that $[\lambda-\delta, \lambda+\delta] \cap \sigma(K^\ast_{\phi_0}) = \{\lambda\}$, and we let $\gamma$ be a simple closed path around $\lambda$  with positive orientation such that $\Gamma^i \cap \mathbb{R} = (\lambda-\delta, \lambda+\delta)$. Since $K^\ast_{\phi_0}$ is compact on $L^2(\partial\Omega)$, $0$ belongs to its spectrum, and thus $(\lambda-\delta, \lambda+\delta)$ cannot contain $0$. Therefore, the interval $(\lambda-\delta, \lambda+\delta)$ may only contain non-zero eigenvalues of $K^\ast_{\phi}$, whose eigenspaces consist of functions in $C^{0,\alpha}(\partial\Omega)$  by Lemma \ref{lem:reg}. Thus, in the proof, we can consider the operator $K^\ast_{\phi}$ acting on $C^{0,\alpha}(\partial\Omega)$ instead of $L^2(\partial\Omega)$. The advantage is that the map taking $\phi\in \mathcal{A}^{1,\alpha}_{\partial\Omega}$ to $K^\ast_\phi\in\mathcal{L}(C^{0,\alpha}(\partial\Omega))$ is real analytic, as stated in \cite[Theorem 3.2 (iii)]{DaLuMu22}, and therefore continuous.  Consequently, as $\phi$ tends to $\phi_0$, the graph distance between $K^\ast_\phi$ and $K^\ast_{\phi_0}$ tends to zero (cf.~\cite[Eq.~(2.27), p.~203]{Ka95}). Then, by \cite[Theorem 3.16, p.~212]{Ka95}, we deduce that for $\phi$ in a sufficiently small neighborhood $\mathcal{U}$ of $\phi_0$ in $(C^{1,\alpha}(\partial\Omega))^n$, the path $\gamma$ splits the spectrum $\sigma(K^\ast_\phi)$ into two parts, denoted $\sigma'(K^\ast_\phi)$ and $\sigma''(K^\ast_\phi)$, with $\sigma'(K^\ast_\phi)$ being the part of the spectrum enclosed by $\gamma$. Moreover, \cite[Theorem 3.16, p.~212]{Ka95} also guarantees that the generalized eigenspace $V'(K^\ast_\phi)$ corresponding to $\sigma'(K^\ast_\phi)$ is isomorphic to the generalized eigenspace $V'(K^\ast_{\phi_0})$ corresponding to $\sigma'(K^\ast_{\phi_0})$. Since we assume that $\lambda$ has multiplicity $m$, $V'(K^\ast_{\phi_0})$ has dimension $m$, and thus $V'(K^\ast_\phi)$ has dimension $m$ as well. That is, the generalized eigenspace corresponding to the eigenvalues surrounded by $\gamma$ has dimension $m$. Since all eigenvalues of $K^\ast_\phi$ are real, and we assume that $\Gamma^i \cap \mathbb{R} = (\lambda-\delta, \lambda+\delta)$, we conclude that, for $\phi\in\mathcal{U}$, $K^\ast_\phi$ has exactly $m$ eigenvalues in $(\lambda-\delta, \lambda+\delta)$, counting multiplicities. Hence, statement (i) is proved.

We now move on to statement (ii), which we prove using an argument of \cite{LaLa04}. To begin, consider a basis $\{\eta_1,\ldots,\eta_m\} \subseteq C^{0,\alpha}(\partial\Omega)$ for $V'(K^\ast_{\phi_0})$, assumed to be orthonormal with respect to the $L^2(\partial\Omega)$ inner product. We define the $(m \times m)$ real matrix $M_\phi$ as follows:
\[ M_{\phi}:= \Big(\langle P_\gamma[K^\ast_{\phi}]\eta_i,\eta_j\rangle_{L^2(\partial\Omega)}\Big)_{i,j=1,\dots,m}\,. \]
Due to the real analyticity of the Riesz projector as proven in Proposition \ref{prop:anR}, there exists an open neighborhood $\mathcal{U}_0$ of $\phi_0$ in $\mathcal{A}^{1,\alpha}_{\partial\Omega}$ such that the functions mapping $\phi\in \mathcal{U}_0 $ to the entries of $M_{\phi}$ are real analytic for all $i,j=1,\dots,m$. In particular, the function mapping $\phi\in\mathcal{U}_0$ to $\mathrm{det}\,M_\phi$ is continuous. Since $M_{\phi_0} = I_n$ by \eqref{Puu}, for $\phi$ in a sufficiently small neighborhood of $\phi_0$, $\mathrm{det}\,M_\phi\neq 0$, and the functions 
\[ P_\gamma[K^\ast_{\phi}]\eta_1,\ldots,\,P_\gamma[K^\ast_{\phi}]\eta_m \]
are linearly independent. By statement (i), we know that the space $V'(K^\ast_\phi)$ has dimension $m$ for $\phi\in \mathcal{U}$. Additionally, by Remark \ref{rem:kerP}, we know that the vectors $P_\gamma[K^\ast_{\phi}]\eta_j$ are contained in $V'(K^\ast_\phi)=P_\gamma[K^\ast_{\phi}]L^2(\partial\Omega)$. Therefore, by possibly shrinking the neighborhood $\mathcal{U}$ of $\phi_0$, $\{P_\gamma[K^\ast_{\phi}]\eta_j\}_{j=1,\dots,m}$ forms a basis for  $V'(K^\ast_\phi)$ for all $\phi\in \mathcal{U}$. 

Next, we use a Gram-Schmidt process to orthonormalize  $\{P_\gamma[K^\ast_{\phi}]\eta_j\}_{j=1,\dots,m}$ with respect to the $\langle \cdot,\cdot\rangle_{\sigma_n[\phi] S_\phi}$ inner product. This yields an orthonormal basis $\{\mu_1[\phi],\ldots,\mu_m[\phi]\}$ for $V'(K^\ast_\phi)$. Since the map that takes $\phi\in\mathcal{A}^{1,\alpha}_{\partial\Omega}$ to $\sigma_n[\phi]\,  S_\phi\in \mathcal{L}(C^{0,\alpha}(\partial\Omega))$ is real analytic (cf.~Lanza and Rossi \cite{LaRo04}  and \cite[Theorem 3.2 (i)]{DaLuMu22}), and given the real analyticity of the Riesz projector from Proposition \ref{prop:anR},  we can verify that the maps $\mathcal{U}\ni\phi\mapsto \mu_j[\phi]\in C^{0,\alpha}(\partial\Omega)$ are real analytic for all $j=1,\ldots,m$. With this, statement (ii) is proven.

Finally, let's address statement (iii). For $\phi\in \mathcal{U}$, we define the matrix
\[ {A}_\phi :=  \Big(\langle {K}^\ast_{\phi}\left[\mu_i[\phi]\right],\mu_j[\phi]\rangle_{\sigma_n[\phi] S_\phi}  \Big)_{i,j=1,\dots,m}, \]
which represents the operator $K^\ast_{\phi}$ on $V'(K^\ast_\phi)$ with respect to the basis $\{\mu_j[\phi]\}_{j=1,\dots,m}$. The matrix $A_\phi$ has the same eigenvalues as the operator $K^\ast_{\phi}$ restricted to $V'(K^\ast_\phi)$, i.e., the same eigenvalues that the operator $K^\ast_{\phi}$ has in $(\lambda-\delta,\lambda+\delta)$. Let's denote these eigenvalues as $\lambda_{1}[\phi]\le\lambda_2[\phi]\le\dots\le\lambda_{m}[\phi]$.

Using \eqref{secular}, we know that the symmetric functions $\Lambda^m_h(\lambda_{1}[\phi],\ldots,\lambda_{m}[\phi])$ are real analytic functions of the entries of the matrix $A_\phi$. Since the functions mapping $\phi\in \mathcal{U}$ to the entries of $A_\phi$ are real analytic, as we can verify using \cite[Theorem 3.2 (iii)]{DaLuMu22} and the real analyticity of the maps $\phi\mapsto\mu_j[\phi]$, we conclude that the symmetric functions $\Lambda^m_h(\lambda_{1}[\phi],\ldots,\lambda_{m}[\phi])$ depend real analytically on $\phi\in \mathcal{U}$. With this, statement (iii) has been proven.
\end{proof}

In Theorem \ref{main}, we could have considered the standard inner product induced by $L^2(\partial\Omega)$ instead of $\langle \cdot,\cdot\rangle_{\sigma_n[\phi] S_\phi}$. By doing so, we would have obtained an $L^2(\partial\Omega)$-orthonormal basis for the  space spanned by the eigenfunctions corresponding to $\lambda_{1}[\phi]\le\lambda_2[\phi]\le\dots\le\lambda_{m}[\phi]$. In the definition of the matrix $A_\phi$ introduced in the proof, we should have used the $L^2(\partial\Omega)$ inner product. However, the advantage of using $\langle \cdot,\cdot\rangle_{\sigma_n[\phi] S_\phi}$ lies in the fact that $K^*_\phi$ is self-adjoint with respect to this inner product, as we can see by computing: 
\begin{equation}\label{dueconti}
\begin{split}
\langle f, K^*_\phi[g]\rangle_{\sigma_n[\phi] S_\phi}&=\int_{\partial\Omega}S_\phi[f]\overline{K^*_\phi[g]} \sigma_n[\phi] \,d\sigma\\
&=\int_{\phi(\partial\Omega)} S_{\phi(\partial\Omega)}[f\circ\phi^{-1}]\,\overline{K^*_{\phi(\partial\Omega)}[g\circ\phi^{-1}]}\, d\sigma\\
&=\int_{\phi(\partial\Omega)} K_{\phi(\partial\Omega)}[S_{\phi(\partial\Omega)}[f\circ\phi^{-1}]]\,\overline{g\circ\phi^{-1}}\, d\sigma\\
&=\int_{\phi(\partial\Omega)} S_{\phi(\partial\Omega)}[K^*_{\phi(\partial\Omega)}[f\circ\phi^{-1}]]\,\overline{g\circ\phi^{-1}}\, d\sigma=\langle K^*_\phi[f],g \rangle_{\sigma_n[\phi] S_\phi}\,.
\end{split}
\end{equation}
The equality above will prove to be crucial in the subsequent sections. 

 In what follows, we will omit the conjugate symbol when it is not necessary, particularly when dealing with scalar products involving eigenfunctions of \(K^*_\phi\) or solutions to the plasmonic problem, which can always be chosen to be real.

We also note that the argument used to prove statement (i) of Theorem \ref{main} also establishes the continuity of the eigenvalues $\lambda_1[\phi]$, $\dots$, $\lambda_m[\phi]$ at $\phi_0$ (see also Kato \cite[IV \S3.5]{Ka95}). Specifically, we have the following proposition:

\begin{proposition}\label{continuity} Under the assumptions and notation of Theorem \ref{main}, the eigenvalues $\lambda_1[\phi]$, $\dots$, $\lambda_m[\phi]$ converge to $\lambda$ as $\|\phi-\phi_0\|_{(C^{1,\alpha}(\partial\Omega))^n}$ tends to zero.
\end{proposition} 
\begin{proof} Let $\delta$ be as defined in statement (i) of Theorem \ref{main}. Through the proof of statement (i) of Theorem \ref{main}, it follows that for all $0<\delta'<\delta$, there exists a neighborhood $\mathcal{U}'\subseteq\mathcal{U}$ of $\phi_0$ such that $\lambda_1[\phi]$, $\dots$, $\lambda_m[\phi]$ are contained within $(\lambda-\delta',\lambda+\delta')$ for all $\phi\in \mathcal{U}'$. Thus, the statement follows.
\end{proof}

\section{First derivative of the symmetric functions of the NP-eigenvalues}\label{derivatives}

Now that we know from Theorem \ref{main} that the map \(\phi \mapsto \Lambda_h^m(\lambda_{1}[\phi], \ldots, \lambda_{m}[\phi])\) is real analytic and thus differentiable, our next goal is to derive an explicit expression for the differential
\begin{equation}\label{dLambda.eq0}
d_\phi \Lambda_h^m(\lambda_{1}[\phi_0], \ldots, \lambda_{m}[\phi_0]),
\end{equation}
which will be done by computing the differential of the matrix $A_\phi$
from the proof of Theorem \ref{main}, and then applying equality \eqref{secular}, which relates the symmetric functions to the characteristic polynomial of \(A_\phi\). The formula for the differential of \(A_\phi\) will be given in Theorem \ref{dAphi}, and the expression for \eqref{dLambda.eq0} will be provided in Theorem \ref{dLambda}.

In the following sections, we explain the steps leading to Theorems \ref{dAphi} and \ref{dLambda}. First, in Section \ref{LtoK}, we show that the task of computing the \(\phi\)-differential of \(A_\phi\) can be reduced to finding the differential \(d_\phi K^*_{\phi_0}\) of the operator \(K^*_{\phi}\). Then, in Section \ref{fromK*toK}, we demonstrate that we can alternatively use the differential \(d_\phi K_{\phi_0}\) of the adjoint operator \(K_{\phi}\). This approach is advantageous because \(d_\phi K_{\phi_0}[c] = 0\) when \(c\) is a constant function. Thus, for a fixed \(x \in \partial \Omega\), we can replace a density function \(\eta\) with \( \eta - \eta(x) \), yielding \(d_\phi K_{\phi_0}[\eta](x) = d_\phi K_{\phi_0}[\eta - \eta(x)](x)\). This allows us to take advantage of the extra integrability provided by the modified density \( \eta - \eta(x) \) to derive an explicit expression for \(d_\phi K_{\phi_0}[\eta]\), which we do in Sections \ref{undersign}, where we justify differentiating under the integral sign, and \ref{dK}, where we manipulate the derivative expression into a suitable form. Finally, in Section \ref{backtoL}, we return to \(d_\phi A_{\phi_0}\) and use the previous results to derive an explicit formula, which we then apply to obtain the desired expression for \(d_\phi \Lambda_h^m(\lambda_{1}[\phi_0], \ldots, \lambda_{m}[\phi_0])\).

We emphasize that in Theorems \ref{dAphi} and \ref{dLambda}, we restrict our focus to the case where \(\lambda \neq 1/2\), as this allows for a simpler expression for the derivatives of \(A_\phi\) and \(\Lambda_h^m(\lambda_{1}[\phi], \ldots, \lambda_{m}[\phi])\) in terms of plasmonic eigenfunctions. On the other hand, the case where \(\lambda = 1/2\) is trivial, since \(1/2\) is always a simple NP-eigenvalue, regardless of the specific domain \(\Omega\), and its shape derivative is therefore 0.

\subsection{From $d_\phi \Lambda^m_h(\lambda_1[\phi_0],\ldots,\lambda_m[\phi_0])$ to $d_\phi K^*_{\phi_0}$}\label{LtoK}

As mentioned above, the first step is to show that computing the shape derivative of the symmetric functions \(\Lambda^m_h(\lambda_1[\phi], \ldots, \lambda_m[\phi])\) can be reduced to determining the differential \(d_\phi K^*_{\phi}\) of \(K^*_{\phi}\).  If $\phi\mapsto F_\phi$ (or $\phi\mapsto F[\phi]$) is a differentiable map from $\mathcal{A}^{1,\alpha}_{\partial\Omega}$ to a Banach space $X$, we denote by 
\[
d_\phi F_{\phi_0}.\theta \quad\text{(or $d_\phi F[\phi_0].\theta$)}
\]
the differential of $F$ at the point $\phi_0\in \mathcal{A}^{1,\alpha}_{\partial\Omega}$ applied to $\theta\in (C^{1,\alpha}(\partial\Omega))^n$. 

Let $\{\mu_1[\phi],\ldots,\mu_m[\phi]\}$ be the $(\langle \cdot,\cdot\rangle_{\sigma_n[\phi] S_\phi})$-orthonormal basis of Theorem \ref{main}, and let  
\begin{equation}\label{Aphi} 
{A}_\phi :=  \Big(\langle {K}^\ast_{\phi}\left[\mu_i[\phi]\right],\mu_j[\phi]\rangle_{\sigma_n[\phi] S_\phi}  \Big)_{i,j=1,\dots,m}
\end{equation}
be the $(m\times m)$ real matrix representing ${K}^\ast_{\phi}$ on the space spanned by ${\mu_1[\phi],\ldots,\mu_m[\phi]}$ (as described in the proof of Theorem \ref{main}). In the following Proposition \ref{der}, we observe that the $\phi$ differential of the matrix $A_\phi$ is obtained by taking the differential of ${K}^\ast_{\phi}$ in the definition of $A_\phi$, a result reminiscent of the Hellmann-Feynman theorem. In what follows, we employ the symbol $\mathbb{M}_m$ to denote the space of $(m\times m)$ real matrices.
\begin{proposition}\label{der} Under the assumptions and notations of Theorem \ref{main}, the differential of  
\[
\mathcal{U}\ni\phi\mapsto A_\phi\in \mathbb{M}_m
\] 
at $\phi_0$ is given by
\[
d_\phi A_{\phi_0}.\theta=\biggl(\langle \mu_i[\phi_0],d_\phi(K^\ast_{\phi_0})[\mu_j[\phi_0]].\theta\rangle_{\sigma_n[\phi_0] S_{\phi_0}}\biggr)_{i,j=1,\dots,m}
\] 
for all $\theta \in (C^{1,\alpha}(\partial\Omega))^n$, where $d_\phi(K^\ast_{\phi_0})[\mu_j[\phi_0]]$ represents the differential of 
\[
\mathcal{A}^{1,\alpha}_{\partial\Omega}\ni\phi\mapsto K^\ast_{\phi}[\mu_j[\phi_0]]\in C^{0,\alpha}(\partial\Omega)
\] 
at $\phi_0$.
\end{proposition}
\begin{proof} Since $\{\mu_1[\phi], \ldots, \mu_m[\phi]\}$ is orthonormal for the inner product $\langle \cdot,\cdot\rangle_{\sigma_n[\phi] S_\phi}$, we have $\langle \mu_i[\phi],\mu_j[\phi]\rangle_{\sigma_n[\phi] S_\phi}=0$ if $i\neq j$, and $\langle \mu_i[\phi],\mu_j[\phi]\rangle_{\sigma_n[\phi] S_\phi}=1$ if $i=j$, for all $\phi\in \mathcal{U}$. It follows that
\[
d_\phi\left(\langle \mu_i[\phi_0],\mu_j[\phi_0]\rangle_{\sigma_n[\phi_0] S_{\phi_0}}\right)=0
\]
for all $i,j=1,\ldots, m$. By the definition of the $\langle \cdot,\cdot\rangle_{\sigma_n[\phi] S_\phi}$ product, we have 
\[
\langle \mu_i[\phi],\mu_j[\phi]\rangle_{\sigma_n[\phi] S_{\phi}}=\langle \sigma_n[\phi]\,  S_{\phi}[\mu_i[\phi]],\mu_j[\phi]\rangle_{L^2(\partial\Omega)}\,,
\] 
and upon computing the differential of the right-hand side, we obtain
\begin{equation}\label{der.eq1}
\begin{split}
&\langle d_\phi(\sigma_n[\phi_0])\,  S_{\phi_0}[\mu_i[\phi_0]],\mu_j[\phi_0]\rangle_{L^2(\partial\Omega)}+\langle \sigma_n[\phi_0]\, d_\phi( S_{\phi_0})[\mu_i[\phi_0]],\mu_j[\phi_0]\rangle_{L^2(\partial\Omega)}\\
&+\langle \sigma_n[\phi_0]\,  S_{\phi_0}[d_\phi(\mu_i[\phi_0])],\mu_j[\phi_0]\rangle_{L^2(\partial\Omega)}+\langle \sigma_n[\phi_0]\,  S_{\phi_0}[\mu_i[\phi_0]],d_\phi(\mu_j[\phi_0])\rangle_{L^2(\partial\Omega)}=0\,.
\end{split}
\end{equation}
Here we have used the fact that $\mathcal{A}^{1,\alpha}_{\partial\Omega}\ni \phi\mapsto \sigma_n[\phi]\in C^{0,\alpha}(\partial\Omega)$ is differentiable (it is ideed real analytic, see \cite[Proposition 3.13]{LaRo04})

This equality will be useful in the following computation. By \eqref{dueconti}, we have 
\[
\langle K^\ast_{\phi}[\mu_i[\phi]],\mu_j[\phi]\rangle_{\sigma_n[\phi] S_{\phi}}=\langle \sigma_n[\phi]\,  S_{\phi}[\mu_i[\phi]],K^\ast_{\phi}[\mu_j[\phi]]\rangle_{L^2(\partial\Omega)}
\]
and differentiating the right-hand side, we obtain
\begin{equation}\label{der.eq2}
\begin{split}
&\langle d_\phi(\sigma_n[\phi_0])\,  S_{\phi_0}[\mu_i[\phi_0]],K^\ast_{\phi_0}[\mu_j[\phi_0]]\rangle_{L^2(\partial\Omega)}\\
&+\langle \sigma_n[\phi_0]\, d_\phi( S_{\phi_0})[\mu_i[\phi_0]],K^\ast_{\phi_0}[\mu_j[\phi_0]]\rangle_{L^2(\partial\Omega)}\\
&+\langle \sigma_n[\phi_0]\,  S_{\phi_0}[d_\phi(\mu_i[\phi_0])],K^\ast_{\phi_0}[\mu_j[\phi_0]\rangle_{L^2(\partial\Omega)}\\
&+\langle \sigma_n[\phi_0]\,  S_{\phi_0}[\mu_i[\phi_0]],d_\phi(K^\ast_{\phi_0})[\mu_j[\phi_0]]\rangle_{L^2(\partial\Omega)}\\
&+\langle \sigma_n[\phi_0]\,  S_{\phi_0}[\mu_i[\phi_0]],K^\ast_{\phi_0}[d_\phi(\mu_j[\phi_0])]\rangle_{L^2(\partial\Omega)}\,.
\end{split}
\end{equation}
The last term in the sum can be expressed as 
\[
\langle \mu_i[\phi_0],K^\ast_{\phi_0}[d_\phi(\mu_j[\phi_0])]\rangle_{\sigma_n[\phi_0] S_{\phi_0}}
\]
and using again \eqref{dueconti}, this becomes 
\[
\langle K^\ast_{\phi_0}[\mu_i[\phi_0]], d_\phi(\mu_j[\phi_0])\rangle_{\sigma_n[\phi_0] S_{\phi_0}}=\langle \sigma_n[\phi_0]\,  S_{\phi_0}[K^\ast_{\phi_0}[\mu_i[\phi_0]]],d_\phi(\mu_j[\phi_0])\rangle_{L^2(\partial\Omega)}\,.
\]
Recalling that $\mu_i[\phi_0]$ and $\mu_j[\phi_0]$ are eigenfunctions of $K^\ast_{\phi_0}$ with eigenvalue $\lambda$, we find that the expression in \eqref{der.eq2} is equal to
\[
\begin{split}
&\lambda\langle d_\phi(\sigma_n[\phi_0])\,  S_{\phi_0}[\mu_i[\phi_0]],\mu_j[\phi_0]\rangle_{L^2(\partial\Omega)}\\
&+\lambda\langle \sigma_n[\phi_0]\, d_\phi( S_{\phi_0})[\mu_i[\phi_0]],\mu_j[\phi_0]\rangle_{L^2(\partial\Omega)}\\
&+\lambda\langle \sigma_n[\phi_0]\,  S_{\phi_0}[d_\phi(\mu_i[\phi_0])],\mu_j[\phi_0]\rangle_{L^2(\partial\Omega)}\\
&+\langle \sigma_n[\phi_0]\,  S_{\phi_0}[\mu_i[\phi_0]],d_\phi(K^\ast_{\phi_0})[\mu_j[\phi_0]]\rangle_{L^2(\partial\Omega)}\\
&+\lambda\langle \sigma_n[\phi_0]\,  S_{\phi_0}[\mu_i[\phi_0]],d_\phi(\mu_j[\phi_0])\rangle_{L^2(\partial\Omega)}\,.
\end{split}
\]
By \eqref{der.eq1}, the sum of the terms multiplied by $\lambda$ cancels out, leaving us with 
\[
\langle \sigma_n[\phi_0]\,  S_{\phi_0}[\mu_i[\phi_0]],d_\phi(K^\ast_{\phi_0})[\mu_j[\phi_0]]\rangle_{L^2(\partial\Omega)}\,,
\] 
which can be rewritten as
\[
\langle \mu_i[\phi_0],d_\phi(K^\ast_{\phi_0})[\mu_j[\phi_0]]\rangle_{\sigma_n[\phi_0] S_{\phi_0}}\,.
\]
The proposition is thus proved.
\end{proof}

By equality \eqref{secular} and Proposition \ref{der}, the task of computing the derivatives of the symmetric functions of the eigenvalues boils down to determining the derivative of $K^\ast_{\phi}$. For instance, the symmetric function
\[
\Lambda^m_1(\lambda_1[\phi],\ldots,\lambda_m[\phi])=\lambda_1[\phi]+\dots+\lambda_m[\phi]
\]
is equivalent to the trace of the matrix $A_\phi$, yielding
\[
d_\phi\left(\Lambda^m_1(\lambda_1[\phi_0],\ldots,\lambda_m[\phi_0])\right)=\sum_{i=1}^m\langle \mu_i[\phi_0],d_\phi(K^\ast_{\phi_0})[\mu_i[\phi_0]]\rangle_{\sigma_n[\phi_0] S_{\phi_0}}\,.
\]
In Theorem \ref{dLambda}, we will provide an explicit expression for the first derivative of all symmetric functions $\Lambda^m_h(\lambda_1[\phi_0],\ldots,\lambda_m[\phi_0])$. Instead of directly computing the derivative of \(K^*_\phi\), we take a different approach by using the derivative of \(K_\phi\), the adjoint operator to \(K^*_\phi\). In the following section, we explain why this substitution is possible.

\subsection{From $d_\phi K^*_{\phi_0}$ to $d_\phi K_{\phi_0}$}\label{fromK*toK}

Without loss of generality, we set $ \phi_0=I $ (the identity mapping on $ \partial\Omega $), and for brevity, we write
\[
\mu_1\,,\ldots,\mu_m\,,
\]
instead of 
\[
\mu_1[\phi_0]\,,\ldots,\mu_m[\phi_0]\,.
\]    
Since for $ \phi_0=I $ we have $ \sigma_n[\phi_0]=1 $ and $ S_{\phi_0}= S_{\partial\Omega} $, the expression of Proposition \ref{der} for the differential of $ A_\phi $ can be written as
\[
d_\phi A_{\phi_0}.\theta=\left(\int_{\partial\Omega}  S_{\partial\Omega}[\mu_i]\left(d_\phi K^*_{\phi_0}[\mu_j].\theta\right)d\sigma\right)_{i,j=1,\ldots,m}
\]
for all $ \theta\in (C^{1,\alpha}(\partial\Omega))^n $.  It appears that the remaining task is to compute an explicit formula for $ d_\phi K^*_{\phi_0}[\mu_j].\theta $. However, it is actually more convenient to compute the derivative of the adjoint operator $ K_{\phi} $: 
\[
d_\phi(K_{\phi_0})[\mu_j].\theta.
\]
We will discuss the advantage of using $ K_{\phi_0} $ in the forthcoming Section \ref{undersign}. By the following lemma, we see that this will do the trick.

\begin{proposition}\label{switcheroo}
With the assumptions and notations of Theorem \ref{main} and $\phi_0=I$, we have
\[
\int_{\partial \Omega} S_{\partial \Omega}[\mu_i]\left(d_\phi(K^\ast_{\phi_0})[\mu_j].\theta\right) d\sigma = \int_{\partial \Omega}\left(d_\phi(K_{\phi_0})[S_{\partial \Omega}[\mu_i]].\theta\right)\mu_j \, d\sigma 
\]
for all $\theta\in (C^{1,\alpha}(\partial\Omega))^n$.
\end{proposition}
\begin{proof}
First we show that, for all $\phi\in\mathcal{A}^{1,\alpha}_{\partial\Omega}$, $\mu \in C^{0,\alpha}(\partial \Omega)$ and $\eta \in C^{1,\alpha}(\partial \Omega)$, we have
\[
\begin{split}
&\int_{\partial \Omega} K^\ast_{\phi}[\mu]\eta \sigma_n[\phi] \, d\sigma=\int_{\phi(\partial \Omega)} K^\ast_{\phi(\partial \Omega)}[\mu\circ \phi^{-1}]\eta\circ \phi^{-1} \, d\sigma\\
&\qquad\qquad=\int_{\phi(\partial \Omega)} \mu\circ \phi^{-1} K_{\phi(\partial \Omega)}[ \eta\circ \phi^{-1}] \, d\sigma=\int_{\partial \Omega} \mu K_{\phi}[ \eta ]\sigma_n[\phi] \, d\sigma\, .
\end{split}
\]
Thus,
\[
d_\phi \left(\int_{\partial \Omega} K^\ast_{\phi}[\mu]\eta \sigma_n[\phi] \, d\sigma\right)_{|\phi=\phi_0} =d_\phi \left(\int_{\partial \Omega} \mu K_{\phi}[ \eta ] \sigma_n[\phi] \, d\sigma\right)_{|\phi=\phi_0} ,
\]
and consequently,
\[
\begin{split}
 & \int_{\partial \Omega}\left(d_\phi(K^*_{\phi_0})[\mu].\theta\right)\eta\, \sigma_n[\phi_0] + K^\ast_{\phi_0}[\mu]\eta \left(d_\phi\sigma_n[\phi_0].\theta\right)d\sigma\\
 &\qquad = \int_{\partial \Omega} \mu \left(d_\phi(K_{\phi_0})[ \eta ].\theta\right) \sigma_n[\phi_0]+ \mu K_{\phi_0}[ \eta ]  \left(d_\phi\sigma_n[\phi_0].\theta \right) d\sigma
\end{split}
\]
for all $\theta\in (C^{1,\alpha}(\partial\Omega))^n$. Here we have also used the fact that $\mathcal{A}^{1,\alpha}_{\partial\Omega}\ni \phi\mapsto \sigma_n[\phi]\in C^{0,\alpha}(\partial\Omega)$ is differentiable. Now, for $\phi_0=I$, we have $\sigma_n[\phi_0]=1$, $K^\ast_{\phi_0}=K^\ast_{\partial\Omega}$, and $K_{\phi_0}=K_{\partial\Omega}$, and the last equality yields
\begin{equation}\label{switcheroo.eq1}
\begin{split}
& \int_{\partial \Omega}\left(d_\phi(K^*_{\phi_0})[\mu].\theta\right)\eta \,d\sigma + \int_{\partial \Omega}K^\ast_{\partial\Omega}[\mu]\eta \left(d_\phi\sigma_n[\phi_0].\theta \right)  d\sigma\\
 &\qquad = \int_{\partial \Omega} \mu \left(d_\phi(K_{\phi_0})[ \eta ].\theta\right)d\sigma + \int_{\partial \Omega}\mu K_{\partial\Omega}[ \eta ]  \left(d_\phi\sigma_n[\phi_0].\theta \right) d\sigma \, .
\end{split}
\end{equation}
In general, $d_\phi\sigma_n[\phi_0].\theta$ is not constant (we shall see that this derivative is equal to the tangential divergence of $\theta$, cf.~\eqref{divT}). Therefore, the fact that $K^\ast_{\partial \Omega}$ and $K_{\partial \Omega}$ are adjoint does not guarantee that 
\[
\int_{\partial \Omega}K^\ast_{\partial\Omega}[\mu]\eta \left(d_\phi\sigma_n[\phi_0].\theta \right)  d\sigma = \int_{\partial \Omega}\mu K_{\partial\Omega}[ \eta ]  \left(d_\phi\sigma_n[\phi_0].\theta \right) d\sigma\, .
\]
 However, the last equality holds true when $\eta = S_{\partial \Omega}[\mu_i]$ and $\mu =\mu_j$. Indeed, since $K^\ast_{\partial \Omega}[\mu_i]=\lambda\mu_i$ and $K^\ast_{\partial \Omega}[\mu_j]=\lambda\mu_j$, we have
\[
K^\ast_{\partial \Omega}[\mu_j] S_{\partial \Omega}[\mu_i]=\lambda\mu_j S_{\partial \Omega}[\mu_i]=\mu_j  S_{\partial \Omega}[\lambda \mu_i]=\mu_j  S_{\partial \Omega}[K^\ast_{\partial\Omega}[\mu_i]]=\mu_j K_{\partial \Omega}[ S_{\partial\Omega}[\mu_i]]\, .
\]
Then, taking $\eta = S_{\partial \Omega}[\mu_i]$ and $\mu =\mu_j$ in \eqref{switcheroo.eq1} we can cancel out the second terms in the left and right-hand sides, and we obtain the desired equality in the proposition.
\end{proof}

\subsection{Taking the derivative under the integral sign}\label{undersign}
So, we are left with the task of computing the differential $d_\phi K_{\phi_0}[\eta].\theta$, with $\eta\in C^{1,\alpha}(\partial\Omega)$ and $\theta\in(C^{1,\alpha}(\partial\Omega))^n$. A classical argument shows that the function $\phi_t:=I+t\theta$ belongs to $\mathcal{A}^{1,\alpha}_{\partial\Omega}$ for $t$ sufficiently small (see, e.g., \cite[\S 5.2.2]{HePi18}). Then, we can write
\[
d_\phi K_{\phi_0}[\eta].\theta=\frac{d}{dt}K_{I+t\theta}[\eta]_{|t=0}.
\]
The advantage of handling $K_\phi$ instead of $K^*_\phi$ lies in the fact that, if the density function $\eta$ is identically equal to $1$ on $\partial\Omega$, then, by the third Green's identity, we have
\[
 D_{\phi_t(\partial\Omega)}[1](\xi)=1
\]
for all $\xi$ in the bounded domain $\Omega[\phi_t]$ with boundary $\phi_t(\partial\Omega)$ (for the third Green's identity refer, for example, to \cite[Corollary 4.6]{DaLaMu21}). By the jump formulas \eqref{jump} for the double-layer potential, it follows that
\[
K_{\phi_t(\partial\Omega)}[1](\xi)=\frac{1}{2}\quad\text{for all $\xi\in \phi_t(\partial\Omega)$,}
\]
and, setting $\xi=\phi_t(x)$, we deduce that
\[
K_{I+t\theta}[1](x)=K_{\phi_t(\partial\Omega)}[1]\circ\phi_t(x)=\frac{1}{2}\quad\text{for all $x\in \partial\Omega$.}
\]
Thus, for a fixed $x\in \partial\Omega$, we have
\[
K_{I+t\theta}[\eta(x)](x)=\frac{1}{2}\eta(x)
\]
and, consequently,
\[
\frac{d}{dt}\left(K_{I+t\theta}[\eta(x)]\right)_{|t=0}(x)=0\,.
\]
By linearity, it follows that
\begin{equation}\label{trattore}
\frac{d}{dt}\left(K_{I+t\theta}[\eta]\right)_{|t=0}(x)=\frac{d}{dt}\left(K_{I+t\theta}[\eta-\eta(x)]\right)_{|t=0}(x)\,.
\end{equation}
The right-hand side of \eqref{trattore} reads
\begin{equation}\label{ruspa}
\frac{d}{dt}\left(-\int_{\partial\Omega}(\eta(y)-\eta(x))\frac{(\phi_t(x)-\phi_t(y))\cdot{\nu_{\Omega[\phi_t]}(\phi_t(y))}}{s_n|\phi_t(x)-\phi_t(y)|^{n}}\sigma_n[\phi_t](y)\,d\sigma_y\right)_{|t=0}
\end{equation}
and we now aim to demonstrate that we can differentiate under the integral sign. The difference $(\eta(y)-\eta(x))$ will facilitate this task by providing a higher degree of integrability of the integrand function, an advantage that could not be leveraged using $K^*_\phi$.

Before proceeding to prove that we can differentiate under the integral sign and compute the derivative of $K_{I+t\theta}$, we introduce some notation. It will be convenient to consider the density $\eta$ in the space $C^{1,\alpha}_0(\mathbb{R}^n)$ of functions with compact support, rather than $C^{1,\alpha}(\partial\Omega)$. This substitution does not sacrifice generality because, as is well known, there exists a bounded extension operator from $C^{1,\alpha}(\partial\Omega)$ to $C^{1,\alpha}_0(\mathbb{B}_n(0,R))$, with $R>0$ such that $\overline\Omega\subset\mathbb{B}_n(0,R)$, and then we can extend by zero from $C^{1,\alpha}_0(\mathbb{B}_n(0,R))$ to $C^{1,\alpha}_0(\mathbb{R}^n)$ (see, e.g., Gilbarg and Trudinger \cite[Lemma 6.38]{GiTr83}). For $x\in\partial\Omega$, the tangential gradient $\nabla_{\partial\Omega}\eta(x)$ is defined by 
\begin{equation}\label{tangrad}
\nabla_{\partial\Omega}\eta(x):=\nabla \eta(x)-\nu_{\partial\Omega}(x)\otimes\nu_{\partial\Omega}(x)\nabla \eta(x)\,,
\end{equation}
where $\otimes$ denotes the tensor product. That is, if $v$ and $w$ are vectors of $\mathbb{R}^n$, then $v\otimes w$ is the $(n\times n)$-matrix defined as
\[
v\otimes w:=(v_i w_j)_{i,j=1,\dots,n}\,.
\]
It is well known that $\nabla_{\partial\Omega}\eta$ only depends on the restriction of $\eta$ to $\partial\Omega$.

Similarly, we will consider $\theta$ in $(C^{1,\alpha}_0(\mathbb{R}^n))^n$ rather than $(C^{1,\alpha}(\partial\Omega))^n$. Then, for sufficiently small $t$, we know that the map $I+t\theta$ is a $C^{1,\alpha}$ diffeomorphism from $\overline\Omega$ to its image $(I+t\theta)(\overline\Omega)$. We define:
\[
\Omega_t:=(I+t\theta)(\Omega)
\]
and
\[
x_t:=x+t\theta(x)\quad\text{for all $x\in\partial\Omega$ (or $x\in\overline\Omega$).}
\]
Then, we set:
\[
\theta_t(x_t):=\theta((I+t\theta)^{-1}(x_t))\left(=\theta(x)\right)\quad\text{for all $x_t\in\partial\Omega_t$ (or $x_t\in\overline\Omega_t$)}
\]
and observe that:
\begin{equation}\label{gradthetat}
\nabla\theta_t(x_t)=(1_n +t\nabla\theta(x))^{-1}\nabla\theta(x)\quad\text{for all $x_t\in\overline\Omega_t$,}
\end{equation}
where $1_n$ denotes the $n\times n$ identity matrix, and where we understand that $\nabla\theta$ is  the transpose of the Jacobian of $\theta$, i.e., $\nabla\theta:=(\partial_{x_i}\theta_j)_{i,j=1,\ldots,n}$ (and similarly for $\nabla\theta_t(x_t)$). For $x\in\partial\Omega$, the tangential  gradient matrix $\nabla_{\partial\Omega}\theta(x)$ is defined as:
\[
\nabla_{\partial\Omega}\theta(x):=\nabla \theta(x)-\nu_{\partial\Omega}(x)\otimes\nu_{\partial\Omega}(x)\nabla \theta(x)\,.
\]
We also set
\[
\mathrm{div}_{\partial\Omega}\,\theta:=\mathrm{trace}\,(\nabla_{\partial\Omega}\theta)\, .
\]

Regarding the outer unit normal $\nu_t$ to $\partial\Omega_t$, we have the following lemma:

\begin{lemma}\label{dtnu} We have
\[
\nu_t(y_t):=\nu_{(I+t\theta)(\Omega)}(y+t\theta(y))=\frac{(1_n +t\nabla \theta (y))^{-1}\nu_\Omega(y)}{|(1_n +t\nabla \theta (y))^{-1}\nu_\Omega(y)|}
\]
and
\[
\frac{d}{dt}\left(\nu_t(y_t)\right)=-(\nabla_{\partial\Omega_{t}}\theta_{t})(y_t)\nu_{t}(y_t)\quad\text{for all }y\in\partial\Omega\,.
\]
\end{lemma}
\begin{proof} For the first equality we refer to \cite[Lemma 3.3]{LaRo04}. For the second equality, we observe that 
\[
\frac{d}{dt}(1_n +t\nabla\theta)^{-1}=-(1_n +t\nabla\theta)^{-1}(\nabla\theta)(1_n +t\nabla\theta)^{-1}\,,
\]
and then we compute
\[
\begin{split}
&\frac{d}{dt}(\nu_t(y_t))=-(1_n +t\nabla\theta)^{-1}(\nabla\theta)\frac{(1_n +t\nabla\theta)^{-1}\nu_\Omega}{|(1_n +t\nabla\theta)^{-1}\nu_\Omega|}\\
&+\frac{(1_n +t\nabla\theta)^{-1}\nu_\Omega}{|(1_n +t\nabla\theta)^{-1}\nu_\Omega|^2}\left(\frac{(1_n +t\nabla\theta)^{-1}\nu_\Omega}{|(1_n +t\nabla\theta)^{-1}\nu_\Omega|}\cdot(1_n +t\nabla\theta)^{-1}(\nabla\theta)(1_n +t\nabla\theta)^{-1}\nu_\Omega\right)\\
&=-\biggl((1_n +t\nabla\theta)^{-1}(\nabla\theta)-\nu_t\otimes\nu_t\,(1_n+t\nabla\theta)^{-1}(\nabla\theta)\biggr)\nu_t\\
&=-\biggl(\nabla\theta_t-\nu_t\otimes\nu_t\,\nabla\theta_t\biggr)\nu_t=-(\nabla_{\partial\Omega_t}\theta_t)\nu_t
\end{split}
\]
(see also \eqref{gradthetat}).
\end{proof}

Incidentally, we observe that $\left(\frac{d}{dt}\nu_t\right)_{|t=0}$ depends solely on the restriction of $\theta$ to $\partial\Omega$. We are now ready to prove the following proposition:

\begin{proposition}\label{DthetaK} Let $\theta\in (C^{1,\alpha}(\partial\Omega))^n$ and $\eta\in C^{1,\alpha}(\partial\Omega)$. We have
\[
\frac{d}{dt}\left(K_{I+t\theta}[\eta]\right)_{|t=0}(x)=-\int_{\partial\Omega}(\eta(y)-\eta(x))\frac{d}{dt}\left(\frac{(x_t-y_t)\cdot{\nu_t(y_t)}}{s_n|x_t-y_t|^{n}}\sigma_n[I+t\theta](y)\right)_{|t=0}\,d\sigma_y\\
\]
for all $x\in\partial\Omega$.
\end{proposition}
\begin{proof} 
To prove the proposition, we need to verify that we can apply the classical theorem on differentiation under the integral sign to the expression in \eqref{ruspa}. Accordingly, we have to verify that, for $t$ in a neighborhood of zero, both the functions 
\[
\partial\Omega\ni y\mapsto (\eta(y)-\eta(x))\frac{(x_t-y_t)\cdot{\nu_t(y_t)}}{s_n|x_t-y_t|^{n}}\sigma_n[I+t\theta](y)\in\mathbb{R}\,,
\]
and 
\begin{equation}\label{DthetaK.eq1}
\partial\Omega\ni y\mapsto \frac{d}{dt}\left[(\eta(y)-\eta(x))\frac{(x_t-y_t)\cdot{\nu_t(y_t)}}{s_n|x_t-y_t|^{n}}\sigma_n[I+t\theta](y)\right]\in\mathbb{R}
\end{equation}
are dominated by integrable functions in $L^1(\partial\Omega)$ that are independent of $t$ (cf.~Folland \cite[Theorem 2.27]{Fo95}). Notice that here we keep $x$ fixed, and we only consider the dependence on the variable $y$ and the parameter $t$. 

As mentioned in the comments before \eqref{tangrad} and \eqref{gradthetat}, we consider extensions of $\eta$ to $C^{1,\alpha}_0(\mathbb{R}^n)$ and of $\theta$ to $(C^{1,\alpha}_0(\mathbb{R}^n))^n$, which we still denote by $\eta$ and $\theta$, respectively. Then, by the mean value theorem, we have 
\[
 |\eta(y)-\eta(x)|\le \|\nabla\eta\|_\infty|x-y|
\]
and
\[
|\theta(x)-\theta(y)|\le \|\nabla\theta\|_\infty|x-y|
\]
for all $y\in\partial\Omega$, where we understand that $\|\nabla\eta\|_\infty:=\sup_{\xi\in\mathbb{R}^n}|\nabla\eta(\xi)|$  and  $\|\nabla\theta\|_\infty:=\sup_{\xi\in\mathbb{R}^n}(\sum_{j=1}^n|\nabla\theta_j(\xi)|^2)^{1/2}$. Therefore, we compute
\[
\begin{split}
\left|x_t-y_t\right|&=|(x-y)+t(\theta(x)-\theta(y))|\\
&\le|x-y|+|t||\theta(x)-\theta(y)|\le (1+\|\nabla\theta\|_\infty)|x-y|
\end{split}
\]
for all $y\in\partial\Omega$ and $t\in(-1,1)$, which implies that 
\[
\left|(x_t-y_t)\cdot{\nu_t(y_t)}\right|\le (1+\|\nabla\theta\|_\infty)|x-y|
\]
for all $y\in\partial\Omega$ and $t\in(-1,1)$. Moreover, we have
\begin{equation}\label{DthetaK.eq1.1}
|x_t-y_t|\ge |x-y|-|t|\,|\theta(x)-\theta(y)|\ge (1-|t|\|\nabla\theta\|_\infty)|x-y|\ge \frac{1}{2}|x-y|
\end{equation}
for $|t|\le (1/2)\|\nabla\theta\|_\infty^{-1}$. Since the map that takes $t$ in a neighborhood of zero to $\sigma_n[I+t\theta]$ is continuous (it is even real analytic, see \cite[Proposition 3.13]{LaRo04}), we conclude that 
\begin{equation}\label{DthetaK.eq2}
\left|(\eta(y)-\eta(x))\frac{(x_t-y_t)\cdot{\nu_t(y_t)}}{s_n|x_t-y_t|^{n}}\sigma_n[I+t\theta](y)\right|\le C{|x-y|^{2-n}}
\end{equation} 
for some $C>0$, uniformly for $t$ in a neighborhood of zero. 

To deal with the function in \eqref{DthetaK.eq1}, we use Lemma \ref{dtnu} to compute 
\[
\begin{split}
\frac{d}{dt}\bigl((x_t-y_t)\cdot \nu_t(y_t)\bigr)&=\frac{d}{dt}\bigl\{\bigl(x-y+t(\theta(x)-\theta(y))\bigr)\cdot\nu_t(y_t)\bigr\}\\
&=(\theta(x)-\theta(y))\cdot\nu_t(y_t)-(x_t-y_t)\cdot(\nabla_{\partial\Omega_t}\theta_t(y_t))\nu_t(y_t)\,.
\end{split}
\]
Using again inequalities $|\theta(x)-\theta(y)|\le \|\nabla\theta\|_\infty|x-y|$ and $|x_t-y_t|\le (1+\|\nabla\theta\|_\infty)|x-y|$, we verify that 
\begin{equation}\label{DthetaK.eq3}
\left|\frac{d}{dt}\bigl((x_t-y_t)\cdot \nu_t(y_t)\bigr)\right|\le 2+\|\nabla\theta\|_\infty)\|\nabla\theta\|_\infty|x-y|
\end{equation}
for all $y\in\partial\Omega$ and $t\in(-1,1)$. Then, we compute 
\begin{equation}\label{DthetaK.eq4}
\frac{d}{dt}\frac{1}{|x_t-y_t|^{n}}=-n\frac{(x_t-y_t)\cdot(\theta(x)-\theta(y))}{|x_t-y_t|^{n+2}}
\end{equation}
and, using \eqref{DthetaK.eq1.1}, we conclude that 
\begin{equation}\label{DthetaK.eq5}
\left|\frac{d}{dt}\frac{1}{|x_t-y_t|^{n}}\right|\le n 2^{n+2}(1+\|\nabla\theta\|_\infty)\|\nabla\theta\|_\infty\frac{1}{|x-y|^n}
\end{equation}
for all $y\in\partial\Omega\setminus\{x\}$ and $|t|\le (1/2)\|\nabla\theta\|_\infty^{-1}$. Putting together \eqref{DthetaK.eq3}, \eqref{DthetaK.eq5}, and the fact that the map $t\mapsto\sigma_n[I+t\theta]$ is real analytic close to zero, we conclude that 
\begin{equation}\label{DthetaK.eq6}
\left|\frac{d}{dt}\left((\eta(y)-\eta(x))\frac{(x_t-y_t)\cdot{\nu_t(y_t)}}{s_n|x_t-y_t|^{n}}\sigma_n[I+t\theta](y)\right)\right|\le C|x-y|^{2-n}
\end{equation}
for some $C>0$, uniformly for $t$ in a neighborhood of zero. Since $\partial\Omega\ni y\mapsto |x-y|^{2-n}\in\mathbb{R}$ is integrable on $\partial\Omega$, the proposition follows by \eqref{DthetaK.eq2} and \eqref{DthetaK.eq6}.
\end{proof}

Let us clarify that, in the proof of the previous Proposition \ref{DthetaK}, the presence of the difference $(\eta(y)-\eta(x))$ in the expression for $K_{I+t\theta}$ is extremely helpful, but not necessary. A more refined estimate of $(x_t-y_t)\cdot{\nu_t(y_t)}$ and its first and second derivatives may demonstrate that these are smaller than a constant times $|x-y|^{1+\alpha}$. Such a result would be sufficient to establish an analogue of Proposition \ref{DthetaK} where $\eta(y)$ replaces $(\eta(y)-\eta(x))$, and a similar result for $K^*_{I+t\theta}$ (cf.~Potthast \cite{Po94}). However, the presence of $(\eta(y)-\eta(x))$ will also be very helpful in the next section.

\subsection{Computing $\frac{d}{dt}(K_{I+t\theta})_{|t=0}$}\label{dK}
We now leverage Proposition \ref{DthetaK} to compute a convenient expression for $\frac{d}{dt}(K_{I+t\theta}[\eta])_{|t=0}$. We begin by recalling that 
\begin{equation}\label{divT}
\frac{d}{dt}{\sigma_n[I+t\theta]}_{|t=0}=\mathrm{div}_{\partial\Omega}\,\theta\, ,
\end{equation}
as proven, for example, in Henrot and Pierre \cite[p.~197]{HePi18} or Costabel and Le Lou\"er \cite[Lemma 4.2]{CoLe12a}.  Then, by Proposition \ref{DthetaK} and  by \eqref{DthetaK.eq4} and \eqref{divT}, we can verify that, for $x\in\partial\Omega$, 
\[
\begin{split}
&\frac{d}{dt}\left(K_{I+t\theta}[\eta]\right)_{|t=0}(x)\\
&\qquad={-}\int_{\partial\Omega}(\eta(y)-\eta(x)){\Bigl(\theta(x)-\theta(y)-(\nabla_{\partial\Omega}{\theta(y)})^\intercal(x-y)\Bigr)\cdot{\nu_\Omega(y)}}\frac{1}{s_n|x-y|^{n}}\,d\sigma_y\\
&\qquad\quad{+n}\int_{\partial\Omega}(\eta(y)-\eta(x))(x-y)\cdot{\nu_\Omega(y)}\frac{(x-y)\cdot(\theta(x)-\theta(y))}{s_n|x-y|^{n+2}}\,d\sigma_y\\
&\qquad\quad{-}\int_{\partial\Omega}(\eta(y)-\eta(x))\frac{(x-y)\cdot{\nu_\Omega(y)}}{s_n|x-y|^{n}}\,(\mathrm{div}_{\partial\Omega}\,\theta(y))\,d\sigma_y\,.
\end{split}
\]
By rearranging the terms on the right-hand side, and recalling that 
\[
\nabla^2 E(x-y)=-\frac{1}{s_n}\frac{1_n}{|x-y|^n}+\frac{n}{s_n}\frac{(x-y)\otimes (x-y)}{|x-y|^{n+2}}\,,
\]
where $1_n$ is the $n\times n$ identity matrix, we see that 
\[
\begin{split}
&\frac{d}{dt}\left(K_{I+t\theta}[\eta]\right)_{|t=0}(x)\\
&\qquad=\int^*_{\partial\Omega}(\eta(y)-\eta(x))\theta(x)^\intercal\nabla^2 E_n(x-y)\nu_\Omega(y)\,d\sigma_y\\
&\qquad\quad-\int_{\partial\Omega}^*(\eta(y)-\eta(x))\theta(y)^\intercal\nabla^2 E_n(x-y)\nu_\Omega(y)\,d\sigma_y\\
&\qquad\quad-\int^*_{\partial\Omega}(\eta(y)-\eta(x))\left[(\nabla_{\partial\Omega}\theta(y))^\intercal \nabla E_n(x-y)\right]\cdot\nu_\Omega(y)\,d\sigma_y\\
&\qquad\quad+\int_{\partial\Omega}(\eta(y)-\eta(x))(\mathrm{div}_{\partial\Omega}\theta)(y)\,\nu_\Omega(y)\cdot\nabla E_n(x-y)\,d\sigma_y\,,
\end{split}
\]
where $\int^*_{\partial\Omega}$ denotes the principal value integral (see \eqref{intstar} in the appendix). The last equality, combined with \eqref{thetangentialjump.eq1} and \eqref{thesistersact.eq1} in the appendix, yields 
\[
\begin{split}
&\frac{d}{dt}\left(K_{I+t\theta}[\eta]\right)_{|t=0}(x)\\
&\qquad=(\theta(x)\cdot\nu_\Omega(x))\;T_{\partial\Omega}[\eta](x)+\theta(x)\cdot\nabla_{\partial\Omega}K_{\partial\Omega}[\eta](x)\\
&\qquad\quad+\int_{\partial\Omega}^*\theta(y)\cdot\nu_\Omega(y)\;(\nabla_{\partial\Omega}\eta(y))\cdot \nabla E_n(x-y)\,d\sigma_y\\
&\qquad\quad-\int_{\partial\Omega}\theta(y)\cdot(\nabla_{\partial\Omega}\eta(y))\;\nu_\Omega(y)\cdot\nabla E_n(x-y)\,d\sigma_y\,,
\end{split}
\]
where $T_{\partial\Omega}$ is defined as in \eqref{Mr.T}.
We readily obtain the following 
\begin{proposition}\label{niceDthetaK} If $\theta\in (C^{1,\alpha}(\partial\Omega))^n$ and $\eta\in C^{1,\alpha}(\partial\Omega)$, then
\[
\begin{split}
&\frac{d}{dt}\left(K_{I+t\theta}[\eta]\right)_{|t=0}(x)\\
&\qquad= \int_{\partial\Omega}^*\theta(y)\cdot\nu_\Omega(y)\;(\nabla_{\partial\Omega}\eta(y))\cdot \nabla E_n(x-y)\,d\sigma_y\\
&\qquad\quad+(\theta(x)\cdot\nu_\Omega(x))\;T_{\partial\Omega}[\eta](x)+\theta(x)\cdot\nabla_{\partial\Omega}K_{\partial\Omega}[\eta](x)-K_{\partial\Omega}[\theta\cdot\nabla_{\partial\Omega}\eta](x)
\end{split}
\]
for all $x\in\partial\Omega$.
\end{proposition}

\subsection{Back to $d_\phi \Lambda ^m_h(\lambda _1[\phi _0],\ldots ,\lambda _m[\phi _0])$}\label{backtoL}

We can now put together the results from the previous sections and return to the problem of computing the shape derivatives of the matrix $A_\phi$ and the symmetric functions $\Lambda^m_h(\lambda_1[\phi],\ldots ,\lambda_m[\phi])$. According to Propositions \ref{der} and \ref{switcheroo}, the differential of $A_\phi$ at $\phi_0=I$ (the identity of $\partial\Omega$) is given by 
\[
(d_\phi A_{\phi_0}.\theta)_{ij}= \int_{\partial \Omega}\left(d_\phi(K_{\phi_0})[S_{\partial \Omega}[\mu_i]].\theta\right)\mu_j \, d\sigma 
\]
for all $\theta\in (C^{1,\alpha}(\partial\Omega))^n$, where $i,j=1,\ldots,n$. Subsequently, by considering $\phi_t:=I+t\theta$, we find that 
\[
(d_\phi A_{\phi_0}.\theta)_{ij}=\Bigl(\frac{d}{dt}(A_{I+t\theta})_{|t=0}\Bigr)_{ij}=\int_{\partial \Omega}\frac{d}{dt}\left(K_{I+t\theta}[S_{\partial \Omega}[\mu_i]]\right)_{|t=0}\mu_j \, d\sigma\,,
\]
and by substituting the expression for $\frac{d}{dt}(K_{I+t\theta})_{|t=0}$ obtained in Proposition \ref{niceDthetaK}, we deduce that
\begin{equation}\label{eq:dtKIt}
\begin{split}
&\Bigl(\frac{d}{dt}(A_{I+t\theta})_{|t=0}\Bigr)_{ij}\\
& =\int_{\partial \Omega}\int_{\partial \Omega}^\ast \left(\theta(y)\cdot \nu_{\Omega}(y)\right)\bigg(\nabla_{\partial \Omega}S_{\partial \Omega}[\mu_i](y)\bigg)\cdot \bigg(\nabla_{\partial \Omega}E_n(x-y)\bigg)\, d\sigma_y\; \mu_j(x)\, d\sigma_x\\
&\qquad +\int_{\partial \Omega}(\theta\cdot \nu_{\Omega})\; T_{\partial \Omega}[S_{\partial \Omega}[\mu_i]]\;\mu_j\, d\sigma\\
&\qquad +\int_{\partial \Omega}\bigg(\theta_{\partial \Omega}\cdot \nabla_{\partial \Omega}K_{\partial \Omega}[S_{\partial \Omega}[\mu_i]]-K_{\partial \Omega}[\theta_{\partial \Omega}\cdot \nabla_{\partial \Omega}S_{\partial \Omega}[\mu_i]]  \bigg)\;\mu_j\, d\sigma\,,
\end{split}
\end{equation}
where $\theta_{\partial\Omega}:=\theta-\nu_\Omega\otimes\nu_\Omega\; \theta$ is the tangential component of $\theta$.
We now study the three terms on the right-hand side of \eqref{eq:dtKIt} one at a time. We begin with the first one. 
\begin{lemma}\label{first} We have
\[
\begin{split}
&\int_{\partial \Omega}\int_{\partial \Omega}^\ast \left(\theta(y)\cdot \nu_{\Omega}(y)\right)\bigg(\nabla_{\partial \Omega}S_{\partial \Omega}[\mu_i](y)\bigg)\cdot \bigg(\nabla_{\partial \Omega}E_n(x-y)\bigg)\, d\sigma_y \mu_j(x)\, d\sigma_x\\
&\qquad=-\int_{\partial \Omega} \left(\theta\cdot \nu_{\Omega}\right)\nabla_{\partial \Omega}S_{\partial \Omega}[\mu_i] \cdot \nabla_{\partial \Omega}S_{\partial \Omega}[ \mu_j]\, d\sigma\,.
\end{split}
\]
\end{lemma}
\begin{proof}
By a standard argument, which utilizes the theorem on taking uniform limits under the integral sign (see, for example, Mikhlin \cite[Section 9]{Mi65}), we can switch the integrals with respect to $x$ and $y$ and obtain
\[
\begin{split}
&\int_{\partial \Omega}\int_{\partial \Omega}^\ast \theta(y)\cdot \nu_{\Omega}(y)\bigg(\nabla_{\partial \Omega}S_{\partial \Omega}[\mu_i](y)\bigg)\cdot \bigg(\nabla_{\partial \Omega}E_n(x-y)\bigg)\, d\sigma_y\; \mu_j(x)\, d\sigma_x\\
&=-\int_{\partial \Omega} \theta(y)\cdot \nu_{\Omega}(y)\nabla_{\partial \Omega}S_{\partial \Omega}[\mu_i](y) \cdot \int^*_{\partial \Omega}\bigg(\nabla_{\partial \Omega}E_n(y-x)\bigg)\mu_j(x)\, d\sigma_x\, d\sigma_y.
\end{split}
\]
Note that the minus sign in front of the right-hand side appears due to the substitution of $\nabla_{\partial \Omega}E_n(x-y)$ with $\nabla_{\partial \Omega}E_n(y-x)$. 
By the jump properties of the derivatives of the single-layer potential \eqref{nablasinglejump}, we have
\[
\int^*_{\partial \Omega}\bigg(\nabla_{\partial \Omega}E_n(y-x)\bigg)\mu_j(x)\, d\sigma_x=\nabla_{\partial \Omega}S_{\partial \Omega}[\mu_j](y)\, ,
\]
and the lemma follows.
\end{proof}

It's now the turn of the second term in the right-hand side of \eqref{eq:dtKIt}. 

\begin{lemma}\label{second} If $\lambda\neq1/2$, then we have
\begin{equation}\label{second.eq1}
\begin{split}
&\int_{\partial \Omega}(\theta\cdot \nu_{\Omega})\; T_{\partial \Omega}[S_{\partial \Omega}[\mu_i]]\;\mu_j\, d\sigma\\
&\qquad=\frac{\frac{1}{2}+\lambda}{\frac{1}{2}-\lambda}\int_{\partial \Omega}  ( \theta \cdot \nu_{\Omega}) \left(\nu_{\Omega} \cdot \nabla S^+_{\partial \Omega}[\mu_i]\right)\left(\nu_{\Omega} \cdot \nabla S^+_{\partial \Omega}[\mu_j]\right)\, d\sigma \, .
\end{split}
\end{equation}
\end{lemma}
\begin{proof} We recall the identity
\begin{equation}\label{second.eq2}
\int_{\partial \Omega}f\; T_{\partial \Omega}[g]\, d\sigma=\int_{\partial \Omega}T_{\partial \Omega}[f]\; g\, d\sigma\, ,
\end{equation}
valid for all functions $f,g\in C^{1,\alpha}(\partial\Omega)$. This identity can be established using an argument based on the first Green's identity. For a proof, refer to McLean \cite[Theorem 8.21]{Mc00}. We aim to apply \eqref{second.eq2} to the expression on the left-hand side of \eqref{second.eq1}. However, while $S_{\partial \Omega}[\mu_i]$ belongs to $C^{1,\alpha}(\partial\Omega)$, the function $\partial\Omega\ni x\mapsto (\theta(x)\cdot\nu(x))\mu_j(x)\in\mathbb{R}$ is only of class $C^{0,\alpha}(\partial\Omega)$. Therefore, we introduce a sequence $\{\eta_k\}_{k=1}^\infty$ in $C^{1,\alpha}(\partial\Omega)$ converging to $(\theta\cdot\nu)\mu_j$ in $C^{0,\alpha}(\partial\Omega)$. Applying \eqref{second.eq2}, we have
\[
\int_{\partial \Omega}\eta_k\; T_{\partial \Omega}[S_{\partial \Omega}[\mu_i]]\, d\sigma=\int_{\partial \Omega}T_{\partial \Omega}[\eta_k]\;S_{\partial \Omega}[\mu_i]\, d\sigma\,.
\]
Using the second Green's identity in $\Omega$, the last integral is equal to 
\[
\int_{\partial \Omega}   D^+_{\partial \Omega}[\eta_k] \bigg(\nu_{\Omega}\cdot\nabla S^+_{\partial \Omega}[\mu_i]\bigg)\, d\sigma\,. 
\]
By the jump properties of the double and single-layer potentials \eqref{jump} and \eqref{nusinglejump}, this expression reads
\[
\int_{\partial \Omega}   \Bigl(\frac{1}{2}I+K_{\partial \Omega}\Bigr)[\eta_k]  \Bigl(\frac{1}{2}I-K_{\partial \Omega}^\ast\Bigr)[\mu_i]\, d\sigma\,.
\]
As $\mu_i$ is an eigenfunction of $K_{\partial \Omega}^\ast$ with eigenvalue $\lambda$, this becomes
\[
\Bigl(\frac{1}{2}-\lambda\Bigr)\int_{\partial \Omega}   \Bigl(\frac{1}{2}I+K_{\partial \Omega}\Bigr)[\eta_k]\; \mu_i \, d\sigma\,,
\]
and, since $\frac{1}{2}I+K_{\partial \Omega}$ and $\frac{1}{2}I+K_{\partial \Omega}^*$ are adjoint, we obtain
\[
\Bigl(\frac{1}{2}-\lambda\Bigr)\int_{\partial \Omega}\eta_k\;\Bigl(\frac{1}{2}I+K_{\partial \Omega}^\ast\Bigr)[\mu_i] \, d\sigma\,.
\]
Using again the fact that  $\mu_i$ is an eigenfunction of $K_{\partial \Omega}^\ast$, we find that the last integral equals
\[
\Bigl(\frac{1}{2}-\lambda\Bigr)\Bigl(\frac{1}{2}+\lambda\Bigr)\int_{\partial \Omega}   \eta_k\,\mu_j  \, d\sigma\,.
\]
Therefore, we have established that 
\[
\int_{\partial \Omega}\eta_k\; T_{\partial \Omega}[S_{\partial \Omega}[\mu_i]]\, d\sigma=\Bigl(\frac{1}{2}-\lambda\Bigr)\Bigl(\frac{1}{2}+\lambda\Bigr)\int_{\partial \Omega}   \eta_k\,\mu_j  \, d\sigma\,,
\]
and taking the limit as $k\to\infty$, we conclude that 
\[
\int_{\partial \Omega}(\theta\cdot \nu_{\Omega})\; \nu_{\Omega}\cdot \nabla D^{\pm}_{\partial \Omega}[S_{\partial \Omega}[\mu_i]]\;\mu_j\, d\sigma=\Bigl(\frac{1}{2}-\lambda\Bigr)\Bigl(\frac{1}{2}+\lambda\Bigr)\int_{\partial \Omega}   \left(\theta \cdot \nu_{\Omega}\right) \mu_i\mu_j  \, d\sigma\,.
\]
Since both $\mu_i$ and $\mu_j$ are eigenfunctions of $K_{\partial \Omega}^\ast$ for the same eigenvalue $\lambda$, which we are assuming to be different from $1/2$, the expression on the right-hand side is equal to
\[
\frac{(\frac{1}{2}-\lambda)(\frac{1}{2}+\lambda)}{(\frac{1}{2}-\lambda)^2}\int_{\partial \Omega}   \left(\theta \cdot \nu_{\Omega}\right) \Big(\frac{1}{2}I-K_{\partial \Omega}^\ast\Big)[\mu_i]\;\Big(\frac{1}{2}I-K_{\partial \Omega}^\ast\Big)[\mu_j]  \, d\sigma\,
\]
and, by the jump properties of the single-layer potential, we establish the validity of the equality in the lemma.
\end{proof}

Finally, for the third term on the right-hand side of \eqref{eq:dtKIt}, we have the following lemma.
\begin{lemma}\label{third}
We have
\[
\int_{\partial \Omega}\bigg(\theta_{\partial \Omega}\cdot \nabla_{\partial \Omega}K_{\partial \Omega}[S_{\partial \Omega}[\mu_i]]-K_{\partial \Omega}[\theta_{\partial \Omega}\cdot \nabla_{\partial \Omega}S_{\partial \Omega}[\mu_i]]  \bigg)\;\mu_j\, d\sigma=0\,.
\]
\end{lemma}
\begin{proof}
By Calderon's identity \eqref{PSP} and the fact that $\mu_i$ is an eigenfunction of $K^*_{\partial\Omega}$ with eigenvalue $\lambda$, we have
\[
\begin{split}
&\int_{\partial \Omega}\theta_{\partial \Omega}\cdot \nabla_{\partial \Omega}K_{\partial \Omega}[S_{\partial \Omega}[\mu_i]]\;\mu_j\,d\sigma\\
&\qquad=\int_{\partial \Omega}\theta_{\partial \Omega}\cdot \nabla_{\partial \Omega}S_{\partial \Omega}[K^*_{\partial \Omega}[\mu_i]]\;\mu_j\,d\sigma=\lambda \int_{\partial \Omega}\theta_{\partial \Omega}\cdot \nabla_{\partial \Omega}S_{\partial \Omega}[\mu_i]\;\mu_j\,d\sigma\,.
\end{split}
\]
Moreover, since $K_{\partial\Omega}$ and $K^*_{\partial\Omega}$ are adjoint operators, and $\mu_j$ is an eigenfunction of $K^*_{\partial\Omega}$ with eigenvalue $\lambda$, we compute
\[
\begin{split}
&\int_{\partial \Omega}K_{\partial \Omega}[\theta_{\partial \Omega}\cdot \nabla_{\partial \Omega}S_{\partial \Omega}[\mu_i]]\;\mu_j\, d\sigma\\
&\qquad=\int_{\partial \Omega}\theta_{\partial \Omega}\cdot \nabla_{\partial \Omega}S_{\partial \Omega}[\mu_i]\;K^*_{\partial \Omega}[\mu_j]\, d\sigma=\lambda \int_{\partial \Omega}\theta_{\partial \Omega}\cdot \nabla_{\partial \Omega}S_{\partial \Omega}[\mu_i]\;\mu_j\,d\sigma\,.
\end{split}
\]
Hence, the lemma follows.
\end{proof}

Now, by equality \eqref{eq:dtKIt} and by Lemmas \ref{first}, \ref{second}, and \ref{third}, we readily deduce the following theorem for the matrix $A_\phi$ defined in \eqref{Aphi}.

\begin{theorem}\label{dAphi}
Under the assumptions and notations of Theorem \ref{main}, with $\lambda\neq1/2$ and $\phi_0=I$, we have
\[
\left(\frac{d}{dt}(A_{I+t\theta})_{|t=0}\right)_{ij}=\int_{\partial \Omega}\left(\theta \cdot \nu_{\Omega}\right)\; \bigg( -\nabla_{\partial \Omega}u_i \cdot \nabla_{\partial \Omega}u_j+\frac{\frac{1}{2}+\lambda}{\frac{1}{2}-\lambda}(\nu_{\Omega} \cdot \nabla u_i)( \nu_{\Omega} \cdot \nabla u_j) \bigg)\, d\sigma
\]
for all $i,j=1,\ldots,m$, where $u_i:=S^+_{\partial \Omega}[\mu_i]$ and $u_j:=S^+_{\partial \Omega}[\mu_j]$.\end{theorem}

We can now establish the theorem on the first derivatives of the symmetric functions of the eigenvalues:

\begin{theorem}\label{dLambda} With the assumptions and notations of Theorems \ref{main} and \ref{dAphi}, we have
\begin{equation}\label{dLambda.eq1}
\begin{split}
&\frac{d}{dt}\Bigl(\Lambda^m_h(\lambda_1[I+t\theta],\ldots,\lambda_m[I+t\theta])\Bigr)_{|t=0}=\lambda^{h-1}\binom{m-1}{h-1}\mathrm{trace}\left(\frac{d}{dt}(A_{I+t\theta})_{|t=0}\right)\\
&\quad=\lambda^{h-1}\binom{m-1}{h-1}\sum_{i=1}^m\int_{\partial \Omega}\left(\theta \cdot \nu_{\Omega}\right)\; \bigg( -\nabla_{\partial \Omega}u_i \cdot \nabla_{\partial \Omega}u_i+\frac{\frac{1}{2}+\lambda}{\frac{1}{2}-\lambda}(\nu_{\Omega} \cdot \nabla u_i)( \nu_{\Omega} \cdot \nabla u_i) \bigg)\, d\sigma
\end{split}
\end{equation}
for all $h=1,\dots,m$.
\end{theorem}
\begin{proof} Given that $\mathcal{U}\ni \phi\mapsto A_\phi\in\mathbb{M}_m$ is real analytic, as established in Proposition~\ref{der}, we deduce that the function  $t\mapsto A_{I+t\theta}$ is real analytic in a neighborhood of zero. Hence, we can invoke the Newton-Puiseux/Rellich-Nagy theorem, which guarantees the existence of analytic branches of eigenvalues $t\mapsto \lambda_{j,t}$ for $j=1,\dots,m$, with $\lambda_{j,0}=\lambda$, such that, for any fixed $t$, the eigenvalues $\lambda_{1,t}$,\ldots,$\lambda_{m,t}$ coincide with the eigenvalues $\lambda_1[I+t\theta]\leq \ldots\leq\lambda_m[I+t\theta]$ from Theorem \ref{main}, albeit potentially in a different order (see Kato \cite[Chapter 7, \S 1]{Ka95}, Rellich \cite[Theorem 1, p.~33]{Re69}). 

Since the symmetric functions are invariant with respect to the order, we have
\[
\Lambda^m_h(\lambda_1[I+t\theta],\ldots,\lambda_m[I+t\theta])=\Lambda^m_h(\lambda_{1,t},\ldots,\lambda_{m,t})\,.
\]
We can then use the expression on the right-hand side to compute the differential of $\Lambda^m_h(\lambda_1[I+t\theta],\ldots,\lambda_m[I+t\theta])$. By definition \eqref{symmetric}, we obtain:
\[
\frac{d}{dt}\Bigl(\Lambda^m_h(\lambda_1[I+t\theta],\ldots,\lambda_m[I+t\theta])\Bigr)_{|t=0}=\sum_{\substack{j_1,\ldots,j_h\in \{1,\ldots, m\}\\ j_1 < \cdots < j_h}} \frac{d}{dt}\left(\lambda_{j_1,t} \cdots \lambda_{j_h,t}\right)_{|t=0}\,.
\]
Since $\lambda_{1,0}=\ldots=\lambda_{m,0}=\lambda$, the right-hand side simplifies to:
\[
\lambda^{h-1}\sum_{\substack{j_1,\ldots,j_h\in \{1,\ldots, m\}\\ j_1 < \cdots < j_h}} \frac{d}{dt}\left(\lambda_{j_1,t}\right)_{|t=0}+\frac{d}{dt}\left(\lambda_{j_2,t}\right)_{|t=0}+\cdots+\frac{d}{dt}\left(\lambda_{j_h,t}\right)_{|t=0}\,.
\]
In this sum, each term $\frac{d}{dt}\left(\lambda_{j,t}\right)_{|t=0}$ appears as many times as the possible  subsets of $h-1$ elements of $\{1,\ldots,m\}\setminus\{j\}$, which is $\binom{m-1}{h-1}$ times. Hence, the last expression equals:
\[
\lambda^{h-1}\binom{m-1}{h-1}\sum_{j=1}^m \frac{d}{dt}\left(\lambda_{j,t}\right)_{|t=0}\,.
\]
Since $\lambda_{1,t}+\dots+\lambda_{m,t}=\mathrm{trace}\left(A_{I+t\theta}\right)$, this can be expressed as:
\[
\lambda^{h-1}\binom{m-1}{h-1}\mathrm{trace}\left(\frac{d}{dt}(A_{I+t\theta})_{|t=0}\right)\,.
\] 
Therefore, the theorem follows directly from Theorem \ref{dAphi}.
\end{proof}

\begin{remark}\label{phit}
The directional derivatives computed in Theorems \ref{dAphi} and \ref{dLambda}, using the perturbation \( t \mapsto I + t\theta \), yield expressions for the differentials of the maps that take \(\phi\) to \(A_\phi\) and to \(\Lambda^m_h(\lambda_1[\phi], \ldots, \lambda_m[\phi])\). For example, when $\phi_0=I$, we have
\[
d_\phi A_{\phi_0}.\theta=\frac{d}{dt}(A_{I + t\theta})_{|t=0}\,.
\]
Then, if we replace the perturbation \( t \mapsto I + t\theta \) with a more general real analytic perturbation \( t \mapsto \phi_t \), where \(\phi_0 = I\), the equality in Theorem \ref{dAphi} becomes  
\[
\left(\frac{d}{dt}(A_{\phi_t})_{|t=0}\right)_{ij} = \int_{\partial \Omega} \left( \frac{d}{dt}(\phi_t)_{|t=0} \cdot \nu_{\Omega} \right) \bigg( -\nabla_{\partial \Omega} u_i \cdot \nabla_{\partial \Omega} u_j + \frac{\frac{1}{2} + \lambda}{\frac{1}{2} - \lambda} (\nu_{\Omega} \cdot \nabla u_i)(\nu_{\Omega} \cdot \nabla u_j) \bigg) \, d\sigma,
\]
where \( u_i \) and \( u_j \) are as defined in the theorem.
\end{remark}

We conclude this section with a Rellich-Nagy-type result concerning the right- and left-hand derivatives of multiple eigenvalues. The proof follows the approach of \cite[Corollary 2.28]{LaLa04}, which deals with families of self-adjoint operators. However, since our framework differs slightly, we include additional details for clarity. We present the result for  a perturbation \( t \mapsto \phi_t \), instead of the more restrictive case \( t \mapsto I + t\theta \).

\begin{corollary}[Rellich-Nagy]\label{corhenry} 
Under the assumptions and notations of Theorem \ref{main}, assume that $\lambda \neq 1/2$ and $\phi_0 = I$. Let $t_0 > 0$ and $(-t_0, t_0) \ni t \mapsto \phi_t \in \mathcal{U}$ be a real-analytic function. Then, for a possibly smaller $t_0 > 0$, there exist $m$ real-analytic functions $(-t_0, t_0) \ni t \mapsto \lambda_{j,t} \in \mathbb{R}$ such that  
\begin{equation}\label{henry0}
\{\lambda_{j,t} : j \in \{1, \dots, m\} \} = \{\lambda_j[\phi_t] : j \in \{1, \dots, m\}\}.
\end{equation}
Moreover, the set $\big\{ \frac{d}{dt}\left(\lambda_{j,t}\right)_{|t=0} : j \in \{1, \dots, m\}\big\}$  
of the first-order derivatives at \( t = 0 \) coincides with the set of eigenvalues of the matrix  
\begin{equation}\label{henry}
\left(\int_{\partial \Omega}\left(\frac{d}{dt}\big(\phi_t\big)_{|t=0} \cdot \nu_{\Omega}\right) \bigg( -\nabla_{\partial \Omega}u_i \cdot \nabla_{\partial \Omega}u_j + \frac{\frac{1}{2} + \lambda}{\frac{1}{2} - \lambda} (\nu_{\Omega} \cdot \nabla u_i)(\nu_{\Omega} \cdot \nabla u_j) \bigg) \, d\sigma\right)_{i,j=1,\dots,m},
\end{equation}
where \( u_i := S^+_{\partial \Omega}[\mu_i] \) and \( u_j := S^+_{\partial \Omega}[\mu_j] \).

In particular, the functions \( (-t_0, t_0) \ni t \mapsto \lambda_j[\phi_t] \in \mathbb{R} \) are real-analytic on \([0, t_0)\) and \((-t_0, 0]\), and the right- and left-hand derivatives  $\frac{d}{dt}\left(\lambda_{j}[\phi_t]\right)_{|t=0^+}$, $\frac{d}{dt}\left(\lambda_{j}[\phi_t]\right)_{|t=0^-}$ are given by the eigenvalues of the matrix in \eqref{henry}.
\end{corollary}

\begin{proof}  As in the proof of Theorem~\ref{dLambda}, the existence of functions $(-t_0, t_0) \ni t \mapsto \lambda_{j,t} \in \mathbb{R}$ satisfying \eqref{henry0} follows from the Newton-Puiseux/Rellich-Nagy theorem. Thus, it remains to prove that the set of their derivatives at zero coincides with the eigenvalues of the matrix in \eqref{henry}. Recall that the eigenvalues $\lambda_{j,t}$ are the roots of the characteristic equation of the matrix $A_{\phi_t}$. Therefore, 

\begin{equation}\label{henry1}
{\rm det} (\tau I - A_{\phi_t}) = \prod_{j=1}^m (\tau - \lambda_{j,t})
\end{equation}
for all $\tau \in \mathbb{R}$. By substituting $\tau$ with $\lambda + \chi t$ in \eqref{henry1} for all $\chi \in \mathbb{R}$ and differentiating both sides of the resulting equation $m$ times with respect to $t$ at $t = 0$, and recallig that $\lambda=\lambda_{1,0}=\ldots=\lambda_{m,0}$, we obtain 

\[
{\rm det} \left(\chi I - \frac{d}{dt}(A_{\phi_t})_{|t=0}\right) = \prod_{j=1}^m \left(\chi - \frac{d}{dt}\left(\lambda_{j,t}\right)_{|t=0}\right).
\]
This, combined with Theorem~\ref{dAphi} and Remark \ref{phit}, completes the proof.
\end{proof}

\begin{remark}\label{henrysym} For computational purposes, it might be useful to observe that the matrix in \eqref{henry} is symmetric. Thus, by possibly modifying the basis \( \mu_1, \ldots, \mu_m \) through an orthogonal transformation, we can always work in a setting where the matrix in \eqref{henry} is diagonal, and the set of derivatives at zero of the analytic branches coincides with the set of its diagonal entries. 
\end{remark}

\section{Comparison with the existing literature}\label{sec:comparison}

\subsection{Comparison with Grieser's findings}\label{sec:grieser}
The results from the previous section can be compared to the findings of Grieser in \cite{Gr14}. In particular, the case of a one-dimensional perturbation generated by the map $ \phi_t $ defined by
\[
\phi_t(x):=x+t\theta(x)\quad\text{for all } x\in\mathbb{R}^n,
\]
where $ \theta\in (C^{1,\alpha}(\partial\Omega))^n $, is sufficiently general to encompass the assumptions of Grieser's paper, in which $\theta=a\nu_\Omega$ for some smooth scalar function $a$ on $\partial\Omega$. 

In \cite[Theorem 4]{Gr14}, Grieser demonstrates the existence of analytic branches of eigenvalues and eigenfunctions and provides a formula to compute the first derivatives of the eigenvalues at zero (he also derives a formula for the second derivatives, but tackling these is beyond our current scope). Up to renormalizing the functions $u_i$, the expression that we can deduce from Corollary \ref{corhenry} coincides with Grieser's one and extends it from the case where $ \theta=a\nu_\Omega $  to the more general case where $ \theta $ is any vector function in $ (C^{1,\alpha}(\partial\Omega))^n $. In other words, we can allow perturbations that are not in the normal direction. Besides, we have reduced the regularity assumptions from the case of smooth sets to the $C^{1,\alpha}$ setting.

We may also notice that Corollary \ref{corhenry} addresses only the eigenvalues, without considering the analytic branches of the corresponding eigenfunctions. These can, however, be recovered using the celebrated result of Rellich and Nagy, which guarantees the existence of real-analytic maps  
\[
(-t_0, t_0) \ni t \mapsto \lambda_{j,t} \in \mathbb{R} \quad \text{and} \quad (-t_0, t_0) \ni t \mapsto v_j(t) \in \mathbb{R}^n,
\]
for \( j = 1, \dots, m \), such that \( \lambda_{j,0} = \lambda \) and  
\[
A_{I+t\theta} \, v_j(t) = \lambda_{j,t} \, v_j(t).
\]
The set \( \{v_1(t), \dots, v_m(t)\} \) is orthonormal with respect to the inner product of \( \mathbb{R}^m \) and, consequently, the matrix \( R(t):=(v_1(t) \dots v_m(t)) \), whose columns are the vectors \( v_j(t) \), is orthogonal. Defining  
\[
\tilde\mu_i(t) := \sum_{j=1}^m R_{ij}(t) \mu_j[I+t\theta] \quad \text{for all } i = 1, \dots, m \text{ and } t \in (-t_0, t_0),
\]
we obtain a \( (\langle \cdot, \cdot \rangle_{\sigma_n[\phi_t] S_{\phi_t}}) \)-orthonormal system \( \{\tilde\mu_1(t), \ldots, \tilde\mu_m(t)\} \) of eigenfunctions of \( K^*_{I+t\theta} \), associated with the eigenvalues \( \lambda_{1,t}, \ldots, \lambda_{m,t} \), and the maps  
\[
(-t_0,t_0) \ni t \mapsto \tilde\mu_i(t) \in C^{0,\alpha}(\partial\Omega)
\]
are real analytic. Using the specific basis $ \{\tilde\mu_1(0),\ldots,\tilde\mu_m(0)\} $ for the eigenspace of $\lambda$, we can refine the result of Corollary \ref{corhenry} and write 
\[
\frac{d}{dt}\left(\lambda_{j,t}\right)_{|t=0}=\int_{\partial \Omega}\left(\theta \cdot \nu_{\Omega}\right)\;\biggl( -\bigl| \nabla_{\partial \Omega}\tilde u_j \bigr|^2+\frac{\frac{1}{2}+\lambda}{\frac{1}{2}-\lambda}(\nu_{\Omega} \cdot \nabla \tilde u_j)^2\biggr)\, d\sigma
\]
for all $ j=1,\ldots,m $, with $ \tilde u_j:=S^+_{\partial \Omega}[\tilde \mu_j(0)] $ (see also Remark \ref{henrysym}).  Then, taking $v_j:=\tilde u_j/\sqrt{\frac{1}{2}-\lambda}$, we obtain
\[
\frac{d}{dt}\left(\lambda_{j,t}\right)_{|t=0}=\int_{\partial \Omega}\left(\theta \cdot \nu_{\Omega}\right)\;\biggl( -\Bigl(\frac{1}{2}-\lambda\Bigr)\bigl| \nabla_{\partial \Omega} v_j \bigr|^2+\Bigl({\frac{1}{2}+\lambda}\Bigr)(\nu_{\Omega} \cdot \nabla v_j)^2\biggr)\, d\sigma
\]
for all $ j=1,\ldots,m $.
By the jump formula \eqref{nusinglejump} and the normalization of $\mu_j$,  it follows that $\| \nabla v_j\|_{L^2(\Omega)}=1$ for all $j=1, \dots , m$, which is the normalization of the plasmonic   eigenfunctions  used in \cite[Theorem 4]{Gr14}.

\subsection{Comparison with Zol\'esio's speed method}\label{sec:Zol\'esio}

We notice that only the normal component of $ \theta $ contributes to the first derivative \eqref{henry}. This observation aligns with what could be expected in light of Zol\'esio's velocity approach to shape perturbations (cf.~Zol\'esio \cite{Zo80}, Delfour and Zol\'esio \cite{DeZo11}, Soko\l owski and Zol\'esio \cite{SoKo92}). The basic idea of this approach is to identify a family $ \Phi_t $ of diffeomorphisms from $ \mathbb{R}^n $ to itself, where $ t $ ranges in an interval $ (-t_0,t_0) $ and $ \Phi_0(x)=x $, with the unique solution $ \xi(t,x) $ of a Cauchy problem 
\[
\begin{cases}
\frac{d}{dt}\xi(t,x)=V(t,\xi(t,x))\,,\\
\xi(0,x)=x\,.
\end{cases}
\]
Under suitable regularity assumptions, one can prove that $ \Phi_t(x)=\xi(t,x) $ whenever the non-autonomous vector field $ V(t,\xi) $ equals $ (\frac{d}{dt}\Phi_t)\circ\Phi_t^{-1}(\xi) $. For example, in the case where $ \Phi_t=I+t\theta $ for some $ \theta\in (C_0^{1,\alpha}(\mathbb{R}^n))^n $, the equivalence holds with $ V(t,\xi)=\theta\circ\Phi^{-1}_t(\xi) $. 

For a suitable vector field $ V(t,\xi) $, the perturbed set is defined by $ \Omega_t(V):=\Phi_t(\Omega) $, and a shape functional $ J $ is said to be shape-differentiable if the limit 
\[
dJ[\Omega,V]:=\lim_{t\to 0}\frac{J[\Omega_t(V)]-J[\Omega]}{t}
\] 
exists and defines a linear continuous functional on the space of admissible vector fields $ V(t,\xi) $ equipped with a suitable topology (this definition is slightly more restrictive than Zol\'esio's, but it suffices for our purposes). 

In our specific case, we consider the matrix $ A_\phi $ as defined in Theorem \ref{der}. Since $ K^*_\phi $ only depends on the image of $ \partial\Omega $ under $ \phi $, i.e., $ K^*_{\phi}=K^*_{\psi} $ whenever $ \phi(\partial\Omega)=\psi(\partial\Omega) $, we can define a (matrix-valued) shape functional by setting
\[
A[\Phi(\Omega)]=A_\phi
\]
for all admissible domains $ \Phi(\Omega) $, where $ \Phi\in (C^{1,\alpha}(\mathbb{R}^n))^n $ is a diffeomorphism such that $  \phi:=\Phi_{|\partial\Omega}\in\mathcal{U} $. Then, by virtue of the real analyticity, and hence differentiability, of the map $\mathcal{U}\ni\phi\mapsto A_\phi\in\mathbb{M}_m$, we have
\[
dA[\Omega,V]=\lim_{t\to 0}\frac{A[\Omega_t(V)]-A[\Omega]}{t}=\frac{d}{dt}(A_{\phi_t})_{|t=0}=d_\phi A_{\phi_0}.\left(\frac{d}{dt}\phi_t\right)_{|t=0}=d_\phi A_{\phi_0}.V(0,\cdot)_{|\partial\Omega}\,,
\]
where we understand that $\phi_t:=\Phi_{t|\partial\Omega}$.

So, we see at once that the shape derivative $dA[\Omega,V]$ only depends on the value at zero of the restriction to $\partial\Omega$ of vector field $V(t,\xi)$. Namely, on $\partial\Omega\ni\xi\mapsto V(0,\xi)\in\mathbb{R}^n$. In particular, the shape derivative $dA[\Omega,\theta\circ\Phi^{-1}_t(\xi)]$ corresponding to $\Phi_t=I+t\theta$, coincides with the shape derivative $dA[\Omega,\theta]$ corresponding to the autonomous equation 
\begin{equation}\label{autonomous}
\frac{d}{dt}\xi(t,x)=\theta(\xi(t,x))\,.
\end{equation}
Moreover,
\[
d_\phi A_{\phi_0}.\theta=dA[\Omega,\theta]\,
\]
and $\theta\mapsto dA[\Omega,\theta]$ defines a linear continuous functional on $(C_0^{1,\alpha}(\mathbb{R}^n))^n$.

Now, following a similar approach to Delfour and Zol\'esio \cite[Proof of Theorem 2.4]{DeZo90}, we can rely on a result by Nagumo \cite{Na42} (see also H\"ormander \cite[Theorem 8.5.11]{Ho03}), which guarantees that if 
\[
\theta\cdot\nu_\Omega=0\quad\text{on }\partial\Omega,
\] 
then every integral curve of equation \eqref{autonomous} that intersects $\partial\Omega$ remains confined to $\partial\Omega$. Thus, for such $\theta$, we have $\Omega_t(\theta)=\Omega$, and therefore $A[\Omega_t(V)]=A[\Omega]$, $dA[\Omega,\theta]=0$, and finally 
\[
d_\phi A_{\phi_0}.\theta=0
\]
(For further details, see Delfour and Zol\'esio \cite{DeZo90}). This approach enables us to prove that the null space of $d_\phi A_{\phi_0}$ contains the set 
\[
(C_0^{1,\alpha}(\mathbb{R}^n))^n_{\nu_\Omega}=\left\{\theta\in(C_0^{1,\alpha}(\mathbb{R}^n))^n\,:\; \theta\cdot\nu_{\Omega}=0\text{ on }\partial\Omega\right\}\,,
\]
a result confirmed by the explicit expression of $d_\phi A_{\phi_0}$ computed in Theorem \ref{dAphi}.

It is now convenient to choose $ R>0 $ such that $ \overline\Omega\subset \mathbb{B}_n(0,R) $ and consider vector fields $ \theta $ with support in $ \mathbb{B}_n(0,R) $. We observe that the corresponding set $ (C_0^{1,\alpha}(\mathbb{B}_n(0,R)))^n $ forms a Banach space, and $ (C_0^{1,\alpha}(\mathbb{B}_n(0,R)))^n_{\nu_\Omega}:=(C_0^{1,\alpha}(\mathbb{R}^n))^n_{\nu_\Omega}\cap (C_0^{1,\alpha}(\mathbb{B}_n(0,R)))^n $ is a closed subspace. Thus, $ d_\phi A_{\phi_0} $ defines a (matrix-valued) functional on the quotient Banach space 
\[
\frac{(C_0^{1,\alpha}(\mathbb{B}_n(0,R)))^n}{(C_0^{1,\alpha}(\mathbb{B}_n(0,R)))^n_{\nu_\Omega}}\,.
\] 
Furthermore, if $ \Omega $ is of class $ C^{2,\alpha} $, so that $ \nu_\Omega $ is of class $ C^{1,\alpha} $, then the map 
\[
\frac{(C_0^{1,\alpha}(\mathbb{B}_n(0,R)))^n}{(C_0^{1,\alpha}(\mathbb{B}_n(0,R)))^n_{\nu_\Omega}}\ni \theta\mapsto \theta\cdot\nu_\Omega\in C^{1,\alpha}(\partial\Omega)
\]
is an isomorphism of Banach spaces. In fact, the inverse of $ \theta\mapsto\theta\cdot\nu_\Omega $ is given by $ a\mapsto \tilde a \,n_\Omega $, where $\tilde a\in C_0^{1,\alpha}(\mathbb{B}_n(0,R))$ is an extension of $a$ and $ n_\Omega\in (C_0^{1,\alpha}(\mathbb{B}_n(0,R)))^n $ is an extension of $ \nu_\Omega $. Thus, provided that $ \Omega $ is of class $ C^{2,\alpha} $, we can conclude that there exists a (matrix-valued) functional $ \mathcal{A} $ on $ C^{1,\alpha}(\partial\Omega) $ such that 
\[
d_\phi A_{\phi_0}.\theta=\mathcal{A}(\theta\cdot\nu_\Omega)\,.
\]

We observe that in Theorem \ref{dAphi}, we reach the same conclusion with a weaker regularity assumption on $ \Omega $. However, it is possible that by combining Zol\'esio's velocity approach \cite{Zo80} with Grieser's findings in \cite{Gr14}, along with the analyticity results of Theorem \ref{main} and Proposition \ref{der}, we may be able to establish the result of Theorem \ref{dAphi} (and the first derivative formula \eqref{henry} for the analytic branches) without going through the computations of Section \ref{derivatives}, at least in the case of $C^{2,\alpha}$ domains. Nonetheless, the approach in Section \ref{derivatives} offers its own advantages. On one hand, we can deal with sets $ \Omega $ of class $ C^{1,\alpha} $. On the other hand, we obtain a proof that is entirely independent of Grieser's computations.

\section{Applications}\label{sec:applications}

\subsection{A Rellich-Poho\v{z}aev-type formula}

We can replicate an argument used to prove the Rellich-Poho\v{z}aev formula for Dirichlet-Laplacian eigenvalues and adapt it to the case of Neumann-Poincar\'e eigenvalues. 

Specifically, for Dirichlet-Laplacian eigenvalues, the Rellich-Poho\v{z}aev formula can be obtained by comparing two different expressions for the shape derivative in the case of domain dilation: one derived from specializing the Hadamard formula, and another derived from knowing that the eigenvalues scale by a factor of $t^{-2}$ when the domain dilates by a factor of $t>0$ (see  \cite{DiLa20}  with di Blasio for the more general case 
of the Finsler $p$-Laplacian).

In the case of the Neumann-Poincar\'e eigenvalues, we can use formula \eqref{henry} instead of the Hadamard formula and we can note that the (pull-back of the) Neumann-Poincar\'e operator $K^*_\phi$ (and thus its eigenelements) remains unchanged under domain dilation.

Indeed, dilations of the domain $\Omega$ correspond to the specific case where the maps $\phi_t$ are given by 
\[
(-1,1)\ni t\mapsto\phi_t(x):=x+tx\quad\text{for all } x\in\partial\Omega.
\]
Namely, this is the case where we have $\phi_t=I+tI$ with $t\in (-1,1)$ and $\theta(x)=x$ is the identity map on $\partial\Omega$. Then, with a straightforward computation based on the rule of change of variable in integrals and on the homogeneity of the (gradient of the) fundamental solution $E_n$, we can verify that $K^*_{I+tI}=K^*_{\partial\Omega}$ for all $t\in(-1,1)$. It follows that
\[
\frac{d}{dt}(K^*_{I+tI})_{|t=0}=0,
\]
and by Proposition \ref{der}, we deduce that
\[
\frac{d}{dt}(A_{I+tI})_{|t=0}=0.
\]
Hence, Theorem \ref{dAphi} implies that
\[
\int_{\partial \Omega}\left(x \cdot \nu_{\Omega}\right)\; \bigg( -\nabla_{\partial \Omega}u_i \cdot \nabla_{\partial \Omega}u_j+\frac{\frac{1}{2}+\lambda}{\frac{1}{2}-\lambda}(\nu_{\Omega} \cdot \nabla u_i)( \nu_{\Omega} \cdot \nabla u_j) \bigg)\, d\sigma_x=0.
\]
In particular, we have
\begin{equation}\label{RPL.eq1}
\int_{\partial \Omega}\left(x \cdot \nu_{\Omega}\right)\; \biggl(-\bigl| \nabla_{\partial \Omega} u_i \bigr|^2+\frac{\frac{1}{2}+\lambda}{\frac{1}{2}-\lambda}(\nu_{\Omega} \cdot \nabla u_i)^2\biggr)\, d\sigma_x=0.
\end{equation}
Now, provided that 
\begin{equation}\label{RPL.eq2}
\int_{\partial \Omega}\left(x \cdot \nu_{\Omega}\right)\bigl| \nabla u_i \bigr|^2\, d\sigma_x\neq 0,
\end{equation}
(note that $| \nabla u_i |^2=| \nabla_{\partial \Omega} u_i |^2+(\nu_{\Omega} \cdot \nabla  u_i)^2$) we can use \eqref{RPL.eq1} to derive an expression for $\lambda$. We obtain
\begin{equation}\label{RPL.eq3}
\lambda=\frac{1}{2}
\frac{\displaystyle\int_{\partial \Omega}\left(x \cdot \nu_{\Omega}\right)\; \biggl(\bigl| \nabla_{\partial \Omega} u_i \bigr|^2-(\nu_{\Omega} \cdot \nabla  u_i)^2\biggr)\, d\sigma_x}
{\displaystyle\int_{\partial \Omega}\left(x \cdot \nu_{\Omega}\right)\bigl| \nabla u_i \bigr|^2\, d\sigma_x}.
\end{equation}
We observe that condition \eqref{RPL.eq2} is certainly satisfied when $\Omega$ is a star-shaped domain with respect to the origin,  a case where $x\cdot \nu_\Omega\ge 0$ on $\partial\Omega$ and $x\cdot \nu_\Omega> 0$ on a subset of $\partial\Omega$ with positive measure (see, e.g., Pucci and Serrin \cite{PuSe86}). Furthermore, we note that $\nabla u_i=0$ on $\partial\Omega$ only when $u_i$ is constant and $\lambda=1/2$, an eigenvalue of multiplicity one, which does not depend on the shape of $\Omega$ (see, e.g., \cite[\S6.6]{DaLaMu21}). 

We may derive variations of \eqref{RPL.eq3} by using other symmetries of the operator $K^*_\phi$: Since $K^*_\phi$ is invariant under translations, for instance, in the direction of $\zeta\in\mathbb{R}^n$, we can derive a similar formula with $\zeta\cdot\nu_\Omega$ substituting $x\cdot\nu_\Omega$. Similarly, given that $K^*_\phi$ is invariant under rotation, we can derive a corresponding formula with $(Zx)\cdot\nu_\Omega$ replacing $x\cdot\nu_\Omega$, where $Z$ is a skew-symmetric matrix. (For this second example, note that for any rotation matrix function $(-t_0,t_0)\ni t\mapsto R_t\in SO(n)$ with $R_0=I$, we have $\frac{d}{dt}(R_t)_{|t=0}=Z$, with $Z$ skew-symmetric. Then refer to Remark \ref{phit} to complete the argument.) Unlike the case of dilations, in the case of rotations and translations, we don't know if there are simple assumptions that ensure the validity of the corresponding condition \eqref{RPL.eq2}.

\subsection{The sphere is critical for $\Lambda_1^{d_k}$}

Let us denote by $\mathbb{S}_{n-1}$ the boundary of the unit ball $\mathbb{B}_n$ in $\mathbb{R}^n$. For $n = 2$, $K^*_{\mathbb{S}_1}$ has only one eigenvalue, $\lambda = 1/2$, which has multiplicity one, and the corresponding eigenfunction is constant. For $n \ge 3$, instead, we can verify that the eigenvalues of $K^*_{\mathbb{S}_{n-1}}$ are given by the numbers
\[
\lambda_k := \frac{1}{2} \frac{n-2}{2k+n-2}, \quad \text{with } k \text{ ranging in } \mathbb{N},
\]
and the eigenspace of $\lambda_k$ is the space $H_k$ of (real)  spherical harmonics of degree $k$, so that $\lambda_k$ has multiplicity
\[
d_k := \mathrm{dim}\,H_k\,.
\]

These are all well-known facts, but for the sake of completeness--and due to the lack of a good reference--we include here a proof. Let $\{Y_{k,i}\}_{i=1}^{d_k}$ be an orthonormal basis for $H_k$, which we assume, for convenience, consists of real functions. We denote by 
\[
P_{k,i}(x) := |x|^k Y_{k,i}\left(\frac{x}{|x|}\right) \quad \text{if } x \in \mathbb{R}^n \setminus \{0\}, \quad P_{k,i}(0) := 0
\]
the homogeneous harmonic polynomial of degree $k$ associated with $Y_{k,i}$, and by 
\[
\tilde{P}_{k,i}(x) := |x|^{-(k+n-2)} Y_{k,i}\left(\frac{x}{|x|}\right) \quad \text{for all } x \in \mathbb{R}^n \setminus \{0\}
\]
its Kelvin transform, which is harmonic in $\mathbb{R}^n \setminus \{0\}$. We readily see that on $\mathbb{S}_{n-1}$, we have $P_{k,i} = \tilde{P}_{k,i}$ and 
\[
\nu_{\mathbb{B}_n} \cdot \nabla P_{k,i} - \nu_{\mathbb{B}_n} \cdot \nabla \tilde{P}_{k,i} = (2k+n-2)Y_{k,i}.
\]
Then, a standard potential theory argument shows that the single layer potential with density $(2k+n-2)Y_{k,i}$ coincides with $P_{k,i}$ on $\overline{\mathbb{B}_n}$ and with $\tilde{P}_{k,i}$ on $\mathbb{R}^n \setminus \mathbb{B}_n$ (cf., e.g., \cite{KhPuSh07}). In particular, we have
\begin{equation}\label{SY=Y}
S^+_{\mathbb{S}_{n-1}}[(2k+n-2)Y_{k,i}] = P_{k,i} \quad \text{on } \overline{\mathbb{B}_n},
\end{equation}
which implies that 
\[
\nu_{\mathbb{B}_n} \cdot \nabla S^+_{\mathbb{S}_{n-1}}[(2k+n-2)Y_{k,i}] = kY_{k,i}.
\]
By the jump properties of the single-layer potential, we deduce that 
\[
\frac{1}{2}(2k+n-2)Y_{k,i}-K^*_{\mathbb{S}_{n-1}}[(2k+n-2)Y_{k,i}] = kY_{k,i},
\]
and thus 
\[
K^*_{\mathbb{S}_{n-1}}[Y_{k,i}] = \frac{1}{2} \frac{n-2}{2k+n-2} Y_{k,i}.
\]
So, each $Y_{k,i}$ is an eigenfunction of $K^*_{\mathbb{S}_{n-1}}$ for the eigenvalue $\lambda_k$, and since $\bigcup_{k=0}^\infty \{Y_{k,i}\}_{i=1}^{d_k}$ is an $L^2(\mathbb{S}_{n-1})$-complete orthonormal system, the claim is proved.

Let us now denote by $\lambda_{k,i}$, with $i=1, \ldots, d_k$, the eigenvalues that spread from $\lambda_k$ as in Theorem $\ref{main}$. We will extend to the $n$-dimensional case a result that Ando et al.~\cite[Theorem 1.1]{AnKaMiUs19}  have proven for the $3$-dimensional case. This theorem states that the sphere is a critical shape for the sum 
\[
\sum_{i=1}^{d_k} \lambda_{k,i}.
\]
In other words, the shape derivative of this sum, computed on the sphere, is zero in all directions $\theta \in (C^{1,\alpha}(\mathbb{S}_{n-1}))^n$:
\[
\frac{d}{dt}\sum_{i=1}^{d_k} \lambda_{k,i}(I+t\theta)_{|t=0} = 0.
\]

Since the sum above coincides with the symmetric function $\Lambda_1^{d_k}(\lambda_{k,1}, \ldots, \lambda_{k,d_k})$, we can use Theorem \ref{dLambda}, where, on the right-hand side of \eqref{dLambda.eq1}, we can take $u_i = P_{k,i}$. From \eqref{SY=Y} and the orthonormality of the spherical harmonics, it follows that all such functions $u_i = P_{k,i}$ share the same $\|\cdot\|_{S_{\partial\Omega}}$ norm and are $\langle\cdot,\cdot\rangle_{S_{\partial\Omega}}$-orthogonal (the norm being $(2k+n-2)$ instead of $1$, but  we don't need to worry as long as we are proving that the derivative is zero).

 So, our goal is now to prove that 
\[
\sum_{i=1}^{d_k} \int_{\mathbb{S}_{n-1}} \left(\theta \cdot \nu_{\mathbb{B}_n}\right) \left( -\left| \nabla_{\partial \Omega} Y_{k,i} \right|^2 + \frac{\frac{1}{2}+\lambda_k}{\frac{1}{2}-\lambda_k} (\nu_{\Omega} \cdot \nabla P_{k,i})^2 \right) \, d\sigma = 0,
\]
or, equivalently,
\begin{equation}\label{ballgoal}
(2\lambda_k - 1) \sum_{i=1}^{d_k} \left| \nabla_{\partial \Omega} Y_{k,i} \right|^2 + (2\lambda_k + 1) \sum_{i=1}^{d_k} (\nu_{\Omega} \cdot \nabla P_{k,i})^2 = 0.
\end{equation}

We will use two known facts about spherical harmonics. One is the Uns\"old Theorem, which states that 
\begin{equation}\label{addf}
\sum_{i=1}^{d_k} {Y_{k,i}}^2 = \frac{d_k}{\omega_n},
\end{equation}
where $\omega_n$ is the $(n-1)$-dimensional measure of $\mathbb{S}_{n-1}$ (see, e.g., Folland \cite[Theorem 2.57]{Fo95}). The exact value $d_k/\omega_n$ is not crucial for our argument, as long as the sum in \eqref{addf} remains constant, which could be proved directly using the rotation invariance of the Laplace operator. The other fact is that 
\[
\Delta_{\mathbb{S}_{n-1}}Y_{k,i}=-k(k+n-2)Y_{k,i}.
\]

Using this second fact, we compute
\begin{align*}
\Delta_{\mathbb{S}_{n-1}}\sum_{i=1}^{d_k} {Y_{k,i}}^2 &= 2\sum_{i=1}^{d_k} Y_{k,i}\Delta_{\mathbb{S}_{n-1}}Y_{k,i} + 2\sum_{i=1}^{d_k} |\nabla_{\mathbb{S}_{n-1}}Y_{k,i}|^2 \\
&= -2k(k+n-2)\sum_{i=1}^{d_k} {Y_{k,i}}^2 + 2\sum_{i=1}^{d_k} |\nabla_{\mathbb{S}_{n-1}}Y_{k,i}|^2,
\end{align*}
while by \eqref{addf} we have
\[
\Delta_{\mathbb{S}_{n-1}}\sum_{i=1}^{d_k} {Y_{k,i}}^2 = \Delta_{\mathbb{S}_{n-1}} \left(\frac{d_k}{\omega_n}\right) = 0.
\]
Combining the last two equalities, we obtain 
\begin{equation}\label{addfnabla}
\sum_{i=1}^{d_k} |\nabla_{\mathbb{S}_{n-1}}Y_{k,i}|^2 = k(k+n-2) \sum_{i=1}^{d_k} {Y_{k,i}}^2.
\end{equation}
Since we clearly have $\nu_{\mathbb{B}_n} \cdot \nabla P_{k,i}=kY_{k,i}$, it follows that
\begin{equation}\label{addfnunabla}
\sum_{i=1}^{d_k} (\nu_{\mathbb{B}_n} \cdot \nabla P_{k,i})^2 = k^2 \sum_{i=1}^{d_k} {Y_{k,i}}^2\,.
\end{equation}
It now suffices to plug \eqref{addfnabla} and \eqref{addfnunabla} into the left-hand side of \eqref{ballgoal} and verify the equality by a straightforward computation. Alternatively, we can avoid the final computation by using \eqref{RPL.eq3}, which, in the case of the sphere, can be written as:
\[
\lambda_k \int_{\mathbb{S}_{n-1}} \left| \nabla_{\mathbb{S}_{n-1}} Y_{k,i} \right|^2 + (\nu_{\mathbb{B}_n} \cdot \nabla P_{k,i})^2\, d\sigma - \frac{1}{2} \int_{\mathbb{S}_{n-1}} \left| \nabla_{\mathbb{S}_{n-1}} Y_{k,i} \right|^2 - (\nu_{\mathbb{B}_n} \cdot \nabla P_{k,i})^2\, d\sigma = 0.
\]
That is,
\[
\int_{\mathbb{S}_{n-1}} \left( 2\lambda_k - 1 \right) \left| \nabla_{\mathbb{S}_{n-1}} Y_{k,i} \right|^2 + \left( 2\lambda_k + 1 \right) (\nu_{\mathbb{B}_n} \cdot \nabla P_{k,i})^2 \, d\sigma = 0.
\]
Then, summing over $i$, we obtain
\[
\int_{\mathbb{S}_{n-1}} \left( 2\lambda_k - 1 \right) \sum_{i=1}^{d_k} \left| \nabla_{\mathbb{S}_{n-1}} Y_{k,i} \right|^2 + \left( 2\lambda_k + 1 \right) \sum_{i=1}^{d_k} (\nu_{\mathbb{B}_n} \cdot \nabla P_{k,i})^2 \; d\sigma = 0.
\]
Knowing from \eqref{addfnabla} and \eqref{addfnunabla} that $\sum_{i=1}^{d_k} \left| \nabla_{\partial \Omega} Y_{k,i} \right|^2$ and $\sum_{i=1}^{d_k} (\nu_{\Omega} \cdot \nabla P_{k,i})^2$ are constant (see also \eqref{addf}), we conclude that \eqref{ballgoal} holds true.

\begin{remark}\label{criticality} The criticality of the sphere for the sum of the eigenvalues of order \(k\) suggests that extending the \(1/2\)-conjecture from the case of dimension \(n = 3\) to any dimension \(n \geq 3\) might be plausible. Specifically, we might put forward the proposition that for any domain \(\Omega\) obtained by perturbing a ball, there exist \(d_k\) NP-eigenvalues (counting multiplicities)  that are smaller than $1/2$ and whose sum equals
\[
\frac{d_k}{2} \frac{n-2}{2k+n-2} = \frac{1}{2}\binom{k+n-3}{k}.
\]
In particular, for \(k = 1\), there would be \(n\) NP-eigenvalues, smaller than $1/2$, whose sum is $(n-2)/2$. Since the second NP-eigenvalue of the sphere equals
\[
\frac{n-2}{2n},
\]
has multiplicity \(n\), and any other set of \(n\) numbers whose sum is \(\frac{n-2}{2}\) must contain a number greater than or equal to \(\frac{n-2}{2n}\), it would follow that the sphere minimizes the second NP-eigenvalue for any $n\ge 3$.
\end{remark}

\subsection*{Acknowledgment}

The authors are members of the `Gruppo Nazionale per l'Analisi Matematica, la Probabilit\`a e le loro Applicazioni' (GNAMPA) of the `Istituto Nazionale di Alta Matematica' (INdAM). The authors acknowledge the support  of the
project funded by the EuropeanUnion - NextGenerationEU under the National Recovery and
Resilience Plan (NRRP), Mission 4 Component 2 Investment 1.1 - Call PRIN 2022 No. 104 of
February 2, 2022 of Italian Ministry of University and Research; Project 2022SENJZ3 (subject area: PE - Physical Sciences and Engineering) `Perturbation problems and asymptotics for elliptic differential equations: variational and potential theoretic methods'.

\newpage

\begin{appendix}

\section{Certain singular boundary integrals}
We present here some results concerning singular integrals that are associated with the boundary behavior of the gradient of the double layer. These results are classical and likely familiar to readers with some knowledge of potential theory. However, locating a reliable reference for the proofs can be challenging. Hence, we opt to provide thorough justifications for our statements.

As in the main body of the paper, $\Omega$ remains a fixed bounded open subset of $\mathbb{R}^n$ of class $C^{1,\alpha}$, where $n\geq 2$ and $0<\alpha<1$. In this Appendix, however, we do not need to assume that $\Omega$ and the exterior domain $\mathbb{R}^n\setminus\overline\Omega$ are connected. In what follows $M_{\partial\Omega,ij}$ denotes the tangential operator defined by
\[
M_{\partial\Omega,ij} f(x):=\nu_{\Omega,i}(x)\partial_{x_j}f(x)-\nu_{\Omega,j}(x)\partial_{x_i}f(x)
\]
for all $i,j=1,\dots,n$, all functions $f\in C^1(\partial\Omega)$, and all $x\in\partial\Omega$. On the right-hand side, we use the same symbol $f$ for a $C^1$ extension of $f$ to $\mathbb{R}^n$. $M_{\partial\Omega,ij}$ is related to the tangential gradient $\nabla_{\partial\Omega}$ through the following equalities:
\begin{equation}\label{Mtangential}
M_{\partial\Omega,ij}=\nu_{\Omega,i}\nabla_{\partial\Omega,j}-\nu_{\Omega,j}\nabla_{\partial\Omega,i}\,,\qquad
\nabla_{\partial\Omega,j}=\sum_{i=1}^n\nu_{\Omega,i}M_{\partial\Omega,ij}\,.
\end{equation}
It is well known that for all functions $f,g\in C^1(\partial\Omega)$ we have
\begin{equation}\label{theMflip}
\int_{\partial\Omega}f(M_{\partial\Omega,ij} g)\,d\sigma=-\int_{\partial\Omega}(M_{\partial\Omega,ij} f)g\,d\sigma
\end{equation}
(see, e.g., \cite[Lemma 2.86]{DaLaMu21} for a proof). We begin with the following

\begin{lemma}\label{theMshuffle}  Let $\eta\in C^{1,\alpha}(\partial\Omega)$ and $\theta\in (C^{1,\alpha}(\partial\Omega))^n$. Then for all $x\in\partial\Omega$ we have the following equalities:
\begin{align}
\nonumber
&\int_{\partial\Omega}^*(\eta(y)-\eta(x))\;\theta(x)^\intercal\nabla^2 E_n(x-y)\nu_\Omega(y)\,d\sigma_y\\
\label{theMshuffle.eq1} 
&\qquad=-\int_{\partial\Omega}^*\sum_{i,j=1}^nM_{\partial\Omega,ij,y}\left[(\eta(y)-\eta(x))\theta_i(x)\right]\partial_{x_j} E_n(x-y)\,d\sigma_y\,,\\
\nonumber
&\int_{\partial\Omega}^*(\eta(y)-\eta(x))\theta(y)^\intercal\nabla^2 E_n(x-y)\nu_\Omega(y)\,d\sigma_y\\
\label{theMshuffle.eq2}
&\qquad=-\int_{\partial\Omega}^*\sum_{i,j=1}^nM_{\partial\Omega,ij,y}\left[(\eta(y)-\eta(x))\theta_i(y)\right]\partial_{x_j} E_n(x-y)\,d\sigma_y\,.
\end{align}  
\end{lemma}
\noindent Where the subscript $y$ of $M_{\partial\Omega,ij,y}$ signifies that the derivatives are taken with respect to the $y$ variable and multiplied by the $\nu_{\Omega}(y)$ components. 
\begin{proof} We recall that for a function $(x,y)\mapsto f(x,y)$ on $\partial\Omega\times\partial\Omega$, which is singular for $x=y$, the principal value integral is defined by 
\begin{equation}\label{intstar}
\int_{\partial\Omega}^* f(x,y)\,d\sigma_y=\lim_{r\to 0^+}\int_{\partial\Omega\setminus\mathbb{B}_n(x,r)}f(x,y)\,d\sigma_y\,.
\end{equation}
Then, we observe that 
\[
\int_{\partial\Omega\setminus\mathbb{B}_n(x,r)}f(x,y)\,d\sigma_y=\int_{\partial(\Omega\cup\mathbb{B}_n(x,r))}\tilde f(x,y)\,d\sigma_y-\int_{\partial\mathbb{B}_n(x,r)\setminus\Omega}\tilde f(x,y)\,d\sigma_y\,,
\]
where $\tilde f$ is an extension of $f$ to $\partial\Omega\times\mathbb{R}^n$. Thus, to analyze the expression on the left-hand side of \eqref{theMshuffle.eq1}, we can examine separately the integral over the boundary of $\Omega\cup\mathbb{B}_n(x,r)$ and the integral over $\partial\mathbb{B}_n(x,r)\setminus\Omega$. The advantage lies in the fact that on $\partial(\Omega\cup\mathbb{B}_n(x,r))$, the integrand function is free of singularities (indeed, it is of class $C^{1,\alpha}$), whereas on $\partial\mathbb{B}_n(x,r)\setminus\Omega$, we can leverage spherical symmetry. 

Without further ado, let us first consider the integral over $\partial\mathbb{B}_n(x,r)\setminus\Omega$. We compute:
\[
\begin{split}
&\int_{\partial\mathbb{B}_n(x,r)\setminus\Omega}(\eta(y)-\eta(x))\;\theta(x)^\intercal\nabla^2 E_n(x-y)\nu_{\mathbb{B}_n(x,r)}(y)\,d\sigma_y\\
&\quad =\frac{1}{s_n}\int_{\partial\mathbb{B}_n(x,r)\setminus\Omega} (\eta(y)-\eta(x))\left(-\frac{\theta(x)\cdot\nu_{\mathbb{B}_n(x,r)}(y)}{|x-y|^n}+n\frac{\theta(x)\cdot(x-y)(x-y)\cdot\nu_{\mathbb{B}_n(x,r)}(y)}{|x-y|^{n+2}}\right)\,d\sigma_y\,,
\end{split}
\]
where we use the same symbols, $\eta$ and $\theta$, for the $C^{1,\alpha}$ extensions of $\eta$ and $\theta$ to $\mathbb{R}^n$. We observe that for $y\in\partial\mathbb{B}_n(x,r)\setminus\Omega$ we have $|x-y|=r$ and $\nu_{\mathbb{B}_n(x,r)}(y)=(y-x)r^{-1}$. Then the expression in the right-hand side equals
\begin{equation}\label{theMshuffle.eq3}
\frac{n-1}{s_n}\int_{\partial\mathbb{B}_n(x,r)\setminus\Omega} (\eta(y)-\eta(x))\frac{\theta(x)\cdot\nu_{\mathbb{B}_n(x,r)}(y)}{r^n}\,d\sigma_y\,.
\end{equation}
Moreover, we have
\[
\eta(y)-\eta(x)=\nabla\eta(x)\cdot(y-x)+(\nabla\eta(\xi)-\nabla\eta(x))\cdot(y-x)
\]
for some $\xi$ in the segment that joins $x$ with $y$. So that
\[
\left|(\eta(y)-\eta(x))-\nabla\eta(x)\cdot(y-x)\right|\le C r^{1+\alpha}
\]
for some $C>0$, uniformly for $y\in\partial\mathbb{B}_n(x,r)$. Then the expression in \eqref{theMshuffle.eq3} equals
\begin{equation}\label{theMshuffle.eq4}
\begin{split}
&\frac{n-1}{s_n}\int_{\partial\mathbb{B}_n(x,r)\setminus\Omega} \nabla\eta(x)\cdot(y-x)\frac{\theta(x)\cdot\nu_{\mathbb{B}_n(x,r)}(y)}{r^n}\,d\sigma_y+O(r^\alpha)\\
&\qquad=\frac{n-1}{s_n}\nabla\eta(x)^\intercal\left(\frac{1}{r^{n-1}}\int_{\partial\mathbb{B}_n(x,r)\setminus\Omega}\nu_{\mathbb{B}_n(x,r)}\otimes\nu_{\mathbb{B}_n(x,r)}\,d\sigma\right)\theta(x)+O(r^\alpha)
\end{split}
\end{equation}
as $r\to 0^+$. Possibly applying a rotation, we orient ourselves in a Cartesian coordinate system where $x=0$ and $\nu_{\Omega}(x)$ aligns with the vector $e_n:= (0,\dots,0,1)$. Let $\partial\mathbb{B}^+_n(0,r)$ denote the intersection of $\partial\mathbb{B}_n(0,r)$ with the half-space $\mathbb{R}^n_+$ of vectors $x=(x_1,\dots,x_n)$ with $x_n>0$. Due to the regularity of $\Omega$, it follows that the symmetric difference between $\partial\mathbb{B}_n(0,r)\setminus\Omega$ and $\partial\mathbb{B}^+_n(0,r)$ has an area smaller than a constant times $r^{n-1+\alpha}$. Then we see that the expression in the right-hand side of \eqref{theMshuffle.eq4} equals
\[
\frac{n-1}{s_n}\nabla\eta(x)^\intercal\left(\frac{1}{r^{n-1}}\int_{\partial\mathbb{B}^+_n(0,r)}\nu_{\mathbb{B}_n(0,r)}\otimes\nu_{\mathbb{B}_n(0,r)}\,d\sigma_y\right)\theta(x)+O(r^\alpha)\\
\]
as $r\to 0^+$. We readily verify that 
\[
\frac{1}{r^{n-1}}\int_{\partial\mathbb{B}^+_n(0,r)}\nu_{\mathbb{B}_n(0,r)}\otimes\nu_{\mathbb{B}_n(0,r)}\,d\sigma=\int_{\mathbb{S}^+_{n-1}}\nu_{\mathbb{B}_n}\otimes\nu_{\mathbb{B}_n}\,d\sigma\,,
\]
where we understand that $\mathbb{B}_n:=\mathbb{B}_n(0,1)$ and  $\mathbb{S}^+_{n-1}:=\mathbb{S}_{n-1}\cap\mathbb{R}^n_+$ is the unit hemisphere. By a symmetry argument, we can prove that
\[
\int_{\mathbb{S}^+_{n-1}}\nu_{\mathbb{B}_n,i}\,\nu_{\mathbb{B}_n,j}\,d\sigma=0
\] 
whenever $i\neq j$.
 Then, using the standard parametrization $\mathbb{B}_{n-1}\ni {\xi}\mapsto ({\xi},\sqrt{1-|{\xi}|^2})\in \mathbb{S}^+_{n-1}$ of the unit hemisphere, we can compute that, for $i<n$, we have
\[
\begin{split}
&\int_{\mathbb{S}^+_{n-1}}(\nu_{\mathbb{B}_n,i})^2\,d\sigma\\
&\quad=\int_{\mathbb{B}_{n-1}}\frac{({\xi}_i)^2}{\sqrt{1-|{\xi}|^2}}\,d{\xi}
=\frac{1}{n-1}\int_{\mathbb{B}_{n-1}}\frac{|{\xi}|^2}{\sqrt{1-|{\xi}|^2}}\,d{\xi}
=\frac{s_{n-1}}{n-1}\int_0^1\frac{\rho^n}{\sqrt{1-\rho^2}}\,d\rho=\frac{s_n}{2n}
\end{split}
\] 
while, for $i=n$, we have
\[
\int_{\mathbb{S}^+_{n-1}}(\nu_{\mathbb{B}_n,n})^2\,d\sigma=\int_{\mathbb{B}_{n-1}}\sqrt{1-|{\xi}|^2}\,d{\xi}=s_{n-1}\int_0^1 \rho^{n-2}\sqrt{1-\rho^2}\,d\rho=\frac{s_n}{2n}\,.
\] 
(Both integrals can be computed using the well-known equality $\int_0^1t^{a-1}(1-t)^{b-1}dt=\Gamma(a)\Gamma(b)/\Gamma(a+b)$.)
This implies
\[
\int_{\mathbb{S}^+_{n-1}}\nu_{\mathbb{B}_n}\otimes\nu_{\mathbb{B}_n}\,d\sigma=\frac{s_n}{2n}\,1_n \,,
\]
where $1_n $ is the $n\times n$ identity matrix. Hence 
\begin{equation}\label{theMshuffle.eq5}
\lim_{r\to 0^+}\frac{1}{r^{n-1}}\int_{\partial\mathbb{B}_n(x,r)\setminus\Omega}\nu_{\mathbb{B}_n(x,r)}\otimes\nu_{\mathbb{B}_n(x,r)}\,d\sigma=\frac{s_n}{2n}\,1_n \,.
\end{equation}
By combining \eqref{theMshuffle.eq4} and \eqref{theMshuffle.eq5}, we conclude that
\begin{equation}\label{theMshuffle.eq6}
\lim_{r\to 0^+}\int_{\partial\mathbb{B}_n(x,r)\setminus\Omega}(\eta(y)-\eta(x))\;\theta(x)^\intercal\nabla^2 E_n(x-y)\nu_\Omega(y)\,d\sigma_y=\frac{n-1}{2n}\nabla\eta(x)\cdot \theta(x)\,.
\end{equation}

We now shift our focus to the integral over $\partial(\Omega\cup\mathbb{B}_n(x,r))$. For $x\neq y$, we have $\Delta_y E_n(x-y)=0$, from which we deduce that
\[
\begin{split}
&\sum_{j=1}^n M_{\partial(\Omega\cup\mathbb{B}_n(x,r)),ij,y}\partial_{y_j}E_n(x-y)\\
&\quad=\sum_{j=1}^n \nu_{(\Omega\cup\mathbb{B}_n(x,r)),i}(y)\Delta_y E_n(x-y)-\nu_{(\Omega\cup\mathbb{B}_n(0,r)),j}(y)\partial_{y_i}\partial_{y_j}E_n(x-y)\\
&\quad=-\nabla^2 E_n(x-y)\nu_{(\Omega\cup\mathbb{B}_n(x,r))}(y)\,.
\end{split}
\] 
Then
\begin{equation}\label{theMshuffle.eq7}
\begin{split}
&\int_{\partial(\Omega\cup\mathbb{B}_n(x,r))}(\eta(y)-\eta(x))\,\theta(x)^\intercal\nabla^2 E_n(x-y)\nu_{(\Omega\cup\mathbb{B}_n(x,r))}(y)\,d\sigma_y\\
&\quad= - \int_{\partial(\Omega\cup\mathbb{B}_n(x,r))}\sum_{i,j=1}^n(\eta(y)-\eta(x))\,\theta_i(x) M_{\partial(\Omega\cup\mathbb{B}_n(x,r)),ij,y}\partial_{y_j}E_n(x-y)\,d\sigma_y\\
&\quad= \int_{\partial(\Omega\cup\mathbb{B}_n(0,r))}\sum_{i,j=1}^nM_{\partial(\Omega\cup\mathbb{B}_n(x,r)),ij,y}\left[(\eta(y)-\eta(x))\theta_i(x)\right] \partial_{y_j}E_n(x-y)\,d\sigma_y\,,
\end{split}
\end{equation}
where the last equality follows from \eqref{theMflip}. We now split the integral over $\partial\mathbb{B}_n(x,r)\setminus\Omega$ from the integral over $\partial\Omega\setminus\mathbb{B}_n(x,r)$. We compute 
\[
\begin{split}
&\int_{\partial\mathbb{B}_n(x,r)\setminus\Omega}\sum_{i,j=1}^nM_{\partial\mathbb{B}_n(x,r),ij,y}\left[(\eta(y)-\eta(x))\theta_i(x)\right] \partial_{y_j}E_n(x-y)\,d\sigma_y\\
&\quad=\theta(x)^\intercal\left(\int_{\partial\mathbb{B}_n(x,r)\setminus\Omega}\nu_{\partial\mathbb{B}_n(x,r)}(y)\nabla\eta(y)\cdot\nabla_y E_n(x-y)-\nabla\eta(y)\nu_{\partial\mathbb{B}_n(x,r)}(y)\cdot\nabla_y E_n(x-y)\,d\sigma_y\right)\\
&\quad= \frac{1}{s_n}\theta(x)^\intercal\left(\frac{1}{r^{n-1}}\int_{\partial\mathbb{B}_n(x,r)\setminus\Omega}\left(-\nu_{\partial\mathbb{B}_n(x,r)}(y)\otimes\nu_{\partial\mathbb{B}_n(x,r)}(y)+1_n \right)\nabla\eta(y)\,d\sigma_y\right)\,.
\end{split}
\]
(Here, we used once more the fact that $\nu_{\partial\mathbb{B}_n(x,r)}=(y-x)|x-y|^{-1}$).
Since $|\nabla\eta(y)-\nabla\eta(x)|=O(r^\alpha)$ as $r\to 0^+$, we can verify that the last integral equals 
\[
\frac{1}{s_n}\theta(x)^\intercal\left(\frac{1}{r^{n-1}}\int_{\partial\mathbb{B}_n(x,r)\setminus\Omega}-\nu_{\partial\mathbb{B}_n(x,r)}\otimes\nu_{\partial\mathbb{B}_n(x,r)}+1_n \,d\sigma\right)\nabla\eta(x)+O(r^\alpha)
\]
as $r\to 0^+$. Then, we can refer back to \eqref{theMshuffle.eq5}, and we conclude that 
\begin{equation}\label{theMshuffle.eq8}
\begin{split}
&\lim_{r\to 0^+}\int_{\partial\mathbb{B}_n(x,r)\setminus\Omega}\sum_{i,j=1}^nM_{\partial\mathbb{B}_n(x,r),ij,y}\left[(\eta(y)-\eta(x))\theta_i(x)\right] \partial_{y_j}E_n(x-y)\,d\sigma_y\\
&\qquad =\left(-\frac{1}{2n}+\frac{1}{2}\right)\theta(x)\cdot\nabla\eta(x)=\frac{n-1}{2n}\;\theta(x)\cdot\nabla\eta(x)\,.
\end{split}
\end{equation}

By arguing as in \eqref{theMshuffle.eq7}, we have
\begin{equation}\label{theMshuffle.eq9}
\begin{split}
&\int_{\partial\Omega\setminus\mathbb{B}_n(x,r)}(\eta(y)-\eta(x))\,\theta(x)^\intercal\nabla^2 E_n(x-y)\nu_{(\Omega\cup\mathbb{B}_n(x,r))}(y)\,d\sigma_y\\
&\quad=\int_{\partial\Omega\setminus\mathbb{B}_n(x,r)}\sum_{i,j=1}^nM_{\partial\mathbb{B}_n(x,r),ij,y}\left[(\eta(y)-\eta(x))\theta_i(x)\right] \partial_{y_j}E_n(x-y)\,d\sigma_y\\
&\qquad+\int_{\partial\mathbb{B}_n(x,r)\setminus\Omega}\sum_{i,j=1}^nM_{\partial\mathbb{B}_n(x,r),ij,y}\left[(\eta(y)-\eta(x))\theta_i(x)\right] \partial_{y_j}E_n(x-y)\,d\sigma_y\\
&\qquad-\int_{\partial\mathbb{B}_n(x,r)\setminus\Omega}(\eta(y)-\eta(x))\,\theta(x)^\intercal\nabla^2 E_n(x-y)\nu_{(\Omega\cup\mathbb{B}_n(x,r))}(y)\,d\sigma_y\,.
\end{split}
\end{equation}
So, by taking the limit as $r\to 0^+$ in \eqref{theMshuffle.eq9} and recalling \eqref{theMshuffle.eq6} and \eqref{theMshuffle.eq8} we obtain \eqref{theMshuffle.eq1}.

The proof of \eqref{theMshuffle.eq2} can be derived from that of \eqref{theMshuffle.eq1}. In particular, we can prove an analogous equality to \eqref{theMshuffle.eq9} where $\theta(y)$ replaces $\theta(x)$. Namely, we have
\begin{equation}\label{theMshuffle.eq10}
\begin{split}
&\int_{\partial\Omega\setminus\mathbb{B}_n(x,r)}(\eta(y)-\eta(x))\,\theta(y)^\intercal\nabla^2 E_n(x-y)\nu_{(\Omega\cup\mathbb{B}_n(x,r))}(y)\,d\sigma_y\\
&\quad=\int_{\partial\Omega\setminus\mathbb{B}_n(x,r)}\sum_{i,j=1}^nM_{\partial\mathbb{B}_n(x,r),ij,y}\left[(\eta(y)-\eta(x))\theta_i(y)\right] \partial_{y_j}E_n(x-y)\,d\sigma_y\\
&\qquad+\int_{\partial\mathbb{B}_n(x,r)\setminus\Omega}\sum_{i,j=1}^nM_{\partial\mathbb{B}_n(x,r),ij,y}\left[(\eta(y)-\eta(x))\theta_i(y)\right] \partial_{y_j}E_n(x-y)\,d\sigma_y\\
&\qquad-\int_{\partial\mathbb{B}_n(x,r)\setminus\Omega}(\eta(y)-\eta(x))\,\theta(y)^\intercal\nabla^2 E_n(x-y)\nu_{(\Omega\cup\mathbb{B}_n(x,r))}(y)\,d\sigma_y\,.
\end{split}
\end{equation}
Since 
\[
\left|M_{\partial\mathbb{B}_n(x,r),ij,y}\left[(\eta(y)-\eta(x))(\theta_i(y)-\theta(x))\right]\right|<Cr
\]
and
\[
\left|(\eta(y)-\eta(x))(\theta(y)-\theta(x))\right|<Cr^2
\]
for some $C>0$, uniformly for $y\in\partial\mathbb{B}_n(x,r)\setminus\Omega$, we see that the limit of the last two integrals in \eqref{theMshuffle.eq10} is the same as the limit of the last two integrals in \eqref{theMshuffle.eq9}. Then, by \eqref{theMshuffle.eq6} and \eqref{theMshuffle.eq8}, we deduce \eqref{theMshuffle.eq2}. 
\end{proof}

We now leverage Lemma \ref{theMshuffle} and equality 
\begin{equation}\label{theMfiddle}
\partial_{x_i} D^{\pm}_{\partial\Omega}[\eta](x)=-\sum_{j=1}^n\partial_{x_j}S_{\partial\Omega}^{\pm}\left[M_{\partial\Omega,ij}[\eta]\right](x)\,,
\end{equation}
which holds for all $x\in\mathbb{R}^n\setminus\partial\Omega$ (cf. e.g., \cite[Lemma 4.29]{DaLaMu21}), to establish a relation between the integrals in \eqref{theMshuffle.eq1} and the boundary behavior of the gradient of the double-layer potential. We prove the following:

\begin{proposition}\label{thetangentialjump} Let $\eta\in C^{1,\alpha}(\partial\Omega)$ and $\theta\in (C^{1,\alpha}(\partial\Omega))^n$. Then for all $x\in\partial\Omega$ we have the following equalities:
\begin{align}
\nonumber
&\int_{\partial\Omega}^*(\eta(y)-\eta(x))\;\theta(x)^\intercal\nabla^2 E_n(x-y)\nu_\Omega(y)\,d\sigma_y\\
&\qquad=\theta(x)\cdot\nabla D^{\pm}_{\partial\Omega}[\eta](x)\mp\frac{1}{2}\theta(x)\cdot\nabla_{\partial\Omega}\eta(x)\label{thetangentialjump.eq1}\\
&\qquad=(\theta(x)\cdot\nu_\Omega(x))\;T_{\partial\Omega}[\eta](x)+\theta(x)\cdot\nabla_{\partial\Omega}K_{\partial\Omega}[\eta](x)\,.\label{thetangentialjump.eq2}
\end{align}
\end{proposition}
We understand that 
\[
\nabla D^{\pm}_{\partial\Omega}[\eta](x)=\lim_{\Omega^{\pm}\ni\xi\to x\in\partial\Omega}\nabla D^{\pm}_{\partial\Omega}[\eta](\xi)\qquad\text{for all $x\in\partial\Omega$,}
\] 
with $\Omega^+:=\Omega$ and $\Omega^-:=\mathbb{R}^n\setminus\overline{\Omega}$, bearing in mind that $D_{\partial\Omega}^{\pm}[\eta]\in C^{1,\alpha}(\overline{\Omega^{\pm}})$ for $\eta\in C^{1,\alpha}(\partial\Omega)$. Thus, formula \eqref{thetangentialjump.eq1} can be interpreted as a jump formula for the gradient of the double-layer potential. Specifically,
\begin{equation}\label{nablaDjump}
\nabla D^{\pm}_{\partial\Omega}[\eta](x)=\pm\frac{1}{2}\nabla_{\partial\Omega}\eta(x) +\int_{\partial\Omega}^*(\eta(y)-\eta(x))\;\nabla^2 E_n(x-y)\nu_\Omega(y)\,d\sigma_y 
\end{equation}
for all $x\in\partial\Omega$. 
\begin{proof}[Proof of Proposition \ref{thetangentialjump}]
Utilizing equality \eqref{theMfiddle} and employing the jump properties of the derivatives of the single layer potential \eqref{nablasinglejump}, we compute
\begin{equation}\label{thetangentialjump.eq3}
\begin{split}
&\lim_{\Omega^{\pm}\ni\xi\to x\in\partial\Omega}\theta(x)\cdot\nabla D_{\partial\Omega}^{\pm}[\eta](\xi)=\lim_{\Omega^{\pm}\ni\xi\to x\in\partial\Omega}-\sum_{i,j=1}^n\theta_i(x)\partial_{x_j}S_{\partial\Omega}^{\pm}\left[M_{\partial\Omega,ij}\eta\right](\xi)\\
&\qquad =\mp \sum_{i,j=1}^n\theta_i(x)\frac{\nu_j(x)}{2}(M_{\partial\Omega,ij}\eta)(x)-\sum_{i=1}^n\theta_i(x)\int_{\partial\Omega}^*\sum_{j=1}^n(M_{\partial\Omega,ij}\eta)(y)\,\partial_{x_j}E_n(x-y)\,d\sigma_y\\
&\qquad =\pm\frac{1}{2}\theta(x)\cdot\nabla_{\partial\Omega}\eta(x)-\int_{\partial\Omega}^*\sum_{i,j=1}^n\theta_i(x) (M_{\partial\Omega,ij}\eta)(y)\,\partial_{x_j}E_n(x-y)\,d\sigma_y\,,
\end{split}
\end{equation}
where, in the last equality, we also used \eqref{Mtangential}. Since
\[
\theta_i(x)(M_{\partial\Omega,ij}\eta)(y)=M_{\partial\Omega,ij,y}[(\eta(y)-\eta(x))\theta_i(x)]\quad\text{for all $x,y\in\partial\Omega$\,,}
\]
the validity of \eqref{thetangentialjump.eq1} follows by \eqref{theMshuffle.eq1} and \eqref{thetangentialjump.eq3}.

Next, we observe that 
\begin{equation}\label{thetangentialjump.eq4}
\theta(x)\cdot\nabla D_{\partial\Omega}^{\pm}[\eta](x)=(\theta(x)\cdot\nu_\Omega(x))\nu_\Omega(x)\cdot\nabla D_{\partial\Omega}^{\pm}[\eta](x)+\theta_{\partial\Omega}(x)\cdot\nabla D_{\partial\Omega}^{\pm}[\eta](x)\,,
\end{equation}
where $\theta_{\partial\Omega}:=\theta-\nu_\Omega\otimes\nu_\Omega\; \theta$ is the tangential component of $\theta$. Since $\nabla_{\partial\Omega} D_{\partial\Omega}^{\pm}[\eta]$ only depends on the trace of $D_{\partial\Omega}^{\pm}[\eta]$ on $\partial\Omega$, we have  
\[
\begin{split}
\theta_{\partial\Omega}(x)\cdot\nabla D_{\partial\Omega}^{\pm}[\eta](x)&=\theta(x)\cdot\nabla_{\partial\Omega} D_{\partial\Omega}^{\pm}[\eta](x)\\
&=\theta(x)\cdot\nabla_{\partial\Omega} \left(\pm\frac{1}{2}\eta(x)+K_{\partial\Omega}[\eta](x)\right)\\
&=\pm\frac{1}{2}\theta(x)\cdot\nabla_{\partial\Omega}\eta(x)+\theta(x)\cdot\nabla_{\partial\Omega}K_{\partial\Omega}[\eta](x)\,,
\end{split}
\]
which, combined with \eqref{thetangentialjump.eq1} and \eqref{thetangentialjump.eq4}, gives \eqref{thetangentialjump.eq2}.
\end{proof}

In the paper, we also use the following lemma, which can be verified by expanding the term $\sum_{i,j=1}^nM_{\partial\Omega,ij,y}\left[(\eta(y)-\eta(x))\theta_i(y)\right]$ appearing in the integral on the right-hand side of \eqref{theMshuffle.eq2} with the help of \eqref{Mtangential}.

\begin{lemma}\label{thesistersact} For $\eta\in C^{1,\alpha}(\partial\Omega)$ and $\theta\in (C^{1,\alpha}(\partial\Omega))^n$, we have
\begin{equation}\label{thesistersact.eq1}
\begin{split}
&\int_{\partial\Omega}^*(\eta(y)-\eta(x))\theta(y)^\intercal\nabla^2 E_n(x-y)\nu_\Omega(y)\,d\sigma_y\\
&\qquad=-\int_{\partial\Omega}^*\theta(y)\cdot\nu_\Omega(y)\;(\nabla_{\partial\Omega}\eta(y))\cdot \nabla E_n(x-y)\,d\sigma_y\\
&\qquad\quad+\int_{\partial\Omega}\theta(y)\cdot(\nabla_{\partial\Omega}\eta(y))\;\nu_\Omega(y)\cdot\nabla E_n(x-y)\,d\sigma_y\\
&\qquad\quad-\int^*_{\partial\Omega}(\eta(y)-\eta(x))\left[(\nabla_{\partial\Omega}\theta(y))^\intercal \nabla E_n(x-y)\right]\cdot\nu_\Omega(y)\,d\sigma_y\\
&\qquad\quad+\int_{\partial\Omega}(\eta(y)-\eta(x))(\mathrm{div}_{\partial\Omega}\theta)(y)\,\nu_\Omega(y)\cdot\nabla E_n(x-y)\,d\sigma_y
\end{split}
\end{equation}
for all $x\in\partial\Omega$.
\end{lemma} 
\end{appendix}

\end{document}